\documentclass{amsart}
\usepackage{hyperref}
\hypersetup{
	colorlinks=true,
	citecolor=blue,
	linkcolor=blue,
	filecolor=magenta,      
	urlcolor=cyan,
}
\usepackage{url}
\usepackage{fullpage}
\usepackage{amsmath, amsthm, amssymb, graphicx,mathtools}
\usepackage{amsrefs}

\usepackage{pgfplots}
\pgfplotsset{compat=1.15}
\usepackage{mathrsfs}
\usetikzlibrary{arrows}

\usepackage[T1]{fontenc}
\usepackage{textcomp}
\usepackage{lmodern}
\usepackage[english]{babel}

\usepackage{amscd}
\usepackage{mathrsfs}
\usepackage{tikz-cd}
\usepackage{color}
\usepackage{bbm}
\usepackage{verbatim}
\usepackage{extpfeil}
\usepackage{microtype}
\usepackage[mathscr]{euscript}

\newtheorem{theorem}{Theorem}[section]
\newtheorem*{theorem*}{Theorem}
\newtheorem{thm}{Theorem}

\newtheorem{lemma}[theorem]{Lemma}
\newtheorem{proposition}[theorem]{Proposition}
\newtheorem{corollary}[theorem]{Corollary}

\theoremstyle{definition}
\newtheorem{hypothesis}[theorem]{Hypothesis}
\newtheorem{remark}[theorem]{Remark}
\newtheorem*{remark*}{Remark}
\newtheorem{definition}[theorem]{Definition}
\newtheorem{defn}{Definition}

\newtheorem{example}[theorem]{Example}
\newtheorem*{example*}{Example}

\newtheorem{construction}[theorem]{Construction}

\numberwithin{equation}{subsection}

\newcommand{\Aut}       {\operatorname{Aut}}

\newcommand{\End}       {\operatorname{End}}
\newcommand{\Epi}       {\operatorname{Epi}}
\newcommand{\Gr}        {\operatorname{Gr}}
\newcommand{\Hom}       {\operatorname{Hom}}
\newcommand{\Inn}       {\operatorname{Inn}}

\newcommand{\Map}       {\operatorname{Map}}

\newcommand{\Out}	{\operatorname{Out}}
\newcommand{\Sp}        {\operatorname{Sp}}
\newcommand{\VHom}	{\operatorname{VHom}}
\newcommand{\Vect}      {\operatorname{Vect}}
\newcommand{\Wide}      {\operatorname{Wide}}

\newcommand{\uHom}      {\underline{\Hom}}

\newcommand{\cok}       {\operatorname{cok}}
\newcommand{\env}       {\operatorname{env}}
\newcommand{\ev}        {\operatorname{ev}}
\newcommand{\grph}      {\operatorname{graph}}
\newcommand{\img}       {\operatorname{image}}
\newcommand{\id}        {\operatorname{id}}

\newcommand{\obj}       {\operatorname{obj}}
\newcommand{\op}        {{\operatorname{op}}}

\newcommand{\rk}        {\operatorname{rk}}

\newcommand{\stab}	{\operatorname{stab}}
\newcommand{\supp}      {\operatorname{supp}}
\newcommand{\base}      {\operatorname{base}}
\newcommand{\tors}	{\operatorname{tors}}

\newcommand{\A}         {\mathcal{A}}
\newcommand{\C}		{\mathcal{C}}
\newcommand{\D}		{\mathcal{D}}
\newcommand{\E}		{\mathcal{E}}
\newcommand{\F}		{\mathcal{F}}
\newcommand{\G}		{\mathcal{G}}

\newcommand{\K}         {\mathcal{K}}
\newcommand{\I}	 	{\mathcal{I}}
\renewcommand{\L}	{\mathcal{L}}
\newcommand{\N}	 	{\mathcal{N}}
\newcommand{\M}         {\mathcal{M}}
\renewcommand{\P}	{\mathcal{P}}
\renewcommand{\S}	{\mathcal{S}}
\newcommand{\T}	 	{\mathcal{T}}
\newcommand{\U}	 	{\mathcal{U}}
\newcommand{\V}         {\mathcal{V}}
\newcommand{\W}         {\mathcal{W}}
\newcommand{\X}	 	{\mathcal{X}}
\newcommand{\Z}	 	{\mathcal{Z}}

\newcommand{\NN}        {\mathbb{N}}
\newcommand{\ZZ}        {\mathbb{Z}}

\newcommand{\tB}	{\widetilde{B}}

\newcommand{\uG}        {\underline{G}}
\newcommand{\uQ}        {\underline{Q}}
\newcommand{\uW}        {\underline{W}}

\newcommand{\ip}[1]     {\langle #1\rangle}
\newcommand{\one}       {\mathbbm{1}}

\newcommand{\sm}{\setminus}

\newcommand{\colim}  {\operatornamewithlimits{\underset{\longrightarrow}{lim}}}

\begin{document}

\title{Representation stability and outer automorphism groups}
\author{Luca Pol}
\address[Pol]{
 Fakult\"{a}t f\"{u}r Mathematik,
 Universit\"{a}t Regensburg,
 Universit\"{a}tsstr. 31,
 Regensburg 93040,
 Deutschland
}
\email{luca.pol@ur.de}

\author{Neil P. Strickland}
\address[Strickland]{
 School of Mathematics and Statistics, 
 Hicks Building, 
 Sheffield S3 7RH, 
 UK
}
\email{N.P.Strickland@sheffield.ac.uk}

\begin{abstract}
 In this paper we study families of representations of the outer
 automorphism groups indexed on a collection of finite groups $\U$. We
 encode this large amount of data into a convenient abelian category
 which generalizes the category of VI-modules appearing in the
 representation theory of the finite general linear groups. Inspired
 by work of Church--Ellenberg--Farb, we investigate for which choices
 of $\U$ the abelian category is locally noetherian and deduce
 analogues of central stability and representation stability results
 in this setting. Finally, we show that some invariants coming from
 rational global homotopy theory exhibit representation stability.
\end{abstract}

\maketitle
\setcounter{tocdepth}{1}
\tableofcontents


\section{Introduction}
\label{chap-intro}

In this paper we develop a framework for studying families of
representations of the outer automorphism groups.  A common theme in
representation theory is that there is a conceptual advantage in
encoding this large amount of (possibly complicated) data into a
single object, which lives in a convenient abelian category.  Using
purely algebraic techniques we will deduce strong constraints on
naturally occurring families of representations of the outer
automorphism groups.  We will then provide a range of examples for our
theory coming from rational global homotopy theory.
  
\subsection*{The main character}
\label{subsec-main-character}
Fix $k$ a field of characteristic zero and let $\G$ denote the
category of finite groups and conjugacy classes of surjective group
homomorphisms. We are interested in the category
$\A=[\G^{\mathrm{op}}, \Vect_k]$ of contravariant functors from $\G$ to
the category of $k$-vector spaces.  More generally, we will restrict
our attention to a replete full subcategory $\U \leq \G$ and then
consider the smaller category $\A\U=[\U^{\mathrm{op}}, \Vect_k]$.
  
Note that the endomorphism group of an object $G \in \U$ is the outer
automorphism group $\U(G,G)=\Out(G)$. Therefore any object
$X \in \A\U$ gives rise to a collection of $\Out(G)$-representations
$X(G)$ for $G \in \U$. The functoriality of $X$ imposes further
compatibility conditions on these representations. There are two main
examples where all this data can be made very explicit.
   
\begin{example*}\label{eg-cyclic-two}
 Consider the category $\C[2^\infty]$ of cyclic $2$-groups.  An object
 $X \in \A\C[2^\infty]$ gives rise to a consistent sequence of 
 representations of cyclic $2$-groups:
 \begin{center}
  \begin{tikzcd}
   X(1)      \arrow[r] \arrow[out=60,in=120,loop,"1"'] & 
   X(C_2)    \arrow[r] \arrow[out=60,in=120,loop,"1"'] & 
   X(C_4)    \arrow[r] \arrow[out=60,in=120,loop,"C_2"'] & 
   X(C_8)    \arrow[r] \arrow[out=60,in=120,loop,"C_4"'] & 
   X(C_{16}) \arrow[r] \arrow[out=60,in=120,loop,"C_8"'] & 
   X(C_{32}) \arrow[r] \arrow[out=60,in=120,loop,"C_{16}"'] & 
   \dotsb
  \end{tikzcd}
 \end{center}
 where the horizontal maps are induced by the canonical projections. 
\end{example*}

\begin{example*}\label{eg-elementary-p}
 Fix a prime number $p$ and consider category $\E[p]$ of elementary 
 abelian $p$-groups. An object $X \in \A\E[p]$ gives rise to a consistent 
 sequence of representations of the finite general linear groups:
 \begin{center}
  \begin{tikzcd}
   X(1)      \arrow[r] \arrow[out=60,in=120,loop,"1"'] & 
   X(C_p)    \arrow[r] \arrow[out=60,in=120,loop,"GL_1(\mathbb{F}_p)"'] & 
   X(C_p^2)  \arrow[r] \arrow[out=60,in=120,loop,"GL_2(\mathbb{F}_p)"'] & 
   X(C_p^3)  \arrow[r] \arrow[out=60,in=120,loop,"GL_3(\mathbb{F}_p)"'] & 
   X(C_p^4)  \arrow[r] \arrow[out=60,in=120,loop,"GL_4(\mathbb{F}_p)"'] & 
   X(C_p^5)  \arrow[r] \arrow[out=60,in=120,loop,"GL_5(\mathbb{F}_p)"'] & 
   \dotsb
  \end{tikzcd}
 \end{center}
 where the horizontal maps are induced by the projection into the first 
 coordinates. 
\end{example*}

As we have already seen in the previous examples, it will often be
convenient to restrict attention to special subcategories $\U$ (always
full and replete) for which certain phenomena stand out more
clearly. For example:
\begin{itemize}
 \item We might fix a prime $p$ and restrict attention to $p$-groups.
 \item We might restrict attention to solvable, nilpotent or abelian groups.
 \item We might impose upper or lower bounds on the exponent, nilpotence class,
  order, or on the size of a minimal generating set.
 \item  As special cases of the above, we might consider only cyclic groups, or only 
  elementary abelian $p$-groups for some fixed prime $p$. 
\end{itemize}

To ensure good homological properties, we will impose additional
conditions on $\U$ such as:
\begin{itemize}
 \item Closure under products: If $G,H \in \U$, then
  $G\times H \in \U$.  If this holds, we say that $\U$ is
  \emph{multiplicative}.
 \item Closure under passage to subgroups: If $G \in \U$ and $H \leq G$, then $H \in \U$. 
 \item Downwards closure (i.e. closure under passage to quotients): 
  If $G \in \U$ and $\G(G,H) \not =\emptyset$, then $H \in \U$.
 \item Upwards closure: If $H \in \U$ and $\G(G,H)\not =\emptyset$, then $G \in \U$.
\end{itemize}
We will see throughout this introduction that $\A\U$ has its best
homological behaviour when $\U$ is \emph{submultiplicative}
(multiplicative and closed under passage to subgroups), or a
\emph{global family} (closed downwards and closed under passage to
subgroups). We refer the reader to Section~\ref{sec-subcategories} for
a detailed list of all the closure properties considered in this paper
together with some examples.
  
Before presenting our results we put the abelian category 
$\A\U$ in the relevant context.

\subsection*{Representations of combinatorial categories}
\label{subsec-rep-comb}

The abelian category $\A\U$ is part of a larger family of categories
appearing in representation theory and algebraic topology. Given a
category $\I$ whose objects are finite sets (with possibly extra
structure) and whose morphisms are functions (possibly respecting the
extra structure), we can consider the associated diagram category
$\A_{\I}=[\I,\Vect_k]$. Some examples of interest include:
  
\begin{itemize}
 \item Let FI be the category of finite sets and injections. The
  associated diagram category is the category of FI-modules which
  appears in~\cite{Schwede2008} in the context of stable homotopy
  groups of symmetric spectra, and in~\cites{CEF15, CEFN} in relation
  to the representation theory of the symmetric groups.
 \item Let VI be the category of finite dimensional
  $\mathbb{F}_p$-vector spaces and injective linear maps. The
  associated diagram category is the category of VI-modules which
  appears in~\cites{Nagpal, Gan2015} in relation to the representation
  theory of the finite general linear groups.  This category is
  equivalent by Pontryagin duality to the category $\A\E[p]$ mentioned
  earlier.
 \item Let VA be the category of finite dimensional
  $\mathbb{F}_p$-vector spaces and all linear maps. The associated
  diagram category have been studied in relation to algebraic
  $K$-theory, rational cohomology, and the Steenrod
  algebra~\cite{Kuhn00}.  
\end{itemize}

Despite the similarities with other abelian categories appearing in
representation theory, there is a major difference between $\A\U$ and
all these categories.  We are no longer considering a one-parameter
family of representations but rather collections of representations
which are indexed by a family of groups. This brings into play
group-theoretic properties of the family $\U$ and so introduces a new
level of complexity into the story which has so far not been explored.
  
\subsection*{Noetherian condition}
\label{subsec-noetherian-intro}

The category $\A\U$ is a Grothendieck abelian category with generators
given by the representable functors
\[ e_G=k[\U(-,G)]\qquad G \in\U.  \] Many of the familiar notions from
the theory of modules carry over to this setting. For example, there are 
notions of finitely generated and finitely presented objects, see Definition~\ref{def-finiteness-conditions} for the details. We then say that abelian category $\A\U$ is \emph{locally noetherian} if all subobjects of $e_G$ are finitely generated for all $G \in \U$. It is then a formal consequence of the definition that subobjects of finitely generated objects are again finitely generated, and that any finitely generated object is also finitely presented.
  
Work of Church--Ellenberg--Farb in the category of FI-modules showed
that the noetherian condition plays a fundamental role when working
with sequences of representations~\cite{CEF15}.  This key technical
innovation allowed them to prove an asymptotic structure theorem for
finitely generated FI-modules which gave an elegant explanation for
the representation-theoretic patterns observed in earlier
work~\cites{Farb}.  Motivated by this, we investigate for which
choices of $\U$ the category $\A\U$ is locally noetherian. 
The next result combines Proposition~\ref{prop-not-noetherian} and
Theorem~\ref{thm-p-groups-is-locally-noetherian} in the body of the paper.
   
\begin{thm}\label{thm-intro-noetherian}
 Let $\U$ be a subcategory of $\G$ and let $p$ be a prime
 number. 
 \begin{itemize}
  \item[(a)] If $\U$ is a multiplicative global family of finite abelian 
   $p$-groups, then $\A\U$ is locally noetherian. Such $\U$ have the form $\Z[p^n]$ 
   for some $0 \leq n \leq \infty$, see Definition~\ref{def-subcat-examples}. 
  \item[(b)] If $\U$ is the global family of cyclic $p$-groups, then 
   $\A\U$ is locally noetherian.
 \end{itemize}
 If $\U$ contains the trivial group and infinitely many cyclic 
 groups of prime order, then $\A\U$ is not locally noetherian. 
 In particular, $\A$ is not locally noetherian.
\end{thm}

There are several combinatorial criteria available in the literature
to show that the category $\A\U$ is locally noetherian.  We prove part
(a) using the theory of Gr\"{o}bner bases developed by 
Sam--Snowden~\cite{Sam}, and part (b) using the criterion developed
in~\cite{Gan2015}.  Our result does not aim to give a complete
classification of locally noetherian categories, as this would be
costly and highly non-trivial, but rather aims to give a good range of
examples and counterexamples to which our theory applies. 

We then turn to study homological properties of our category of
interest.

\subsection*{Homological properties}  
\label{subsec-homological-intro}

The levelwise tensor product of $k$-vector spaces gives $\A\U$ a
symmetric monoidal structure in which the unit object $\one$ is
the constant functor with value $k$.  For all $X,Y \in \A\U$, there
exists an internal hom object that we denote by
$\underline{\Hom}(X,Y)\in \A\U$.
   
We list a few interesting homological properties that our category
enjoys.
\begin{itemize}
 \item[(i)] As is typical for diagram categories, the finitely
  generated projective objects are not strongly dualizable. In
  particular this means that the canonical map
  \[
   e_G \otimes \underline{\Hom}(e_G, \one) \to
     \underline{\Hom}(e_G,e_G)
  \]
  is in general not an isomorphism, see Remark~\ref{rem-dualizable}.
  However, the finitely generated projective objects of $\A\U$ still
  form a subcategory that is closed under tensor products and internal
  hom, see Proposition~\ref{prop-wide} and Theorem~\ref{thm-hom-fg-projective}.
 \item[(ii)] As is typical for diagram categories, any projective
  object is a retract of a direct sum of generators, see
  Lemma~\ref{lem-i-preserves-projective}.  However, under mild
  conditions on $\U$ (satisfied by $\G$) the projective objects of
  $\A\U$ coincide with the torsion-free injective objects, see
  Proposition~\ref{prop-projective-is-injective}.  In
  particular, the generators $e_G$ are injectives.
 \item[(iii)] Under mild conditions on $\U$ (which are satisfied by
  $\G$ itself), the only objects with a finite projective resolution
  are the projective ones, see
  Proposition~\ref{prop-perfect-is-projective}.
 \item[(iv)] The abelian category $\A\U$ is semisimple if and only if
  $\U$ is a groupoid, see Proposition~\ref{prop-semisimple}.
\end{itemize}

\subsection*{Representation stability}
\label{subsec-rep-stable-intro}

A common goal in the representation theory of categories is to give an
uniform description of the representations encoded into an object
$X\in \A\U$.  For example, one may prove that a finitely presented object can be recovered by finite amount of data via a ``stabilization recipe''. This phenomenon is called
\emph{central stability} and it was first introduced by
Putman~\cite{Putman15} for describing certain stability phenomena of
the general linear groups. Since then, central stability has been
shown to hold for various diagram categories such as
FI-modules~\cite{CEFN} and complemented categories~\cite{PS17}. 
In our framework this phenomenon can be formulated in the following way. 

\begin{defn}
 Let $\U$ be a subcategory of $\G$. 
 We say that an object $X\in\A\U$ satisfies \emph{central stability} if there exists a 
 natural number $ n\in \mathbb{N}$ such that for all $G \in \U$, we have
 \[ X(G)= \colim_{H \in N(G,n)} X(G/H) \]
 where $N(G,n)=\{H \triangleleft G \mid |G/H| \leq n\}$. 
\end{defn}

In Section~\ref{sec-central-stability} we give a slight generalization 
of the machinery described in~\cite{GanLi} and show that any finitely presented 
objects satisfies central stability. 
This result illustrates the fact that the representations encoded in a finitely 
presented object need to satisfy strong compatibility conditions.  
It tells us that we can recover the value $X(G)$ from a finite amount of data,
namely the poset $N(G,n)$ and the representations $X(G/H)$.  We note
that the poset $N(G,n)$ is always finite and often can be determined by
purely combinatorial means. For instance, in the abelian $p$-group
case its cardinality can be explicitly calculated using the Hall
polynomials~\cite{Butler}*{2.1.1}.

Given an epimorphism $\alpha \colon B \to A$, we also investigate the
behaviour of the structure maps $\alpha^* \colon X(A) \to X(B)$ for
sufficiently large groups $A$ and $B$. In this case however, we need
to restrict to the locally noetherian case.  Consider the following
families of finite abelian $p$-groups:
\[
  \F[p^n] = \{\text{free} \;\ZZ/p^n\text{-modules}\} \quad \mathrm{and} \quad 
  \C[p^\infty]= \{\text{cyclic}\; p\text{-groups} \}.
\] 
The following is an adaptation in our setting of the injectivity and
surjectivity conditions in the definition of representation stability
due to Church--Farb~\cite{Farb}*{1.1}.
  
\begin{defn}\label{def-intro-stability}
 Let $\U$ be either $\C[p^\infty]$ or $\F[p^n]$ for some $n\geq 1$. 
 Consider an object $X \in \A\U$. 
 \begin{itemize}
  \item We say that $X$ is \emph{eventually torsion-free} if there
   exists $r_0 \in \mathbb{N}$ such that for every morphism
   $\alpha\colon B\to A$ with $|A|\geq r_0$, the induced map
   $\alpha^* \colon X(A) \to X(B)$ is injective.
  \item We say that $X$ is \emph{generated in finite degree} if there exists
   $r_0 \in \mathbb{N}$ such that the canonical map 
   \[
     X(A) \otimes k[\U(B,A)] \to X(B), \quad (x,\alpha) \mapsto \alpha^*(x)
   \]
   is surjective, for all $|B|\geq |A| \geq r_0$.
 \end{itemize}
\end{defn}

We are finally ready to state our second result, see Theorem~\ref{thm-stability} and 
Proposition~\ref{prop-condition-left-induced} in the body of the paper.
  
\begin{thm}\label{thm-intro-stability}
 Fix a prime number $p$. Let $\Z[p^\infty]$ be the family of finite abelian
 $p$-groups and consider a finitely generated object $X \in\A\Z[p^\infty]$.
 Then the restriction of $X$ to $\A\C[p^\infty]$ and 
 $\A\F[p^n]$, for $n \geq 1$, is generated in finite degree and eventually torsion-free.
 Moreover, $X$ satisfies central stability. 
\end{thm}

Since the family of cyclic $p$-groups is 
closed downwards, one can easily verify that the restriction of $X$ to $\A\C[p^\infty]$ is again finitely generated. Therefore the first part of the previous result follows by combining our Theorem~\ref{thm-intro-noetherian} with~\cite{Gan2015}*{Section 5}.
On the other hand, it is not immediately clear that the restriction of $X$ to $\A\F[n]$ is again finitely generated so an additional argument is required in this case.

\subsection*{Global homotopy theory} 
\label{subsec-global-intro}

A good source of examples of finitely generated objects satisfying
representation stability comes from global stable homotopy theory: the
study of spectra with a uniform and compatible group action for all
groups in a specific class. These are particular kind of spectra that
give rise to cohomology theories on $G$-spaces for all groups in the
chosen class.  The fact that all these individual cohomology theories
come from a single object imposes extra compatibility conditions as
the group varies.  In this paper we will use the framework of global
homotopy theory developed by Schwede~\cite{Schwede}.  His approach has
the advantage of being very concrete as the category of global spectra
is the usual category of orthogonal spectra but with a finer notion of
equivalence, called global equivalence.  As any orthogonal spectrum is
a global spectrum, this approach comes with a good range of
examples. For instance, there are global analogues of the sphere
spectrum, cobordism spectra, $K$-theory spectra, Borel cohomology
spectra and many others.  It is a special feature of such a global
spectrum $X$ that the assignment
$G\mapsto \pi_0(\Phi^G X) \otimes \mathbb{Q}$ defines an object
$\underline{\Phi}_0(X) \in \A$, where we put $k=\mathbb{Q}$.  The
connection with global homotopy theory is even stronger as there is a
triangulated equivalence

\begin{equation}\label{equivalence global}
 \Phi^{\G} \colon\Sp^{\mathbb{Q}}_{\G} \simeq_{\triangle} \mathcal{D}(\A) 
\end{equation}
between the homotopy category of rational $\G$-global spectra and the
derived category of $\A$~\cite{Schwede}*{4.5.29}.  This equivalence is
compatible with geometric fixed points in the sense that
$\pi_*(\Phi^GX)=H_*(\Phi^{\G}(X))(G)$. 

We obtain the following application to global homotopy theory which
highlights the good behaviour of the geometric fixed points functor on
the full subcategory of compact global spectra. Recall that an object
$X$ in a triangulated category $\T$ is said to be compact if the
representable functor $\T(X,-)$ preserves arbitrary sums. The proof of
the following result can be found in
Section~\ref{sec-central-stability}.
  
\begin{thm}\label{thm-homotopy-stability}
 Let $\Z[p^\infty]$ be the family of finite abelian
 $p$-groups and let $X$ be a rational $\Z[p^\infty]$-global spectrum. 
 If $X$ is compact, then for all $k \in \ZZ$ the geometric fixed points
 homotopy groups $\underline{\Phi}_k(X) \in \A\Z[p^\infty]$ is generated in finite 
 degree, eventually torsion-free and satisfy central stability.
\end{thm}

An interesting source of examples is given by the rational $n$-th
symmetric product spectra.
\begin{example*}\label{eg-symmetric-power}
 For $n \geq 1$, we let $\Sp^n$ denote the orthogonal spectrum whose
 value at inner product space $V$ is given by
 \[ \Sp^n(V)= (S^V)^{\times n}/ \Sigma_n. \] Its rationalization is a
 compact rational $\Z[p^\infty]$-global spectrum by~\cite{Markus}*{2.10,
  5.1}. Therefore its geometric fixed points are eventually torsion-free,
 stably surjective and they satisfy central stability.
\end{example*}

\subsection*{Related work}
\label{subsec-related-intro}

Our study of the representation theory and homological algebra of
$\A\U$ is inspired by earlier work in the categories of
FI-modules~\cites{CEF15, CEFN} and VI-modules~\cites{Gan, Nagpal}.
Our Theorem~\ref{thm-intro-noetherian} recovers the result that the
category of VI-modules is locally noetherian, which was proved
independently by Sam--Snowden~\cite{Sam}*{8.3.3} and
Gan--Li~\cite{Gan2015}.  Versions of our representation stability
theorems were already known to hold for the category of
FI-modules~\cite{CEFN}, VI-modules~\cite{Gan} and complemented
categories~\cite{PS17}.  Finally our study of indecomposable injective
objects recovers part of the classification of injective VI-modules
due to Nagpal~\cite{Nagpal}.
    
  Nonetheless, to the best of our knowledge the results of this introduction 
  are new and they generalize several known results to a wider class of examples 
  of interest.

\subsection*{Acknowledgements} 
The first author thanks the SFB 1085 Higher Invariants in Regensburg for support, John Greenlees, Birgit Richter, Sarah Whitehouse and Jordan Williamson for reading preliminary versions of this paper and Liping Li and Peter Patzt for several helpful discussions.


\section{Preliminaries}
\label{sec-prelim}

We start by introducing the main object of study of this paper, the
abelian category $\A$.

\begin{definition}\label{def-G}
 \leavevmode
 \begin{itemize}
  \item Let $H$ be a group, and $h$ an element of $H$.  We write
   $c_h\colon H\to H$ for the inner automorphism $x\mapsto hxh^{-1}$.
  \item Let $G$ be another group, and let $\varphi,\psi\colon G\to H$ be
   homomorphisms.  We say that $\varphi$ and $\psi$ are
   \emph{conjugate} if $\psi=c_h\circ\varphi$ for some $h\in H$.  This
   is easily seen to be an equivalence relation that is compatible
   with composition.   We write $[\varphi]$ for the conjugacy
   class of $\varphi$.
  \item We write $\G$ for the category whose objects are finite
   groups, and whose morphisms are conjugacy classes of surjective
   homomorphisms.  We also write $\Out(G)=\G(G,G)$.
 \end{itemize}
\end{definition}

\begin{lemma}\label{lem-epimorphism}
 Let $\alpha \colon H \to G$ be a surjective group homomorphism
 between finite groups. Then $[\alpha]$ is an epimorphism in $\G$.
\end{lemma}
\begin{proof}
 Consider two surjective group homomorphisms
 $\beta,\gamma \colon G \to K$ , and suppose that
 $[\beta \alpha]=[\gamma \alpha]$. This means that
 $c_k \beta \alpha=\gamma \alpha$ for some $k\in K$. Since $\alpha$ is
 surjective we have $c_k \beta=\gamma$ which shows that
 $[\beta]=[\gamma]$.
\end{proof}

\begin{definition}\label{def-A}
 Fix a field $k$ of characteristic zero, and set
 $\A=[\G^{\op}, \Vect_{k}]$.  Given $G\in\G$, we write
 $\psi^G\colon \A\to\Vect_k$ for the evaluation functor $X\mapsto X(G)$.
\end{definition}

\begin{definition}\label{def-AU}
 Let $\U$ be a subcategory of $\G$.  Unless we explicitly say
 otherwise, such subcategories are assumed to be full and replete.
 (Replete means that any object of $\G$ isomorphic to an object of
 $\U$ is itself in $\U$.)  We then put $\A\U=[\U^{\op},\Vect_k]$.
\end{definition}

\begin{remark}\label{rem-bicomplete}
 The category $\A\U$ is abelian and admits limits and colimits for all
 small diagrams.  These (co)limits are computed pointwise, so they are
 preserved by the evaluation functors $\psi^G \colon\A\U \to \Vect_{k}$.
\end{remark}

\begin{definition}\label{def-base-support}
 Consider an object $X\in\A\U$. 
 \begin{itemize}
 \item The \emph{base} of $X$ is defined by 
 $\base(X)= \min \{|G| \mid X(G)\not = 0 \} \in \mathbb{N}.$ If $X$ is zero, 
 we set $\base(X)=\infty$.
 \item The \emph{support} of $X$ is defined by
 $\supp(X)= \{[G] \mid X(G)\not =0 \}$
 where $[G]$ denotes the isomorphism class of the group $G$. 
 We equipped the support with the partial order $[G]\gg [H]$ 
 if and only if $\U(G,H)\not =\emptyset$.
 \end{itemize}
\end{definition}

\begin{definition}\label{def-main-objects}
 Consider a subcategory $\U\leq\G$.  We define certain objects of $\A\U$ as
 follows.  Most of them depend on an object $G\in\U$, and possibly
 also a module $V$ over $k[\Out(G)]$.  
 \begin{itemize}
  \item We define $e_G$ by $e_G(T)= k[\G(T,G)]$. 
   Yoneda's Lemma tells us that $\A\U(e_G,X)=X(G)=\psi^G(X)$.
  \item We define objects $e_{G,V}$ and $t_{G,V} $ by 
   \[
    e_{G,V}(T) = V \otimes_{k[\Out(G)]}k[\G(T,G)] \qquad 
    t_{G,V}(T) = \Hom_{k[\Out(G)]}(k[\G(G,T)], V).
   \]
  \item We put 
   \[ c_{G}(T)=e_G(T)^{\Out(G)}=k[\G(T,G)/\Out(G)]. \]
   Note that the basis set $\G(T,G)/\Out(G)$ here can be identified
   with the set of normal subgroups $N\leq T$ such that
   $T/N\simeq G$.  Alternatively, we can regard $k$ as a $k$-linear
   representation of $\Out(G)$ with trivial action, and then
   $c_G=e_{G,k}$. 
  \item The groups $e_{G,V}(G)$ and $t_{G,V}(G)$ are both canonically
   identified with $V$, and one can check that there is a unique
   morphism $\alpha\colon e_{G,V}\to t_{G,V}$ with $\alpha_G=1$.  We
   write $s_{G,V}$ for the image of this.  If $T=G$ then
   $s_{G,V}(T)=V$.  If $T\simeq G$ then $s_{G,V}(T)$ is canonically
   isomorphic to $e_{G,V}(T)$ or $t_{G,V}(T)$, but a choice of
   isomorphism $T\to G$ is needed to identify $s_{G,V}(T)$ with $V$.  
   If $T\not\simeq G$ then $s_{G,V}(T)=0$.  
  \item Now let $\C$ be a subcategory of $\U$.  Suppose that $\C$ is
   \emph{convex}, which means that whenever $G\to H\to K$ are
   surjective homomorphisms with $G,K\in\C$ and $H\in\U$ we also have
   $H\in\C$.  We then define the 
   ``characteristic function'' $\chi_{\C}\in\A\U$ by
   \[ \chi_{\C}(T) = \begin{cases}
       k & \text{if}\;\; T \in \C \\
       0 & \text{if}\;\; T\not \in \C.
      \end{cases}
   \]
   (Convexity ensures that this can be made into a functor in an
   obvious way: the map $\chi_\C(T)\to\chi_C(T')$ is the identity if
   both groups are nonzero, and zero otherwise.)
 \end{itemize}
 If we need to specify the ambient category $\U$, we may write
 $e^\U_G$ rather than $e_G$, and so on.
\end{definition}

\begin{remark}\label{rem-grothendieck}
 The abelian category $\A\U$ is Grothendieck with generators given by
 $e_{G}$ for all $G\in\U$.  This means that filtered colimits are
 exact and that any $X \in \A$ admits an epimorphism $P \to X$ where
 $P$ is a direct sum of generators.
\end{remark}

\begin{lemma}\label{lem-projective-and-injective}
 For $G\in\U$, we let $\mathcal{M}_G$ denote the category of
 $k[\Out(G)]$-modules. Then the evaluation functor
 \[
  \ev_G \colon \A\U \to \mathcal{M}_G, \quad X \mapsto X(G)
 \]
 has a left and right adjoint which are respectively given by
 $e_{G,\bullet}$ and $t_{G,\bullet}$.  In particular, $e_{G,V}$ is
 projective and $t_{G,V}$ is injective.
\end{lemma}
\begin{proof}
 The unit of the adjunction $\eta_V \colon V \to e_{G,V}(G)=V$ is the 
 identity, and the counit is given by
 \[
   \epsilon_X(T) \colon e_{G, X(G)}(T) \to X(T), \quad
   x\otimes[\alpha] \mapsto\alpha^*(x)
 \] 
 for all $T \in \G$.  Similarly, the counit map $t_{G,V}(G) \to V$ is
 the identity, and the unit is given by
 \[
  \eta_X(T) \colon X(T) \to t_{G, X(G)}(T), \quad 
  x \mapsto ([\beta] \mapsto \beta^* (x))
 \]
 for all $T \in \G$. We leave to the reader to check that these maps
 are natural and that they satisfy the triangular identities. The
 second part of the claim follows immediately from the fact that the
 evaluation functor is exact as colimits are computed pointwise. 
 Here we are implicitly using that the field $k$ has characteristic $0$, so that all 
 finitely generated $k[Out(G)]$-modules are projective.
\end{proof}

\begin{remark}\label{rem-proj-inj}
 If $\C$ is a groupoid with finite hom sets, it is standard and easy
 that all objects in $[\C^{\op},\Vect_k]$ are both projective and
 injective.  (We will review these arguments in
 Section~\ref{sec-finite-groupoids}.) In some other cases where $\C$
 is finite and an associated algebra is Frobenius, we find that the
 projectives and injectives are the same, but that general objects do
 not have either property.  For a typical small category, the
 projectives and injectives are unrelated.  For many of the categories
 $\U\leq\G$ arising in this paper, we will show that the projectives
 in $\A\U$ are a strict subset of the injectives, which are a strict
 subset of the full subset of objects.  We are not aware of any
 examples where this pattern has previously been observed; it has a
 number of interesting consequences.
\end{remark}

\section{Subcategories and their properties}
\label{sec-subcategories}

Throughout this paper we will consider a wide range of subcategories
$\U\leq\G$, and we will impose different conditions on $\U$ in
different places.  It is convenient to collect together the main
examples and conditions here.

\begin{definition}\label{def-subcat-props}
 Let $\U$ be a subcategory of $\G$ (assumed implicitly to be full and
 replete, as usual).
 \begin{itemize}
  \item We say that $\U$ is \emph{subgroup-closed} if whenever
   $H\leq G\in\U$ we also have $H\in\U$.
  \item We say that $\U$ is \emph{closed downwards} if whenever
   $G\to H$ is a surjective homomorphism with $G\in\U$, we also have
   $H\in\U$. 
  \item We say that $\U$ is \emph{closed upwards} if whenever
   $H\to K$ is a surjective homomorphism with $K\in\U$, we also have
   $H\in\U$. 
  \item We say that $\U$ is \emph{convex} if whenever $G\to H\to K$
   are surjective homomorphisms with $G,K\in\U$, we also have
   $H\in\U$. 
  \item We say that $\U$ is \emph{multiplicative} if $1\in\U$, and
   $G\times H\in\U$ whenever $G,H\in\U$.  Equivalently, $\U$ should
   contain the product of any finite family of its objects, including
   the empty family.
  \item We say that $\U$ is \emph{widely closed} if whenever
   $G\xleftarrow{}H\xrightarrow{}K$ are surjective homomorphisms with
   $G,H,K\in\U$, the image of the combined morphism $H\to G\times K$
   is also in $\U$.  (We will show that almost all of our examples
   have this property.)
  \item We say that $\U$ is \emph{finite} if it has only finitely many
   isomorphism classes.
  \item We say that $\U$ is \emph{groupoid} if all morphisms in $\U$
   are isomorphisms.  
  \item We say that $\U$ is \emph{colimit-exact} if the functor
   $X\mapsto\colim_{G\in\U^{\op}}X(G)$ is an exact functor
   $\A\U\to\Vect_k$.  (We will show that almost all of our examples
   have this property.)
  \item We say that $\U$ is \emph{submultiplicative} if it is
   multiplicative and subgroup-closed.
  \item We say that $\U$ is a \emph{global family} if it is
   subgroup-closed and also closed downwards.
 \end{itemize}
\end{definition}

\begin{remark}\label{rem-subcat-implications}\leavevmode
 \begin{itemize}
  \item If $\U$ is closed upwards or downwards or is a groupoid, then
   it is convex.
  \item If $\U$ is submultiplicative then it is clearly widely closed.
  \item If $\U$ is convex, then it is also widely closed.  Indeed, if
   $G\xleftarrow{}H\xrightarrow{}K$ are surjective homomorphisms with
   $G,H,K\in\U$ and $L$ is the image of the resulting map
   $H\to G\times K$ then we have evident surjective homomorphisms
   $H\to L\to G$, showing that $L\in\U$.
  \item In particular, if $\U$ is closed upwards or downwards or is a
   groupoid, then it is widely closed. 
 \end{itemize}
\end{remark}

\begin{definition}\label{def-subcat-examples}
 We define subcategories of $\G$ as follows.  Some of them depend on a
 prime number $p$ and/or an integer $n\geq 1$.
 \begin{itemize}
  \item $\Z$ is the multiplicative global family of finite abelian
   groups. 
  \item $\C$ is the global family of finite cyclic groups.
  \item $\G[p^\infty]$ is the subcategory of finite $p$-groups.
  \item $\Z[p^\infty]=\Z\cap\G[p^\infty]$ is the multiplicative global
   family of finite abelian $p$-groups.
  \item $\C[p^\infty]=\C\cap\G[p^\infty]$ is the global family of
   finite cyclic $p$-groups.
  \item $\G[p^n]$ is the multiplicative global family of finite groups
   of exponent dividing $p^n$.
  \item $\Z[p^n]=\Z\cap\G[p^n]$ is the multiplicative global family of
   finite abelian groups of exponent dividing $p^n$, which is
   equivalent to the category of finitely generated modules over
   $\ZZ/p^n$. 
  \item $\C[p^n]=\C\cap\G[p^n]$ is the global family of finite cyclic
   groups of exponent dividing $p^n$, which is equivalent to the
   category of cyclic modules over $\ZZ/p^n$.
  \item $\F[p^n]$ is the subcategory of groups isomorphic to
   $(\ZZ/{p^n})^r$ for some $r\geq 0$, which is equivalent to the
   category of finitely generated free modules over $\ZZ/p^n$.
  \item $\E[p]$ is the multiplicative global family of elementary
   abelian $p$-groups, which is the same as $\Z[p]$ or $\F[p]$.
 \end{itemize}
 We also consider the following subcategories, primarily as a source
 of counterexamples:
 \begin{itemize}
  \item $\W_0$ is the subcategory of finite simple groups, which is a
   groupoid. 
  \item $\W_1$ is the subcategory of (necessarily cyclic) groups of
   prime order, which is also a groupoid.
  \item $\W_2$ is the subcategory of finite $2$-groups in which every
   square is a commutator.  This is easily seen to be multiplicative
   and closed downwards.  However, it contains the quaternion group
   $Q_8$ but not the cyclic group $C_4<Q_8$, so it is not
   subgroup-closed. 
  \item $\W_3$ is the subcategory of finite $p$-groups in which all
   elements of order $p$ commute.  This is clearly submultiplicative,
   but it is not closed downwards.  Indeed, one can check that $\W_3$
   contains the upper triangular group $UT_3(\ZZ/p^2)$ (provided that
   $p>2$), but not the quotient group $UT_3(\ZZ/p)$.  (We thank Yves
   de Cornulier, aka MathOverflow user YCor, for this
   example~\cite{YCor}.)
 \end{itemize}
 Given a subcategory $\U$, we also define further subcategories as
 below, depending on an integer $n>0$ or an object $N\in\U$:
 \begin{itemize}
  \item $\U_{\leq n}=\{G\in\U\mid |G|\leq n\}$.  This is always
   finite.  If $\U$ is subgroup-closed, closed downwards, convex,
   widely-closed or a groupoid then $\U_{\leq n}$ inherits the same
   property. 
  \item $\U_{\geq n}=\{G\in\U\mid |G|\geq n\}$.  If $\U$ is closed
   upwards, convex, widely closed, finite or a groupoid then
   $\U_{\geq n}$ inherits the same property.  
  \item $\U_{=n}=\{G\in\U\mid |G|=n\}=\U_{\leq n}\cap\U_{\geq n}$.
   This is always a finite groupoid, and so is convex and widely closed.
  \item $\U_{\leq N}=\{G\in\U\mid\G(N,G)\neq\emptyset\}$.  This is
   always  finite.  If $\U$ is closed downwards, convex,
   widely-closed or a groupoid then $\U_{\leq N}$ inherits the same
   property. 
  \item $\U_{\geq N}=\{G\in\U\mid \U(G,N)\neq\emptyset\}$.  If $\U$ is
   closed upwards, convex, widely closed, finite or a groupoid then
   $\U_{\geq N}$ inherits the same property.  
  \item $\U_{\simeq N}=\{G\in\U\mid G\simeq N\}=
    \U_{\leq N}\cap\U_{\geq N}$.  This is always a
   finite groupoid, and so is convex and widely closed.
 \end{itemize}
\end{definition}

\begin{example}\label{ex-not-widely-closed}
 Using Remark~\ref{rem-subcat-implications} we see that almost all of
 the specific subcategories listed above are widely closed.  One
 exception is the subcategory $\F[p^n]$ for $n>1$.  We will identify
 this with the category of finitely generated free modules over
 $\ZZ/p^n$ and so use additive notation.  We take $G=K=\ZZ/p^n$ and
 $H=(\ZZ/p^n)^2$, and we define maps
 $G\xleftarrow{\alpha}H\xrightarrow{\beta}K$ by $\alpha(i,j)=i$ and
 $\beta(i,j)=i+pj$.  We find that the image of the combined map
 $H\to G\times K$ is isomorphic to $\ZZ/p^n\times\ZZ/p^{n-1}$ and so
 does not lie in $\F[p^n]$.
\end{example}

\section{Closed monoidal structure}
\label{sec-monoidal-structure}

It is convenient to add a bit of structure on $\A$.

\begin{definition}\label{def-dual}
 We give $\A\U$ the symmetric monoidal structure given by
 $(X\otimes Y)(T)=X(T)\otimes Y(T)$.  The unit object $\one$ is
 the constant functor with value $k$ (so $\one=e_1$ provided
 that $1\in\U$).  We also put
 \[ \uHom(X,Y)(T)= \A(e_T \otimes X, Y). \]
 Standard arguments show that this defines an object of $\A\U$ with
 \[
   \A\U(W, \uHom(X,Y)) \simeq \A(W \otimes X, Y),
 \]
 so $\A\U$ is a closed symmetric monoidal category. We write $DX$ for
 $\uHom(X,\one)$, and call this the \emph{dual} of
 $X$.
\end{definition}

\begin{remark}\label{rem-flat}
 Note that the tensor product is both left and right exact, so all
 objects are flat.
\end{remark}

\begin{remark}\label{rem-dualizable}
 We warn the reader that $DX$ is not obtained from $X$ by taking
 levelwise duals, so the canonical map
 $X \otimes DX \to \uHom(X,X)$ is usually not an
 isomorphism.  To demonstrate this consider the case $X=e_G$ for any
 non-trivial group $G$.  If we evaluate at the trivial group, we find
 $e_G(1) \otimes De_G(1)=0$ and
 $\uHom(e_G, e_G)(1)=k[\Out(G)]$.  Therefore the map is far
 from being an isomorphism.
\end{remark}

For the rest of this section we study the effect of the tensor product
and internal hom functor on the generators. The main results are
Proposition~\ref{prop-wide} and Theorem~\ref{thm-hom-fg-projective}
and they both rely on the following notion.

\begin{definition}\label{def-permuted-family}
 Let $\U$ be a subcategory of $\G$.  A \emph{permuted family} of
 groups consists of a finite group $\Gamma$, a finite $\Gamma$-set
 $A$, a family of groups $G_a\in\U$ for each $a\in A$, and a system of
 isomorphisms $\gamma_*\colon G_a\to G_{\gamma(a)}$ (for
 $\gamma\in\Gamma$ and $a\in A$) satisfying the functoriality
 conditions $1_*=1$ and $(\delta\gamma)_*=\delta_*\gamma_*$.  The
 system of isomorphisms gives maps $\stab_\Gamma(a)\to\Aut(G_a)$ for
 each $a\in A$.  We say that the family is \emph{outer} if the image
 of this map contains the inner automorphism group $\Inn(G_a)$
 for all $a$.  Given a permuted family $\uG$ which is outer,
 we define the set
 \[ \tB(\uG)(T) =
      \{(a,\alpha) \mid a\in A,\;\alpha\in\Epi(T,G_a)\}. 
 \]
 The group $\Gamma$ acts on $\tB(\uG)(T)$ via the formula
 $\gamma \cdot(a, \alpha)= (\gamma(a), \gamma_*\circ \alpha)$. We define
 $B(\uG)(T)=\tB(\uG)(T)/\Gamma$ and
 $F(\uG)(T)=k[B(\uG)(T)]$.  This is
 contravariantly functorial in $T$, so $F(\uG)\in\A\U$. 
\end{definition}

\begin{proposition}\label{prop-permuted-family}
 For all $X\in\A\U$ there is a natural isomorphism
 \[ \A\U(F(\uG),X) = \left(\prod_{a\in A} X(G_a)\right)^\Gamma. \]
 If we choose a subset $A_0\subset A$ containing one element of each
 $\Gamma$-orbit, we get an isomorphism
 \[ F(\uG) = \bigoplus_{a\in A_0} e_{G_a}^{\stab_{\Gamma}(a)}. \]
 Thus, $F(\uG)$ is finitely projective (see 
 Definition~\ref{def-finiteness-conditions}).
\end{proposition}
\begin{proof}
 We can reduce to the case where $A$ is a single orbit, say
 $A=\Gamma a\simeq\Gamma/\Delta$, where $\Delta=\stab_\Gamma(a)$.  We
 can define $\phi\colon\Epi(T,G_a)/\Delta\to B(\uG)(T)$ by
 $\phi[\alpha]=[a,\alpha]$.  If $[b,\beta]\in B(\uG)(T)$
 then $b=\gamma(a)$ for some $a$.  We can then put
 $\alpha=\gamma_*^{-1}\circ\beta\colon T\to G_a$ and we find that
 $[b,\beta]=\phi[\alpha]$.  On the other hand, if
 $\phi(\alpha)=\phi(\alpha')$ then there exists $\gamma\in\Gamma$ with
 $(\gamma(a),\gamma_*\circ\alpha)=(a,\alpha')$ which means that
 $\gamma\in\Delta$ and $[\alpha]=[\alpha']$ in $\Epi(T,G_a)/\Delta$.
 It follows that $\phi$ is a natural bijection.  Thus, if we let $\Phi$
 denote the image of $\Delta$ in $\Out(G_a)$, we have
 $F(\uG)\simeq e_{G_a}^{\Phi}$.  Note that the inclusion
 $e_{G_a}^{\Phi}\leq e_{G_a}$ is split by the map
 $x \to |\Phi|^{-1}\sum_{\phi \in \Phi} \phi \cdot x $. It follows that
 $e_{G_a}^{\Phi}$ is projective.
\end{proof}

\begin{definition}\label{def-wide}
 Let $(G_i)_{i\in I}$ be a finite family of groups in $\U$ with
 product $P=\prod_iG_i$.  
 \begin{itemize}
 \item We say that a subgroup $W\leq P$ is
 \emph{wide} if all the projections $\pi_i\colon W\to G_i$ are
 surjective.  
 \item We say that a homomorphism $f\colon T\to P$ is
 \emph{wide} if all the morphisms $\pi_i\circ f$ are surjective, or
 equivalently $f(T)$ is a wide subgroup of $P$. 
 \end{itemize}
 For $G,H \in \U$, we let $\Wide(G,H)$ denote the set of wide subgroups 
 of $G \times H$ which belong to $\U$. 
 This set is covariantly functorial in $G$ and $H$ with respect to 
 morphisms in $\U$. Given $\varphi\colon G' \to G$ in $\U$ 
 and $W' \in \Wide(G',H)$, we put 
 $\varphi_* W'= (\varphi \times \mathrm{id}_H)(W')$ which is wide in 
 $G\times H$. This comes with a map 
 $j_{\varphi}\colon W' \to \varphi_*W'$ which makes the following diagram
  \[
  \begin{tikzcd}
  G' \times H \arrow[r,"\varphi \times \mathrm{id}"] & G\times H \\
  W' \arrow[u, hook] \arrow[r,"j_\varphi"] &\varphi_* W' \arrow[u, hook]
  \end{tikzcd}
  \]
 commute. The assignment $W' \mapsto \varphi_* W'$ defines a map 
 $\varphi_* \colon \Wide(G',H) \to \Wide(G,H)$ between the set of wide 
 subgroups. Similar functoriality holds for $H$ as well.   
\end{definition}

\begin{example}\label{ex-wide}
 Let $G_1$ and $G_2$ be finite groups.
 \begin{itemize}
  \item[(a)] The full group $G_1\times G_2$ is always wide.  If $G_1$
   and $G_2$ are nonisomorphic simple groups, then one can check 
   (perhaps using Lemma~\ref{lem-wide-graph} below) that
   this is the only example.  Similarly, if $|G_1|$ and $|G_2|$ are
   coprime, then $G_1\times G_2$ is the only wide subgroup.
  \item[(b)] If $\alpha\colon G_1\to G_2$ is a surjective homomorphism,
   then the graph 
   \[ \Gr(\alpha)=\{(g,\alpha(g)) \mid g \in G_1\} \]
   is always wide.  If $G_1$ and $G_2$ are isomorphic simple groups, 
   then one can check that every wide subgroup is of the form~(a)
   or~(b).  Moreover, in~(b) we see that $\alpha$ must be an
   isomorphism. 
  \item[(c)] Now let $\U\leq\G$ be a groupoid, and suppose that
   $G_1,G_2\in\U$.  If $W\leq G_1\times G_2$ is wide and lies in $\U$, we
   see easily that $W$ is the graph of an isomorphism
   $\alpha\colon G_1\to G_2$.
  \item[(d)] Now consider the case
   $\U=\C[p^\infty]=\{\text{cyclic $p$-groups}\}$.  If
   $|G_1|\geq|G_2|$ then it is not hard to see that any cyclic wide
   subgroup of $G_1\times G_2$ is the graph of a surjective homomorphism
   $\alpha\colon G_1\to G_2$ as in~(b).  Similarly, if
   $|G_1|\leq|G_2|$ then any cyclic wide subgroup of $G_1\times G_2$ is
   the graph of a surjective homomorphism $\beta\colon G_2\to G_1$.
   Of course, if $|G_1|=|G_2|$ then any surjective homomorphism
   $\alpha\colon G_1\to G_2$ is an isomorphism, and the graph of
   $\alpha\colon G_1\to G_2$ is the same as the graph of
   $\alpha^{-1}\colon G_2\to G_1$.
 \end{itemize}
\end{example}

\begin{definition}\label{def-wide-graph}
 Suppose we have finite groups $G_1$ and $G_2$, and normal subgroups
 $N_i\lhd G_i$, and an isomorphism
 $\alpha \colon G_1/N_1 \to G_2/N_2$.  We can then put
 \[
  H(N_1, \alpha, N_2)= \lbrace (x_1,x_2) \in G_1 \times G_2 \mid 
  \alpha(x_1N_1)=x_2N_2\rbrace \leq G_1 \times G_2.
 \]
 This is easily seen to be a wide subgroup.
\end{definition}

\begin{lemma}\label{lem-wide-graph}
 Every wide subgroup $K\leq G_1\times G_2$ has the form
 $H(N_1,\alpha,N_2)$ for a unique triple $(N_1,\alpha,N_2)$ as above. 
\end{lemma}
\begin{proof}
 Put
 \[  N_1= \lbrace n_1 \in G_1 \mid (n_1, 1) \in K \rbrace, \]
 and similarly for $N_2$. If $n_1 \in N_1$ and $g_1 \in G_1$ then
 wideness gives $g_2 \in G_2$ such that $(g_1,g_2) \in K$. It follows
 that the element
 $(g_1 n_1 g_1^{-1},1)=(g_1,g_2)(n_1,1)(g_1,g_2)^{-1}$ lies in $K$ and
 that $N_1$ is normal. The same argument shows that $N_2$ is normal in
 $G_2$ too. This means that $K$ is the preimage in $G_1 \times G_2$ of
 the subgroup
 $\bar{K}= K/(N_1 \times N_2) \leq (G/N_1) \times (G/N_2)$. We now
 find that the projections $\pi_i \colon \bar{K} \to G_i/N_i$ are both
 isomorphisms, so we can define
 $\alpha \colon \pi_2 \pi_1^{-1} \colon G/N_1 \to G/N_2$.  It is now
 easy to see that $K = H(N_1, \alpha, N_2)$, as required.
\end{proof}

\begin{definition}\label{def-wide-permuted-family}
 Given $G,H \in \U$, we let $\uW(G,H)$ denote the 
 tautological family indexed by $\Wide(G,H)$, so the group indexed 
 by $U \in \Wide(G,H)$ is $U$
 itself. Then $G \times H$ acts on $\Wide(G,H)$ by conjugation. We use
 this to regard $\uW(G,H)$ as a permuted family, and thus define
 a finitely projective object $F(\uW(G,H)) \in \A\U$.
\end{definition}

We now consider tensor products of generators.

\begin{proposition}\label{prop-wide}
 Let $\U$ be a widely closed subcategory of $\G$, and suppose that
 $G,H\in\U$.  Then $e_G \otimes e_H$ is naturally isomorphic to
 $F(\uW(G,H))$ (and so is a finitely generated projective
 object of $\A\U$).
\end{proposition}
\begin{proof}
 Consider another object $T\in\U$ and a pair
 $(\alpha,\beta)\in\Epi(T,G)\times\Epi(T,H)$.  This
 gives a wide subgroup $U=\ip{\alpha,\beta}(T)\leq G\times H$, which
 lies in $\U$ because $\U$ is assumed to be widely closed.   We
 can regard $\ip{\alpha,\beta}$ as a surjective homomorphism from 
 $T$ to $U$, so we have an element
 $\phi(\alpha,\beta)=(U,\ip{\alpha,\beta})\in\tB(\uW(G,H))(T)$.
 This is easily seen to give a $(G\times H)$-equivariant natural
 bijection
 \[
  \phi\colon \Epi(T,G)\times\Epi(T,H)\to\tB(\uW(G,H))(T).
 \]
 It follows easily that we get an induced bijection
 $\U(T,G)\times\U(T,H)\to B(\uW(G,H))(T)$ and an isomorphism
 $e_G\otimes e_H\to F(\uW(G,H))$ as required.
\end{proof}

\begin{remark}\label{rem-wide-abelian}
 If $G$ and $H$ are abelian, then $G \times H$ acts trivially on
 $\uW$ and so $e_G \otimes e_H=\bigoplus_{U\in\Wide(G,H)} e_U$.
\end{remark}

\begin{remark}\label{rem-tensor-gens-not-free}
 It is not true that $e_G\otimes e_H$ is always a direct sum of
 objects of the form $e_K$.  In particular, this fails when $G=H=D_8$.
 To see this, let $N$ be the subgroup of $G$ isomorphic to $C_4$, and
 put $W=\{(g,h)\in G\times H \mid gN=hN\}$. This is wide, and has
 index $2$ in $G\times H$, so it is normal in $G\times H$. 
 The group $G\times H$ acts by conjugation of the set $\Wide(G,H)$ 
 and the stabilizer of the conjugacy class of $W$ is 
 the quotient $Q=(G\times H)/W$. 
 Then the summand in the tensor product $e_G \otimes e_H$ 
 corresponding to the conjugacy class of $W$ is given by $e_W^Q$ 
 which is not of the form $e_K$.
\end{remark}

\begin{definition}\label{def-vhom}
 A \emph{virtual homomorphism} from $G$ to $H$ is a pair
 $\alpha=(A,A')$ where $A'\lhd A\leq G\times H$ and $A$ is wide and
 $A'\cap(1\times H)=1$ and $A/A'\in\U$.  We write $\VHom(G,H)$ for the
 set of virtual homomorphisms.  We then let $\uQ(G,H)$ denote the collection of groups 
 $Q_\alpha=A/A'$ indexed by all
 $\alpha=(A,A')\in\VHom(G,H)$.  We call $Q_\alpha$ the \emph{spread}
 of $\alpha$.  Note that $G\times H$ acts compatibly on $\VHom(G,H)$
 and $Q_\alpha$ by conjugation.  We use this to regard $\uQ(G,H)$
 as a permuted family, and thus to define a finitely projective object
 $F(\uQ(G,H))\in\A\U$.
\end{definition}

\begin{example}\label{ex-unspread}
 Suppose that $\U$ contains the trivial group. 
 For any surjective homomorphism $u\colon G\to H$, we can define
 \[ A = A' = \grph(u) = \{(g,u(g))\mid g\in G\}. \]
 This gives a virtual homomorphism with trivial spread.
 We claim that every virtual homomorphism with trivial spread arises in this way 
 from a unique homomorphism. Indeed, let $\alpha=(A,A)$ be any such virtual 
 homomorphism and consider the projection map 
 $A \leq G \times H \to G$. The condition  $A \cap (1 \times H)=1$ ensures 
 that every element $g \in G$ has a unique preimage $(g, u(g)) \in A$ under the 
 projection. It is easy to check that the assignment $u \colon G \to H$ defines 
 a surjective group homomorphism, and by construction $A =\grph(u)$.
\end{example}

\begin{example}\label{ex-vhom-from-one}
 Consider a virtual homomorphism $\alpha=(A,A')\in\VHom(1,G)$.  The group $A$
 must be wide in $1\times G$, which just means that $A=1\times G$.  The
 group $A'\leq 1\times G$ must satisfy $A'\cap(1\times G)=1$, which means
 that $A'=1$.  Thus, there is a unique virtual homomorphism
 $\alpha=(1\times G,1)$, whose spread is $G$.  
\end{example}

\begin{example}\label{ex-vhom-to-one}
 Consider a virtual homomorphism $\alpha=(A,A')\in\VHom(G,1)$.  
 We find that
 $A$ must be equal to $G\times 1$ (which we identify with $G$) and $A'$
 can be any normal subgroup of $G$ such that $G/A' \in \U$.
\end{example}

\begin{theorem}\label{thm-hom-fg-projective}
 Let $\U$ be a multiplicative global family of finite groups. Fix
 groups $G,H \in \U$ and let $\uQ(G,H)$ be the permuted
 family of virtual homomorphisms from $G$ to $H$. Then
 $\uHom(e_G,e_H)$ is isomorphic to $F(\uQ(G,H))$
 (and so is a finitely generated projective object of $\A\U)$.
\end{theorem}

The general structure of the proof is as follows.  We will fix $G$ and
$H$, and define finite sets $\L(T)$, $\M(T)$ and $\N(T)$ depending
on a third object $T\in\U$.  All of these will have actions of
$G\times H$ by conjugation, and we will construct equivariant
bijections between them.  We will also construct isomorphisms
$\uHom(e_G,e_H)(T)\simeq k[\N(T)]^{G\times H}$ and 
$F(\uQ(G,H))(T)\simeq k[\M(T)]^{G\times H}$.  All of this is
natural with respect to isomorphisms $T'\to T$, but unfortunately not
with respect to arbitrary morphisms $T'\to T$ in $\U$.  However, we
will introduce filtrations of all the relevant objects and show that 
the failure of naturality involves terms that shift filtration.  It
will follow that the associated graded object for
$\uHom(e_G,e_H)$ is isomorphic to $F(\uQ(G,H))$.
As this object is projective, we see that the filtration splits, so 
$\uHom(e_G,e_H)$ itself is isomorphic to
$F(\uQ(G,H))$, as claimed.

\begin{definition}\label{def-LMN}
 Fix groups $G,H\in\U$.  Let $T$ be another group in $\U$. 
 \begin{itemize}
  \item[(a)] We define $\L(T)$ to be the set of wide subgroups
   $V\leq T\times G\times H$ such that $V\cap H=1$.  (Here we identify
   $H$ with the subgroup $1\times 1\times H\leq T\times G\times H$,
   and we will make similar identifications in various places below.)
  \item[(b)] We define $\M(T)$ to be the set of triples
   $(A,A',\theta)$ where $(A,A')\in\VHom(G,H)$ and
   $\theta\in\Epi(T,A/A')$.
  \item[(c)] We define $\N(T)$ to be the set of pairs $(W,\lambda)$
   where $W$ is a wide subgroup of $T\times G$, and
   $\lambda\in\Epi(W,H)$. 
 \end{itemize}
 All of these sets have evident actions of $G\times H$ by conjugation.
\end{definition}

\begin{definition}\label{def-LMN-functor}
 Given a surjective homomorphism $\varphi\colon T'\to T$, we define
 maps $\varphi^*\colon \L(T)\to \L(T')$, and similarly for $\M$ and
 $\N$, as follows:
 \begin{itemize}
  \item[(a)] $\varphi^*(V)=(\varphi\times 1\times 1)^{-1}(V)=
   \{(t',g,h)\in T'\times G\times H\mid (\varphi(t'),g,h)\in V\}$
  \item[(b)] $\varphi^*(A,A',\theta)=(A,A',\theta\varphi)$
  \item[(c)] $\varphi^*(W,\lambda)=((\varphi\times
   1)^{-1}(W),\lambda\circ(\varphi\times 1))$.
 \end{itemize}
 These constructions are clearly functorial.
\end{definition}

\begin{construction}\label{cons-L(T)-to-M(T)}
 We define a bijection $\mu\colon\L(T)\to\M(T)$ as follows.  Given
 $V\in\L(T)$ we put $A=\pi_{G\times H}(V)\leq G\times H$ and
 $A'=\{(g,h)\in G\times H\mid (1,g,h)\in V\}$.  As $V$ is wide in
 $T\times G\times H$, we see that $A$ is wide in $G\times H$.  As
 $V\cap H=1$, we see that $A'\cap H=1$.  This means that the pair
 $(A,A')$ is an element of $\VHom(G,H)$.  Next, for $t\in T$ we put
 \[ \theta(t)=\{(g,h)\in G\times H\mid (t,g,h)\in V\}. \]
 This is a coset of $A'$ in $A$, or in other word an element of
 $A/A'$.  It is not hard to check that this gives a homomorphism
 $\theta\colon T\to A/A'$.  From the definition of $A$ we see that
 $\theta$ is surjective.  We have thus defined an element
 $\mu(V)=(A,A',\theta)\in\M(T)$.  

 In the opposite direction, suppose we start with an element
 $(A,A',\theta)\in\M(T)$.  We can then define 
 \[ V = \{(t,g,h)\in T\times A \mid \theta(t) = (g,h).A'\}. \]
 It is clear that $\pi_T(V)=T$ and $\pi_{G\times H}(V)=A$.  As $A$ is
 wide in $G\times H$, it follows that $V$ is wide in $T\times G\times H$.
 Now suppose that $(1,1,h)\in V$, so $(1,h)\in A$ and the coset
 $(1,h).A'$ is the same as $\theta(1)$, or in other words
 $(1,h)\in A'$.  It then follows from the definition of a virtual
 homomorphism that $h=1$.  This proves that $V\in\L(T)$.  It is easy
 to check that this construction gives a map $\M(T)\to\L(T)$ that is
 inverse to $\mu$.  It is also straightforward to check that these
 bijections are natural with respect to the functoriality in
 Definition~\ref{def-LMN-functor}.
\end{construction}

\begin{construction}\label{cons-L(T)-to-N(T)}
 We define a bijection $\nu\colon\L(T)\to\N(T)$ as follows.  Given
 $V\in\L(T)$ we define $W=\pi_{T\times G}(V)\leq T\times G$.  As
 $V\in\L(T)$ we have $V\cap H=1$, which means that the projection
 $\pi_{T\times G}\colon V\to W$ is an isomorphism.  We define
 $\lambda$ to be the composite 
 \[ W \xrightarrow{\pi_{T\times G}^{-1}} V \xrightarrow{\pi_H} H. \]
 As $V$ is wide in $T\times G\times H$, we see that $\lambda$ is
 surjective, so we have an element $\nu(V)=(W,\lambda)\in\N(T)$.

 In the opposite direction, suppose we start with an element
 $(W,\lambda)\in\N(T)$.  We then put 
 \[ V = \{(t,g,h)\in W\times H\mid \lambda(t,g)=h\}. \]
 As $W$ is wide in $T\times G$ and $\lambda\colon W\to T$ is
 surjective, we see that $V$ is wide in $T\times G\times H$.  If
 $(1,1,h)\in V$ then we must have $h=\lambda(1,1)=1$.  This proves
 that $V\in\L(T)$.  It is easy to check that this construction gives a
 map $\N(T)\to\L(T)$ that is inverse to $\nu$.  It is also
 easy to check that these bijections are natural with
 respect to the functoriality in Definition~\ref{def-LMN-functor}.
\end{construction}

\begin{remark}
 It is straightforward to identify $\tB(\uQ(G,H))(T)$ with
 $\M(T)$, and so to identify $F(\uQ(G,H))(T)$ with
 $k[\M(T)]^{G\times H}$.
\end{remark}

\begin{definition}\label{def-LMN-sigma}
 For each element $x$ in $\L(T)$, $\M(T)$ or $\N(T)$ we define a
 positive integer $\sigma(x)$ as follows.
 \begin{itemize}
  \item[(a)] For $V\in\L(T)$ we put 
   \[ V^\# = \{(t,g,h)\in V\mid (t,1,1),\;(1,g,h)\in V\}, \]
   and $\sigma(V)=|V|/|V^\#|$.
  \item[(b)] For $(A,A',\theta)\in\M(T)$ we put
   $\sigma(A,A',\theta)=|A/A'|$. 
  \item[(c)] For $(W,\lambda)\in\N(T)$ we put 
   \[ K(W,\lambda) =
       \{t\in T\mid (t,1)\in W \text{ and } \lambda(t,1)=1\}
   \]
   and $\sigma(W,\lambda)=|T|/|K(W,\lambda)|$.
 \end{itemize}
 We then put 
 \begin{align*}
  F^n\L(T) &= \{x\in \L(T)\mid\sigma(x)\geq n\} \subseteq \L(T) \\
  F^nk[\L(T)] &= k[F^n\L(T)] \leq k[\L(T)] \\
  Q^nk[\L(T)] &= F^nk[\L(T)]/F^{n+1}k[\L(T)].
 \end{align*}
\end{definition}

\begin{remark}\label{rem-sigma-max}
 For $(A,A',\theta)\in\M(T)$ it is clear that
 $\sigma(A,A',\theta)\leq|G||H|$.  It follows that
 $\sigma(x)\leq|G||H|$ for $x\in\L(T)$ or $x\in\N(T)$ as well.
\end{remark}

\begin{lemma}\label{lem-LMN-sigma}
 Suppose that the elements $V\in\L(T)$ and $(A,A',\theta)\in\M(T)$ and
 $(W,\lambda)\in\N(T)$ are related as in
 Constructions~\ref{cons-L(T)-to-M(T)} and~\ref{cons-L(T)-to-N(T)}.
 Then $\sigma(V)=\sigma(A,A',\theta)=\sigma(W,\lambda)$.  Thus, those
 constructions give bijections $F_n\L(T)\simeq F_n\M(T)\simeq F_n\N(T)$.
\end{lemma}
\begin{proof}
 As in Construction~\ref{cons-L(T)-to-M(T)}, we have a surjective
 projection $\pi\colon V\to A$, and it follows that
 $|A/A'|=|V|/|\pi^{-1}(A')|$.  Moreover, we have
 $A'=\{(g,h)\mid(1,g,h)\in V\}$, and it follows easily that
 $\pi^{-1}(A')=V^\#$; this makes it clear that
 $\sigma(V)=\sigma(A,A',\theta)$.  On the other hand, we also have a
 surjective projection $\pi'\colon V\to T$, and it follows that 
 $|T|/|K(W,\lambda)|=|V|/|(\pi')^{-1}(K(W,\lambda))|$.  Suppose we
 have $(t,g,h)\in V$ with $t\in K(W,\lambda)$.  It then follows that
 $(t,1,1)\in V$, and thus that the product
 $(t,g,h).(t,1,1)^{-1}=(1,g,h)$ also lies in $V$, so 
 $(t,g,h)\in V^\#$.  This argument is reversible so we find that 
 $(\pi')^{-1}(K(W,\lambda))=V^\#$ and $\sigma(V)=\sigma(W,\lambda)$.
\end{proof}

We now want to define an isomorphism 
\[ \zeta\colon k[\N(T)]^{G\times H} \to 
     \uHom(e_G,e_H)(T) = \A\U(e_T\otimes e_G,e_H).
\]
One approach would be to split $e_T\otimes e_G$ as a sum over
conjugacy classes of wide subgroups, but that involves choices which
are awkward to control.  We will therefore define $\zeta$ in a
different way, and then use the splitting of $e_T\otimes e_G$ to
verify that it is an isomorphism.

\begin{construction}\label{cons-N(T)-to-uHom}
 Fix an element $(W,\lambda)\in\N(T)$.  Now consider an object
 $P\in\U$ and a pair of surjective homomorphisms $\alpha\colon P\to T$
 and $\beta\colon P\to G$, giving an element
 $[\alpha]\otimes[\beta]\in(e_T\otimes e_G)(P)$ and a wide subgroup
 $\langle\alpha,\beta\rangle(P)\leq T\times G$.  If there exists an
 element $(t,g)\in T\times G$ such that
 $c_{(t,g)}(\langle\alpha,\beta\rangle(P))=W$, then we can form
 the composite
 \[ P \xrightarrow{\langle c_t\alpha,c_g\beta\rangle} W 
      \xrightarrow{\lambda} H.
 \]
 This is a surjective homomorphism.  Its conjugacy class depends only
 on the conjugacy classes of $\alpha$ and $\beta$, and not on the
 choice of $(t,g)$.  Moreover, everything that we have done is natural
 for morphisms $P'\to P$ in $\U$.  We can thus define an element
 $\zeta_0(W,\lambda)\in\A\U(e_T\otimes e_G,e_H)$ by
 $\zeta_0(W,\lambda)([\alpha]\otimes[\beta])=
  \lambda\circ\langle c_t\alpha,c_g\beta\rangle$ in the case
 discussed above, and $\zeta_0(W,\lambda)([\alpha]\otimes[\beta])=0$
 in the case where $\langle\alpha,\beta\rangle(P)$ is not conjugate to
 $W$.  It is easy to see that if $(W_0,\lambda_0)$ and
 $(W_1,\lambda_1)$ lie in the same $(G\times H)$-orbit of $\N(T)$,
 then $\zeta_0(W_0,\lambda_0)=\zeta_0(W_1,\lambda_1)$.  We now extend
 linearly to get a map $k[\N(T)]\to\uHom(e_G,e_H)(T)$, and restrict to
 get a map $\zeta\colon k[\N(T)]^{G\times H}\to\uHom(e_G,e_H)(T)$.

 We can now choose a list of wide subgroups
 $W_1,\dotsc,W_r\leq T\times G$ containing precisely one
 representative of each conjugacy class, and let $\Delta_i$ be the
 normaliser of $W_i$ in $T\times G$.  We have seen that this gives a
 decomposition $e_T\otimes e_G=\bigoplus_ie_{W_i}^{\Delta_i}$, and
 thus an isomorphism 
 \[ (e_T\otimes e_G)(H) = \bigoplus_k k[\Epi(W_i,H)/\Delta_i]. \]
 (Note here that $\Delta_i\geq W_i$ so the conjugation action of
 $\Delta_i$ on $\Epi(W_i,H)$ encompasses the action of inner
 automorphisms.)  From this it is not hard to see that $\zeta$ is an
 isomorphism. 
\end{construction}

\begin{definition}
 We put $F^n\uHom(e_G,e_H)(T)=\zeta(F^nk[\N(T)])$, and 
 \[ Q^n\uHom(e_G,e_H)(T) =
    \frac{F^n\uHom(e_G,e_H)(T)}{F^{n+1}\uHom(e_G,e_H)(T)}.
 \]
\end{definition}

Now consider a surjective homomorphism $\varphi\colon T'\to T$.  This
gives a map $\varphi^*\colon\M(T)\to\M(T')$ given by
$\varphi^*(A,A',\theta)=(A,A',\theta\varphi)$, and this is
straightforwardly compatible with our identification
$F(\uQ(G,H))(T)\simeq k[\M(T)]^{G\times H}$.  However, the situation
with $\N(T)$ and $\uHom(e_G,e_H)(T)$ is more complicated.

\begin{definition}
 Consider an element $(W,\lambda)\in\N(T)$, and a surjective
 homomorphism $\varphi\colon T'\to T$.  Let $E(\varphi,(W,\lambda))$
 be the set of pairs $(W',\lambda')\in\N(T')$ such that
 $(\varphi\times 1)(W')=W$ and $\lambda'$ is the same as the composite 
 \[ W' \xrightarrow{\varphi\times 1} W \xrightarrow{} H. \]
 It is easy to see that the element $\varphi^*(W,\lambda)=
 ((\varphi\times 1)^{-1}(W),\lambda\circ(\varphi\times 1))$
 is an element of $E(\varphi,(W,\lambda))$.
\end{definition}

\begin{lemma}\label{lem-sigma-bound}
 Suppose that $(W',\lambda')\in E(\varphi,(W,\lambda))$.  Then
 $\varphi$ restricts to give a surjective homomorphism
 $K(W',\lambda')\to K(W,\lambda)$.  It follows that
 $\sigma(W',\lambda')\geq\sigma(W,\lambda)$, with equality iff
 $(W',\lambda')=\varphi^*(W,\lambda)$. 
\end{lemma}
\begin{proof}
 Suppose that $t'\in K(W',\lambda')$, so that $(t',1)\in W'$ and
 $\lambda'(t',1)=1$.  As $(\varphi\times 1)(W')=W$, we see that
 $(\varphi(t'),1)\in W$.  As $\lambda\circ(\varphi\times 1)=\lambda'$,
 we see that $\lambda(t,1)=1$.  This shows that $t\in K(W,\lambda)$.

 Conversely, suppose that $t\in K(W,\lambda)$.  This means that
 $(t,1)\in W=(\varphi\times 1)(W')$, so there exists $(t',g)\in W'$
 with $(\varphi(t'),g)=(t,1)$.  In other words, there exists
 $t'\in T'$ such that $(t',1)\in W'$ and $\varphi(t')=t$.  Using the
 relation $\lambda\circ(\varphi\times 1)=\lambda'$ again, we also see
 that $\lambda'(t',1)=1$, so $t'\in K(W',\lambda')$.

 We now see that 
 \[ |K(W',\lambda')| =
      |K(W,\lambda)||\ker(\varphi)\cap K(W',\lambda')|
    \leq |K(W,\lambda)|.\frac{|T'|}{|T|}.
 \]
 Rearranging this gives 
 \[ \sigma(W',\lambda') = \frac{|T'|}{|K(W',\lambda')|} 
     \geq \frac{|T|}{|K(W,\lambda)|} = \sigma(W,\lambda).
 \]
 We have equality iff $\ker(\varphi)\leq K(W',\lambda')$.  Because
 $\lambda'$ factors through $\varphi\times 1$, we see that the second
 condition in the definition of $K(W',\lambda')$ is automatic, so we
 have equality iff $\ker(\varphi)\times 1\leq W'$.  This clearly holds
 if $W'=(\varphi\times 1)^{-1}(W)$.  

 Conversely, suppose that $\ker(\varphi)\times 1\leq W'$.  We are
 given that $(\varphi\times 1)(W')=W$, so
 $W'\leq(\varphi\times 1)^{-1}(W)$.  In the other direction, suppose
 that $(t',g)\in(\varphi\times 1)^{-1}(W)$, so
 $(\varphi(t'),g)\in W$.  As $W=(\varphi\times 1)(W')$, we can choose
 $(t'_0,g_0)\in W'$ with $(\varphi\times 1)(t'_0,g_0)=(t',g)$.  In
 other words, we can find $t'_0\in T'$ such that
 $\varphi(t'_0)=\varphi(t')$ and $(t'_0 ,g)\in W'$.  We now have
 $t'=t'_0t'_1$ for some $t'_1\in\ker(\varphi)$, so $(t'_1,1)\in W'$ by
 assumption.  It follows that the product $(t',g)=(t'_0,g)(t'_1,1)$
 also lies in $W'$ as required.
\end{proof}

\begin{proposition}\label{prop-uHom-greded}
 The subspaces $F^n\uHom(e_G,e_H)(T)$ form a subobject of
 $\uHom(e_G,e_H)$ in $\A\U$, so the quotient $Q^n\uHom(e_G,e_H)$ can
 also be regarded as an object of $\A\U$.  Moreover, the sum
 $Q^*\uHom(e_G,e_H)=\bigoplus_nQ^n\uHom(e_G,e_H)$ is naturally
 isomorphic to $F(\uQ(G,H))$.
\end{proposition}
\begin{proof}
 Consider an element $m\in F^n\uHom(e_G,e_H)(T)$ and a surjective
 homomorphism $\varphi\colon T'\to T$.  We can regard $m$ as a
 morphism $e_T\otimes e_G\to e_H$.  Now suppose we have a surjective
 homomorphism $\varphi\colon T'\to T$.  Now $\varphi^*m$
 corresponds to the composite
 $m\circ(e_\varphi\otimes 1)\colon e_{T'}\otimes e_G\to e_H$.
 Consider a wide subgroup $W'\leq T'\times G$, and the resulting map
 $j'\colon e_{W'}\to e_{T'}\otimes e_G$.  Put
 $W=(\varphi\times 1)(W')$, which is wide in $T\times G$, and let $j$
 be the resulting map $e_W\to e_T\otimes e_G$.  The composite
 $mj\colon e_W\to e_H$ can be expressed as a $k$-linear combination of
 morphisms $\lambda\in\Epi(W,H)$.  The condition
 $m\in F^n\uHom(e_G,e_H)(T)$ means that for all $\lambda$ appearing
 here, we have $\sigma(W,\lambda)\geq n$.  It follows that
 $\varphi^*(m)j'$ can be expressed as a $k$-linear combination of the
 corresponding morphisms
 $\lambda'=\lambda\circ(\varphi\times 1)\colon W'\to H$. 
 Lemma~\ref{lem-sigma-bound} tells us that the resulting pairs satisfy
 $\sigma(W',\lambda')\geq n$.  It follows that $F^n\uHom(e_G,e_H)$ is
 a subobject, as claimed.  Moreover, the edge case in
 Lemma~\ref{lem-sigma-bound} tells us that in the associated graded,
 we see only terms of the form $\varphi^*(W,\lambda)$.  This means
 that the associated graded is isomorphic in $\A\U$ to $k[\N]$ or
 $k[\L]$ or $k[\M]$ or $F(\uQ(G,H))$, as claimed. 
\end{proof}

\begin{proof}[Proof of Theorem~\ref{thm-hom-fg-projective}]
 The subobjects $F^n\uHom(e_G,e_H)$ form a finite-length filtration of
 $\uHom(e_G,e_H)$ with finitely projective quotients, so the
 filtration must split.  The claim follows easily from this.
\end{proof}

\begin{remark}\label{rem-natural-splitting}
 We do not know whether there is a splitting of the filtration that is
 natural in $G$ and $H$ as well as $T$.  There may be some interesting
 group theory and combinatorics involved here.
\end{remark}



\section{Functors for subcategories}
\label{sec-subcat-functors}

In this section we study the formalism that relates the abelian
category $\A$ to its smaller subcategories $\A\U$.
	
\begin{definition}\label{def-subcat-functors}
 Let $\U$ and $\V$ be full and replete subcategories of $\G$, with
 $\U\subseteq\V$.  The inclusion $i=i_{\U\V}\colon\U\to\V$ gives a pullback
 functor $i^*\colon\A\V\to\A\U$.  We write $i_!$ and $i_*$ for the
 left and right adjoints of $i^*$ (so $i_!,i_*\colon\A\U\to\A\V$).
 These are given by the usual Kan formulae (in their contravariant
 versions): 
 \begin{itemize}
  \item[(a)] $(i_!X)(G)$ is the colimit over the comma category
   $(G\downarrow\U)$ of the functor sending each object
   $(G\xrightarrow{u}iH)$ to $X(H)$. 
  \item[(b)] $(i_*X)(G)$ is the limit over the comma category
   $(\U\downarrow G)$ of the functor sending each object
   $(iK\xrightarrow{v}G)$ to $X(K)$. 
 \end{itemize}
\end{definition}

\begin{remark}\label{rem-U-times}
 The above definition covers most of the inclusion functors that we
 need to consider, with one class of exceptions, as follows.  Let $\U$
 be a replete full subcategory of $\G$.  We then let $\U^\times$ be
 the category with the same objects, but with only group isomorphisms
 as the morphisms, and we let $l\colon\U^\times\to\U$ be the
 inclusion.  Then $\U^\times$ is not a full subcategory of $\U$, so
 definition~\ref{def-subcat-functors} does not officially apply.
 Nonetheless, we still have functors $l^*$, $l_!$ and $l_*$, whose
 behaviour is slightly different from what we see in the main case.  
 Details will be given later.
\end{remark}

\begin{lemma}\label{lem-omnibus}
  Let $i\colon\U\to\V$ be an inclusion of replete full subcategories
  of $\G$.
  \begin{itemize}
   \item[(a)] The (co)unit maps $i^*i_*(X)\to X\to i^*i_! (X)$ are
    isomorphisms, for all $X \in \A\U$.  Thus, the functors $i_!$ and
    $i_*$ are full and faithful embeddings.
   \item[(b)] The essential image of $i_!$ is 
    $\{ Y \in\A\V \mid \epsilon_Y \colon i_! i^* (Y) \to Y \; \text{is iso}\}$.  \\
    The essential image of $i_*$ is 
    $\{ Y \in\A\V \mid \eta_Y \colon Y \to i_*i^* (Y) \; \text{is iso} \}$.
   \item[(c)] There are natural isomorphisms
    $i^*(\one)=\one$ and $i^*(X\otimes Y)=i^*(X)\otimes i^*(Y)$
    giving a strong monoidal structure on $i^*$.  However, the corresponding map 
    $i^* \uHom(X,Y) \to \uHom(i^*X,i^*Y)$ is
    typically not an isomorphism.
   \item[(d)] There are natural maps $i_!\one\to\one\to i_*\one$ and 
    $i_!(X \otimes Y) \to i_!(X) \otimes i_!(Y)$ and 
    $i_*(X) \otimes i_*(Y) \to i_*(X \otimes Y)$ giving (op)lax monoidal structures. 
   \item[(e)] In all cases $i_!$ preserves all colimits and $i_*$
    preserves all limits and $i^*$ preserves both limits and
    colimits. Also $i_!$ preserves projective objects and $i_*$
    preserves injective objects. Both $i_*$ and $i_!$ preserve
    indecomposable objects.
   \item[(f)] If $\U$ is closed upwards in $\V$, then $i_!$ is
    extension by zero and so preserves all limits, colimits and
    tensors (but not the unit).
   \item[(g)] If $\U$ is closed downwards in $\V$ then $i_*$ is
    extension by zero and so preserves all limits, colimits and
    tensors (but not the unit). 
   \item[(h)] If $\U$ is submultiplicative, then $i_!$ preserves the
    unit and all tensors; in other words, is strongly monoidal. 
   \item[(i)] If $i$ has a left adjoint $q\colon \V \to \U$
    then $i_!=q^*$ (and so $i_!$ preserves all (co)limits). 
   \item[(j)] Suppose that $G \in \U$ and $\C \leq \U$ is
    convex. Then, for the objects defined in
    Definition~\ref{def-main-objects} we have 
    \begin{align*} 
     i^*(e_{G,V}) &= e_{G,V} &
     i^*(t_{G,V}) &= t_{G,V} &
     i^*(s_{G,V}) &= s_{G,V} \\ 
     i^*(\chi_\C) &= \chi_{\C} &
     i_!(e_{G,V}) &= e_{G,V} &
     i_*(t_{G,V}) &= t_{G,V}.
   \end{align*}
   If $\U$ is closed upwards, we also have 
   \[
    i_!(\chi_\C)=\chi_\C \quad
    i_!(s_{G,V})=s_{G,V} \quad 
    i_!(t_{G,V})=\chi_\U \otimes t_{G,V}.
   \]
   On the other hand, if $\U$ is closed downwards, we also have
   \[
    i_*(e_{G,S})=\chi_\U \otimes e_{G,S} \quad
    i_!(s_{G,V})=s_{G,V} \quad 
    i_*(\chi_{\C})=\chi_\C.
   \]
 \end{itemize}
\end{lemma}
\begin{proof}
 Almost all of this is standard, but we recall proofs for ease of
 reference. 

 If $G\in\U$ then $(G\xrightarrow{1}G)$ is terminal in the comma category
 $\U\downarrow G$, so the Kan formula reduces to $(i_!X)(G)=X(G)$.
 Using this, we see that the unit map $X\to i^*i_!(X)$ is an
 isomorphism for all $X$.  It follows that the map
 $i_!\colon\A\U(X,Y)\to\A\V(i_!X,i_!Y)$ is an isomorphism, with
 inverse essentially given by $i^*$, so $i_!$ is a full and faithful
 embedding.  A dual argument shows that the counit map 
 $i^*i_*(X)\to X$ is also an isomorphism, and that $i_*$ is also a
 full and faithful embedding.  This proves claim~(a).

 Now put 
 \begin{align*}
  \mathcal{B} &=
   \{ Y \in\A\V \mid \epsilon_Y \colon i_!i^*(Y) \to Y  \; \text{is iso}\} \\
  \mathcal{C} &=
   \{ Y \in\A\V \mid \eta_Y     \colon Y \to i_* i^*(Y) \; \text{is iso}\}.
 \end{align*}
 If $Y\in\mathcal{B}$ then $Y\simeq i_!i^*(Y)$ so $Y$ is in the
 essential image of $i_!$.  Conversely, for $X\in\A\U$ we have seen
 that the unit $X\to i^*i_!(X)$ is an isomorphism, so the same is true
 of the map $i_!(X)\to i_!i^*i_!(X)$.  By the triangular identities
 for the $(i_!,i^*)$-adjunction, it follows that the counit
 $i_!i^*i_!(X)\to i_!(X)$ is also an isomorphism, so
 $i_!(X)\in\mathcal{B}$, so any object isomorphic to $i_!(X)$ also
 lies in $\mathcal{B}$.  This proves that $\mathcal{B}$ is the
 essential image of $i_!$, and a dual argument shows that
 $\mathcal{C}$ is the essential image of $i_*$.  This proves
 claim~(b). 

 We now consider claim~(c).  For all $G\in\U$, we have
 $(i^*\one)(G)=k=\one(G)$ and
 $i^*(X \otimes Y)(G)= X(G)\otimes Y(G)=(i^*X \otimes i^*Y)(G)$ which
 proves that $i^*$ is strongly monoidal as claimed.  For the negative
 part of (c), consider the case where $\U=\{1\}$.  For any $G$ we have 
 $(i^*\uHom(e_G,e_G))(1)=\A\V(e_G,e_G)=k[\Out(G)]\neq 0$, but if 
 $G$ is nontrivial, then $i^*(e_G)=0$ and so
 $\uHom(i^*(e_G),i^*(eG))(1)=0$.  This shows that the natural map
 $i^*\uHom(X,Y)\to\uHom(i^*(X),i^*(Y))$ (adjoint to the
 evaluation map) is not always an isomorphism.
  
 From claim~(c) we get a natural isomorphism
 \[
   i^*(i_*(X) \otimes i_*(Y)) \simeq i^*i_*(X) \otimes i^*i_*(Y)
   \xrightarrow{\epsilon_X \otimes \epsilon_Y} X \otimes Y,
 \]
 and using the $(i^*,i_*)$-adjunction we get a natural map
 $i_*(X)\otimes i_*(Y)\to i_*(X\otimes Y)$.  A standard argument shows
 that this makes $i_*$ into a lax monoidal functor.  By a dual
 construction, we get a natural map
 $i_!(X\otimes Y)\to i_!(X)\otimes i_!(Y)$ making $i_!$ into an oplax
 monoidal functor.  This proves~(d).
	
 Most of claim~(e) is formal and follows from the properties of
 adjunctions.  If $P$ is indecomposable, we see that the only
 idempotent elements in $\End(P)$ are $0$ and $1$, and that $0\neq 1$.
 As $i_!$ is full and faithful, we see that $\End(i_!P)$ is isomorphic
 to $\End(P)$, and so has the same idempotent structure. A similar
 proof works for $i_*$ too.
  
 Now consider claim~(f).  Suppose that $\U$ is closed upwards in $\V$,
 that $X\in\A\U$,and that $G\in\V$.  If $G\in\U$ then $i_!(X)(G)=X(G)$
 by claim~(a).  If $G\not\in\U$ then the upward closure assuption
 implies that $G\downarrow\U=\emptyset$, so the Kan formula reduces
 to $i_!(X)(G)=0$.  In other words, $i_!$ is extension by zero, and
 the rest of claim~(f) follows immediately.  A dual argument
 proves~(g). 

 Now suppose instead that $\U$ is submultiplicative, as in~(h), so in
 particular $1\in\U$.  We claim that the natural map
 $i_!(X\otimes Y)\to i_!(X)\otimes i_!(Y)$ is an isomorphism.  As $i_!$ and
 the tensor product preserve colimits, we can reduce to the case where
 $X=e_G$ and $Y=e_H$ for some $G,H\in\U$.  Recall
 that $i_!(e_G^{\U})=e_G^{\V}$ (or more briefly $i_!(e_G)=e_G$), and
 similarly for $H$.  Using Proposition~\ref{prop-wide} we see that
 $i_!(e_G)\otimes i_!(e_H)$ can be expressed in terms of the wide
 subgroups $W\leq G\times H$ such that $W\in\V$, whereas
 $i_!(e_G\otimes e_H)$ is similar but involves only groups $W$ that
 lie in $\U$.  However, the submultiplicativity condition ensures that
 any wide subgroup $W\leq G\times H$ lies in $\U$, so we see that 
 $i_!(e_G)\otimes i_!(e_H)=i_!(e_G\otimes e_H)$.  We also have
 $i_!(\one)=i_!(e_1)=e_1=\one$.   This shows that $i_!$ is strongly
 monoidal. 

 Now suppose that $i$ has a left adjoint $q$ as in (i). Then the comma
 category $T \downarrow \U$ is equivalent to $qT \downarrow \U$ which
 has a terminal object $(qT \xrightarrow{1}qT)$ giving $Y(T)=X(qT)$.
 It follows that $q^*$ and $i_!$ are naturally isomorphic as claimed.
  
 In claim~(j), all the statements about $i^*$ are straightforward.
 For any $X \in \A\V$ we have
 \[
  \A\V(i_!(e_{G,V}), X)=\A\U(e_{G,V}, i^*(X))=
  \mathcal{M}_G(V,X(G))= \A\V(e_{G,V},X)
 \]
 where we used Lemma~\ref{lem-projective-and-injective}.  It follows
 by the Yoneda Lemma that $i_!(e_{G,V})=e_{G,V}$, and a similar
 proof gives that $i_*(t_{G,V})=t_{G,V}$.  The remaining claims in (j)
 follows from (f) and (g) as the functor $i_!$ and $i_*$ are extension
 by zero.
\end{proof}

\begin{remark}\label{rem-shriek-not-monoidal}
 Part~(f) of the lemma gives conditions under which $i_!$ preserves
 tensor products.  However, this does not always hold if we remove
 those conditions, as shown by the following counterexample.  Take
 $\U=\C[2^\infty]$ (the family of cyclic $2$-groups).  Note that the
 only wide subgroups of $C_4 \times C_2$ are the whole group
 $C_4 \times C_2$ and the graph subgroup $\Gr(\pi) \simeq C_4$ of the
 canonical projection $\pi: C_4 \to C_2$.  Using
 Proposition~\ref{prop-wide}, we see that
 $e_{C_2} \otimes e_{C_4} \simeq
  e_{C_4 \times C_2} \oplus e_{\Gr(\pi)}$ in $\A$ but
 $e^{\U}_{C_2} \otimes e^{\U}_{C_4} \simeq e^{\U}_{\Gr(\pi)}$ in $\A\U$.
 Thus, the canonical map
 \[
  e_{\Gr(\pi)}
   = i_!(e^{\U}_{\Gr(\pi)}) 
   \simeq i_!(e^{\U}_{C_2} \otimes e^{\U}_{C_4})
   \to i_!(e^{\U}_{C_2})\otimes i_!(e^{\U}_{C_4}) 
   = e_{C_2} \otimes e_{C_4}
   \simeq e_{C_4 \times C_2} \oplus e_{\Gr(\pi)}
 \]
 is not an isomorphism in $\A$.
\end{remark} 

\begin{lemma}\label{lem-extra-adjoints}
 Let $\V$ be a replete full subcategory of $\G$.  Let $\U$ and $\W$ be
 two replete full subcategories of $\V$ that are complements of each
 other, with inclusions $i\colon\U\to\V$ and $j\colon\W\to\V$.
 Suppose that $\U$ is closed upwards, or equivalently, that $\W$ is
 closed downwards.  Then:
 \begin{itemize}
  \item[(a)] The functor $i_!\colon\A\U\to\A\V$ admits a left adjoint
   $i^!\colon\A\V\to\A\U$ given by $i^!(Y)=i^*(\cok(j_!j^*(Y)\to Y))$.
  \item[(b)] The functor $j_*\colon\A\W\to\A\V$ admits a right adjoint 
   $j^\sharp\colon\A\V\to\A\W$ given by
   $j^{\sharp}(X)=i^*(\ker(X \to j_*j^* X))$.
 \end{itemize}
\end{lemma}
\begin{proof}
 We will only prove (a) as the argument for (b) is similar.  Consider
 a morphism $u\colon Y\to i_!(X)$.  This fits in a naturality square 
 as follows:
 \begin{center}
  \begin{tikzcd}
   j_!j^*(Y) \arrow[r,"\epsilon_Y"] \arrow[d,"j_!j^*(u)"'] &
   Y \arrow[d,"u"] \\
   j_!j^*i_!(X) \arrow[r,"\epsilon_{i_!(X)}"'] & i_!(X). 
  \end{tikzcd}
 \end{center}
 Lemma~\ref{lem-omnibus}(f) tells us that $i_!$ is extension by zero,
 so $j^*i_!=0$, so the bottom left corner of the square is zero, so
 there is a unique morphism
 $\overline{u}\colon\cok(\epsilon_Y)\to i_!(X)$ induced by $u$.  We
 can now compose $i^*(\overline{u})$ with the inverse of the unit map
 $X\to i^*i_!(X)$ to get a morphism $u^\#\colon i^!(Y)\to X$.  We
 leave it to the reader to check that this construction gives the
 required bijection $\A\V(Y,i_!(X))\simeq\A\U(i^!(Y),X)$.
\end{proof}

We can use the formalism of change of subcategory to construct functorial 
projective and injective resolutions. 

\begin{construction}\label{con-proj-inj-resolutions}
 As in Remark~\ref{rem-U-times}, we let $\U^{\times}$ denote the
 subcategory with the same objects as $\U$ but only isomorphisms as
 morphisms.  Let $l\colon \U^\times\to\U$ be the inclusion, and
 consider the functors $l_!,l_*\colon\A\U^\times\to\A\U$.  If we
 choose a skeleton $\U'\subset\U$, it is not hard to check that 
 \begin{align*}
  l_!(W) &= \bigoplus_{G\in\U'} e_{G,W(G)} \\
  l_*(W) &= \prod_{G\in\U'} t_{G,W(G)}
 \end{align*}
 for all $W\in\A\U^\times$.  It follows that $l_!(W)$ is always
 projective and $l_*(W)$ is always injective.  Moreover, we see that
 the counit $l_!l^*(X)\to X$ is always an epimorphism for all
 $X\in\A\U$, and the unit $X\to l_*l^*(X)$ is always a monomorphism.

 We now set 
 \begin{align*}
  P_0     &= l_!l^*(X) & 
  P_1     &= l_!l^*(\ker(P_0 \to X)) & 
  P_{i+2} &= l_!l^*(\ker(P_{i+1} \to P_i)) \;\; \forall i\geq 0 \\
  I_0     &= l_*l^*(X) & 
  I_1     &= l_*l^*(\cok(X \to I_0)) &
  I_{i+2} &= l_*l^*(\cok(I_{i} \to I_{i+1})) \;\; \forall i\geq 0
 \end{align*}
 Then $P_{\bullet} \to X$ and $X \to I_\bullet$ define functorial
 projective and injective resolutions of $X$, respectively. 
\end{construction}

 Recall the base of an object 

\begin{remark}\label{rem-base-proj-resolution}
 Recall from Definition~\ref{def-base-support} that 
 \[ \base(X) = \min\{|G|\mid X(G)\neq 0\}. \]
 If $|G|<\base(X)$ then we find that $(l_!l^*(X))(G)=0$, and if
 $|G|=\base(X)$ we find that the counit map $(l_!l^*(X))(G)\to X(G)$
 is an isomorphism.  Using this, we see that
 $\base(P_k)\geq\base(X)+k$ for all $k\geq 0$.  Thus, our canonical
 projective resolution is convergent in a convenient sense.
\end{remark}

We now give other useful constructions and examples that we will use later on.

\begin{construction}\label{con-q-V}
 Let $\V\leq \G$ be a subcategory with only finitely many 
 isomorphism classes.  Let $\V^\star$ be the submultiplicative 
 closure of $\V$, so $G\in\V^*$ iff $G$ is isomorphic to a subgroup of
 $\prod_{i=1}^nH_i$ for some family of groups $H_i\in\V$.  
 For a finitely generated group $F$ we put 
 \[ \K(F;\V) = \{N\lhd F\mid F/N\in \V\}. \]
 We can choose a finite list of groups containing one representative
 of each isomorphism class in $\V$, then each group in
 $\K(F;\V)$ will occur as the kernel of one of the finitely many
 surjective homomorphisms from $F$ to one of these groups.  It follows
 that $\K(F;\V)$ is a finite collection of normal subgroups of finite
 index.  We define $\N(F;\V)$ to be the intersection of all the
 subgroups in $\K(F;\V)$, then we put $q_{\V}(F)=F/\N(F;\V)$.  This is
 isomorphic to the image of the natural map 
 \[ F\to\prod_{N\in\K(F;\V)} F/N, \]
 so the submultiplicativity condition ensures that
 $q_{\V}(F)\in\V^\star$.  It is straightforward to check that any
 surjective homomorphism $\phi\colon F_0\to F_1$ has
 $\phi(\N(F_0;\V))\leq\N(F_1;\V)$ and so induces a homomorphism
 $q_\V(F_0)\to q_\V(F_1)$.  This makes $q_{\V}$ into a functor on the
 category of finitely generated groups and surjective homomorphisms.
 If we restrict to finite groups, then the functor $q_{\V}$ is the
 left adjoint to the inclusion $i\colon\V^\star \to \G$.  We therefore
 have $i_!=q_\V^*$ by Lemma~\ref{lem-omnibus}(i).
\end{construction}

\begin{example}\label{ex-q-leq-n}
 Let $\U$ be a submultiplicative subcategory of $\G$ and fix an
 integer $n\geq 1$. If we take $\V=\U_{\leq n}$ then
 $\V^\star \subseteq \U$.  In this case we will use the abbreviated
 notation $q_{\leq n}(F), \K_{\leq n}(F)$ and $\N_{\leq n}(F)$.
\end{example}

\begin{example}\label{ex-free-group}
 Let $\U\leq \G$ be a submultiplicative subcategory and fix an integer
 $n\geq 1$.  For any finite set $X$ of cardinality $n$, let $FX$ be
 the free group on $X$.  Then we put
 $TX = q_{\leq n}(FX) \in\U_{\leq n}^* \subseteq \U$.  This is finite
 and functorial for bijections of $X$.  If $G$ is any group in
 $\U_{\leq n}$, then we can choose a surjective map $X\to G$, and
 extend it to a surjective homomorphism $FX\to G$.  The kernel of this
 homomorphism is in $\mathcal{K}_{\leq n}(FX)$ and so contains
 $\mathcal{N}_{\leq n}(FX)$, so we get an induced surjective
 homomorphism $TX\to G$.  In particular, we can take $X=G$ and use the
 identity map to get a canonical epimorphism
 $\epsilon \colon TG\to G$.
\end{example}

We now consider a natural filtration on objects of $\A\U$ which will 
be useful later on.

\begin{construction}\label{con-filtration-projective}
 Consider an object $X \in \A\U$. For $n \geq 0$, we let $L_{\leq n}X$
 denote the image of the counit map $i_!^{\leq n}i_{\leq n}^*X \to X$. 
 By construction, $L_{\leq n}X$ is the smallest subobject of $X$ 
 containing $X(H)$ for all $H\in\U_{\leq n}$ 
 This gives a filtration
 \[
  0 = L_{\leq 0}X \leq L_{\leq 1}X \leq \dotsb \leq L_{\leq n}X
  \leq L_{\leq n+1}X \leq \dotsb\leq X
 \]
 with subquotients denoted by $L_nX$. Consider a map
 $f \colon X \to Y$ and an element $x\in(L_{\leq n}X)(G)$. We can
 write $x= \sum_{i=1}^s \alpha^*_i(x_i)$ where $x_i \in X(H_i)$ with
 $|H_i|\leq n$ and $\alpha_i\in\U(G,H_i)$. Note that
 \[
  f(x)=\sum_{i} f\alpha_i^*(x) =\sum_{i} \alpha_i^*f(x) \in 
  (L_{\leq n}Y)(G),
 \]
 so the filtration is natural in $X$. Therefore we also have induced 
 maps $L_{\leq n}f \colon L_{\leq n}X \to L_{\leq n}Y$ and
 $L_nf \colon L_nX \to L_nY$ for all $n$.
\end{construction}

\begin{example}\label{ex-filtration-on-generators}
 For all $G \in \U$, we have
 \[ L_{\leq n}(e_G) = \begin{cases}
     0   & \text { if } n < |G| \\
     e_G & \text{if}\;\; n \geq |G|.
    \end{cases}
 \]
 From this we see that $L_n(e_G)=e_G$ if $|G|=n$, and $L_n(e_G)=0$
 otherwise. 
\end{example}

\begin{construction}\label{con-indecomposables}
 Consider an object $X\in \A\U$. We define 
 \[ (QX)(G)= X(G)/\sum_{1 \not = N \triangleleft G} \pi^* X(G/N) \]
 where $\pi \colon G \to G/N$ denotes the projection. Equivalently, if
 $|G|=n$ then this is 
 \[ (QX)(G)=\cok(i_!^{< n}i^*_{< n}X \to X)(G)= X(G)/X_{<n}(G). \]
 We refer to $QX$ as the object of indecomposables of $X$.  This is 
 functorial for isomorphisms of $G$, so we can regard $Q$ as a functor
 $\A\U\to\A\U^\times$.  In fact, it is not hard to see that $Q$ is
 left adjoint to the functor $l_!\colon\A\U^\times\to\A\U$, and that
 the counit map $Ql_!(W)\to W$ is an isomorphism for all
 $W\in\A\U^\times$.  (Indeed, it is sufficient to check this in the
 case $W=e_{G,V}$.)
\end{construction}

\begin{lemma}\label{lem-indecomposables-epi}
 If $f \colon X \to Y$ in $\A\U$ and $Qf$ is an epimorphism, then $f$ is an 
 epimorphism.
\end{lemma}
\begin{proof}
 We will show by induction on $n=|G|$ that $f(G)\colon X(G) \to Y(G)$
 is surjective.  If $n=1$, then $f(G)=(Qf)(G)$ is surjective by
 assumption.  Now suppose that $n>1$ and consider the diagram
 \[
  \begin{tikzcd}
   (i_{!}^{<n}i^*_{<n} X)(G) \arrow[r] \arrow[d,"(i_{!}^{<n}i^*_{<n}f)(G)"] &
   X(G) \arrow[d, "f(G)"] \arrow[r] & 
   (QX)(G) \arrow[d, "(Qf)(G)"] \arrow[r] & 0 \\
   (i_{!}^{<n}i^*_{<n} Y)(G) \arrow[r] &
   Y(G)\arrow[r] &
   (QY)(G) \arrow[r] & 0.
  \end{tikzcd}
 \] 
 By induction we know that $i^*_{<n}f$ is an epimorphism and it follows that 
 the left vertical map in an epimorphism. As $(Qf)(G)$ is an epimorphism too, the 
 claim follows from the four lemma.
\end{proof}

\begin{remark}\label{rem-min-proj}
 We can now use $Q$ to build minimal projective resolutions.  Consider
 an object $X\in\A\U$.  Then $QX$ is a quotient of $l^*X$ in the
 semisimple category $\A\U^\times$, so we can choose a section
 $QX\to l^*X$.  By passing to the adjoint, we get an morphism
 $e\colon P'_0\to X$, where $P'_0=l_!QX$.  We find that $Qe$ is an
 isomorphism, so $e$ is an epimorphism.  We can iterate this in the
 same way as in Construction~\ref{con-proj-inj-resolutions} to get a
 projective resolution $P'_\bullet$ which is minimal in the sense that
 the differential $d_k\colon P'_k\to P'_{k-1}$ has $Q(d_k)=0$ for all
 $k$.  As is familiar for minimal resolutions in other contexts, it
 follows that $P'_\bullet$ is a summand in any other projective
 resolution of $X$.
\end{remark}

\section{Simple objects}
\label{sec-simple}

In this section we classify the simple objects and show that $\A\U$ is
semisimple if and only if $\U$ is a groupoid.

\begin{definition}\label{def-simple-obj}
 Let $\U$ be a replete full subcategory of $\G$.
 \begin{itemize}
  \item An object $X \in \A\U$ is \emph{simple} if the only subobjects
   are $0$ and $X$.
  \item An object $X \in \A\U$ is \emph{semisimple} if it is a sum of
   simple objects.
  \item The abelian category $\A\U$ is \emph{semisimple} if every
   object is semisimple. 
  \item The abelian category $\A\U$ is \emph{split} if every short 
   exact sequence in $\A\U$ splits. Equivalently, every object of 
   $\A\U$ is both injective and projective
 \end{itemize}
\end{definition}

We immediately get the following result. 

\begin{lemma}\label{lem-simple-obj}
 An object $X\in\A\U$ is simple if and only if it is isomorphic to 
 $s_{G,V}$ for some $G$ and some irreducible $k[\Out(G)]$-module $V$.
\end{lemma}
\begin{proof}
 Consider a simple object $X\in\A\U$.  Choose $G$ of minimal order so 
 that $X(G)\neq 0$.  It is standard that the category of
 $k[\Out(G)]$-modules is semisimple, so we can choose a simple
 quotient $V$ of $X(G)$ in this category.  The projection $X(G)\to V$
 is adjoint to a morphism $X\to t_{G,V}$ in $\A\U$.  As $X(H)=0$ when
 $|H|<|G|$, we see that this factors through the subobject
 $s_{G,V}\leq t_{G,V}$.  The morphism $X\to s_{G,V}$ is then an
 epimorphism whose kernel is a proper subobject, and so must be zero
 by simplicity.  Thus $X\simeq s_{G,V}$ as required.
\end{proof}

We are now ready to study when our abelian category is semisimple.

\begin{proposition}\label{prop-semisimple}
 The following are equivalent:
 \begin{itemize}
  \item[(a)] $\U$ is a groupoid;
  \item[(b)] the abelian category $\A\U$ is split;
  \item[(c)] the abelian category $\A\U$ is semisimple.
 \end{itemize}  
\end{proposition}
\begin{proof}
 The fact that (b) and (c) are equivalent is well-known and proved for
 instance in~\cite{Stenstrom}*{V.6.7}.  It is also standard that (a)
 implies (b); the argument will be recalled as
 Proposition~\ref{prop-groupoid-reps-a}(a) below.  Thus, we need only
 prove that~(b) implies~(a), or the contrapositive of that.  Suppose
 that $\U$ is not a groupoid so there exists an epimorphism
 $\varphi\colon G\to H$ which is not an isomorphism.  
 Consider the canonical epimorphism $\pi\colon e_{H,k}\to s_{H,k}$.
 The map $\varphi^*\colon e_{H,k}(H)\to e_{H,k}(G)$ is easily seen to
 be injective.  The map $\varphi^*\colon s_{H,k}(H)\to s_{H,k}(G)$ is
 of the form $k\to 0$ and so is not injective.  It follows that
 $s_{H,k}$ cannot be a retract of $e_{H,k}$, so $\pi$ cannot split.
 Thus, $\A\U$ is not a split abelian category.
\end{proof}

\section{Finite groupoids}
\label{sec-finite-groupoids}

In this section we study the abelian category $\A\U$ in the special
case that $\U\leq \G$ is a finite groupoid. For example we could take
$\U= \lbrace G \in \G \mid |G|=n \rbrace$.

\begin{lemma}\label{lem-groupoid-equivalence}
 Suppose we choose a list of groups $G_1, \ldots, G_r$ containing
 precisely one representative of each isomorphism class of groups in
 $\U$, so $\G(G_i, G_j)= \emptyset$ for $i \not = j$. Let
 $\mathcal{M}_i$ be the category of modules for the group ring
 $k[\Out(G_i)]$ and put $\mathcal{M}= \prod_{i=1}^r
 \mathcal{M}_i$. Then the functor $X \mapsto (X(G_i))_{i=1}^r$ gives
 an equivalence of categories $\A\U \to \mathcal{M}$.
\end{lemma}
\begin{proof}
 The inverse functor is given by $(V_i)_{i=1}^r \mapsto 
 \bigoplus_{i=1}^r e_{G_i,V_i}$.
\end{proof}

\begin{remark}\label{rem-functors}
 Let $i\colon \U \to \G$ denote the inclusion functor. After choosing
 a list of groups $G_1, \ldots, G_r \in \U$ as in
 Lemma~\ref{lem-groupoid-equivalence}, we have identifications
 \[
  i_!= \bigoplus_{i=1}^r e_{G_i, \bullet} \quad \text{and} \quad 
  i_*= \bigoplus_{i=1}^r t_{G_i, \bullet}.
 \]
\end{remark}		

\begin{proposition}\label{prop-groupoid-reps-a}
 Suppose that $\U\subseteq\V\subseteq\G$, and let $i\colon\U\to\V$
 denote the inclusion. 
 \begin{itemize}
  \item[(a)] All monomorphisms and epimorphisms in $\A\U$ are split.
  \item[(b)] All objects in $\A\U$ are both injective and projective.
  \item[(c)] All objects in the image of $i_!$ are projective, and all
   objects in the image of $i_*$ are injective.
  \item[(d)] The functor $i_!$ preserves all limits and colimits, as
   does the functor $i_*$.
 \end{itemize}
\end{proposition}	
\begin{proof}
 We identify $\A\U$ with $\mathcal{M}$ as in Lemma above. Maschke's
 Theorem shows that $(a)$ and $(b)$ hold in $\mathcal{M}_i$, and it
 follows that they also hold in $\mathcal{M}$ and $\A\U$. If
 $X\in\A\U$ then the functor $\A(i_!(X), -)$ is isomorphic to
 $\A\U(X, i^*(-))$. Here $i^*$ and $\A\U(X,-)$ preserve epimorphisms,
 so $i_!(X)$ is projective. Similarly, we see that $i_*(X)$ is
 injective, which proves $(c)$.

 We next claim that $i_*$ preserves all limits and colimits.  As it is
 a right adjoint it is enough to show that it preserves all colimits.
 By Remark~\ref{rem-functors}, it is enough to show that the functor
 $t_{G_k,\bullet}$ preserves colimits for all $1\leq k\leq r$.
 Choose $f_1, \ldots, f_s \in \G(G_k, G)$, containing precisely one
 element from each $\Out(G_k)$-orbit. Let $\Delta_s \leq \Out(G_k)$ be
 the stabiliser of $f_s$. We find that
 \[  t_{G_k,V}(G) = 
     \Hom_{k[\Out(G_k)]}(k[\G(G_k,G)], V) =
     \prod_{s} V^{\Delta_s},
 \]
 and this is easily seen to preserve all colimits as required.  A
 similar argument shows that $i_!$ preserves all limits and
 colimits. As before, it is enough to show that the functor
 $e_{G_k, \bullet}$ preserves all limits.  We find that
 \[
   e_{G_k,V}(G)= k[\G(G,G_k)] \otimes_{k[\Out(G_k)]} V =\bigoplus_s 
   V_{\Delta_s}
 \]
 and this is easily seen to preserve all limits.
\end{proof}

The following results are standard. 

\begin{proposition}\label{prop-groupoid-reps-b}
\leavevmode
 \begin{itemize}
  \item[(a)] The simple objects of $\A\U$ are the same as the 
   indecomposable objects, and these are precisely the objects 
   $e_{G,V}$ for some $G\in\U$ and irreducible 
   $k[\Out(G)]$-module $V$. 
  \item[(b)] Every nonzero morphism to a simple object is a split 
   epimorphism, and every nonzero morphism from a simple object is a 
   split monomorphism.
  \item[(c)] If $S$ and $S'$ are non-isomorphic simple objects in $\A\U$, 
   then $\A\U(S,S')=0$.
  \item[(d)] If $S$ is a simple object in $\A\U$, then $\End(S)$ 
   is a division algebra of finite dimension over $k$.
  \item[(e)] The category $\A\U$ has finitely many isomorphism classes of 
   simple objects.
  \item[(f)] Suppose that the list $S_1,\ldots, S_s$ contains precisely 
   one simple object from each isomorphism class, and put 
   $D_j= \End(S_j)$. Let $\mathcal{N}_j$ be the category of right 
   modules over $D_j$, and put $\mathcal{N}=\prod_j \mathcal{N}_j$. Define 
   functors
   \[
    \A\U \xrightarrow{\phi} \mathcal{N} 
    \xrightarrow{\psi} \A\U
   \]
   by $\phi(X)_j=\A\U(S_j,X)$ and $\psi(N)=\bigoplus_jN_j\otimes_{D_j} S_j$.
   Then $\phi$ and $\psi$ are inverse to each other, and so are equivalences.
\end{itemize}
\end{proposition}

\begin{proof}
 The first part of (a) is clear from the fact that all monomorphisms
 are split. As any morphism in $\U$ is an isomorphism we see that
 $e_{G,V}=s_{G,V}$ and this is simple when $V$ is irreducible, see
 Lemma~\ref{lem-simple-obj}. For (b), suppose that
 $\alpha\colon X \to S$ is nonzero, where $S$ is simple. Then
 $\img(\alpha)$ is a nonzero subobject of $S$, so it must be all of
 $S$, so $\alpha$ is an epimorphism, and all epimorphisms are
 split. This gives half of $(b)$, and the other half is similar. Now
 suppose that $\alpha \colon S \to S'$, where both $S$ and $S'$ are
 simple. If $\alpha \not = 0$ then $(b)$ tell us that $\alpha$ is both
 a split monomorphism and a split epimorphism, so it is an
 isomorphism. The contrapositive gives claim (c), and the special case
 $S'=S$, gives most of (d), apart from the finite-dimensionality
 statement. For that, we choose a list of groups $G_i$ as in Lemma
 \ref{lem-groupoid-equivalence}, and put $U= \bigoplus_i e_{G_i}$
 which is a generator for $\A\U$. We can decompose $U$ as a finite
 direct sum of indecomposables, say $U= \bigoplus_{j=1}^s S_j^{d_j}$
 with $0 < d_{j}< \infty $ and $S_j \not \simeq S_k$ for $j \not =
 k$. If $S$ is simple, there is an nonzero map $U \to S$ and so a
 nonzero map $S_j \to S$ for some $j$, that has to be an isomorphism
 from (b). This proves (e). We also note that $S$ is a summand in $U$,
 so $\End(S)$ is a summand in $\End(U)$ and hence it has finite
 dimension over $k$, completing the proof of (d).

 Now define $\phi$ and $\psi$ as in (f).  Put
 $T_m= \psi(S_m) \in \mathcal{N}$, so $(T_m)_m= D_m$ and $(T_m)_j=0$
 for $j \not = m$. Define
 \[
  \eta_N \colon N \to \phi \psi (N)= \A\U(S_j, \bigoplus_{k} N_k
  \otimes_{D_k} S_k)
 \]
 \[
  \epsilon_X \colon \psi \phi (X)= \bigoplus_j \A\U(S_j, X)
  \otimes_{D_j}S_j \to X
 \]
 as follows. First, any $n \in N_j$ gives a map $D_j \to N_j$ and thus
 a map
 \[
  S_j = S_j \otimes_{D_j} D_j \to S_j \otimes_{D_j} N_j \leq
  \bigoplus_k N_k \otimes_{D_k} S_k
 \]
 we take this to be the $j$-th component of $\eta_N$. Similarly, there
 is an evaluation morphism $\A\U(S_j,X) \otimes S_j \to X$, which is
 easily seen to factor through $\A\U(S_j,X) \otimes_{D_j} S_j$. We
 combine these maps to give $\epsilon_X$.

 We claim that $\epsilon_X$ is an isomorphism. Indeed, we know that
 the object $U$ is a generator for $\A\U$, so the objects $S_j$ form a
 generating family. As all epimorphisms in $\A\U$ split, we see that
 every object is a retract of a direct sum of objects of the form
 $S_m$. We also see that both $\phi$ and $\psi$ preserve all direct
 sums. It will therefore suffice to check that $\epsilon_{S_m}$ is an
 isomorphism, and this follows easily from our description of
 $T_m= \psi(S_m)$.

 Because every module over a division algebra is free, we also see
 that every object of $\mathcal{N}$ is a direct sum of objects of the
 form $T_m$. It is easy to see that $\eta_{T_m}$ is an isomorphism,
 and it follows that $\eta_N$ is an isomorphism for all $N$.
\end{proof}

\section{Projectives}
\label{sec-projectives}

In this section we study and classify the projective
objects of $\A\U$ for a replete full subcategory $\U$ of $\G$. 
  
\begin{lemma}\label{lem-proj-as-summand}
 Consider an object $P$ in $\A\U$. Then the following are equivalent:
 \begin{itemize}
  \item[(a)] $P$ is projective in $\A\U$.
  \item[(b)] $P$ is isomorphic to a retract of a direct sum of 
   objects of the form $e_G$ with $G \in \U$.
 \end{itemize}
\end{lemma}
\begin{proof}
 First suppose that $(a)$ holds. Let $\U_0$ be a countable collection
 of objects of $\U$ that contains at least one representative of every
 isomorphism class. Put
 \[
  FP= \bigoplus_{G \in \U_0}\bigoplus_{x \in P(G)}e_G \in \A\U.
 \]
 Each pair $(G,x)$ defines a morphism
 $\epsilon_{(G,x)} \colon e_G \to P$ by the Yoneda Lemma. By combining
 these for all pairs $(G,x)$, we get a morphism
 $\epsilon \colon FP \to P$ which is an epimorphism by construction.
 As $P$ is assumed to be projective this epimorphism must split, so
 $P$ is a retract of $FP$, so $(b)$ holds. 
 Conversely, suppose that $P$ is as in (b) and note that $e_G$ is projective 
 since $\A\U(e_G, -)$ is exact by the Yoneda Lemma. 
 Sums and retracts of projective objects is again 
 projective so (a) follows.
\end{proof}

\begin{lemma}\label{lem-i-preserves-projective}
 Let $i\colon\U\to\V$ be an inclusion of replete full subcategories of
 $\G$, and let $P$ be an object of $\A\U$.  Then $P$ is projective in
 $\A\U$ iff $i_!(P)$ is projective in $\A\V$.
\end{lemma}
\begin{proof}
 First, if $P$ is projective then the functor $\A\V(i_!(P),-)$ is
 isomorphic to the composite of the exact functors
 $i^*\colon\A\V\to\A\U$ and $\A\U(P,-)$, so it is exact, so $i_!(P)$
 is projective.  

 Conversely, suppose that $i_!(P)$ is projective.  We can certainly
 choose a projective object $Q\in\A\U$ and an epimorphism
 $u\colon Q\to P$.  As $i_!$ is a left adjoint, it preserves
 colimits and epimorphisms, so $i_!(u)\colon i_!(Q)\to i_!(P)$ is an
 epimorphism.  As $i_!(P)$ is assumed projective, we can choose
 $v\colon i_!(P)\to i_!(Q)$ with $i_!(u)\circ v=1$.  We now apply
 $i^*$ to this, recalling that $i^*i_!\simeq 1$; we find that $i^*(v)$
 gives a section for $u$, so $u$ is a split epimorphism, so $P$ is
 projective. 
\end{proof}

\begin{proposition}\label{prop-projective-criteria}
 Let $i\colon\U\to\V$ be an inclusion of replete full subcategories of
 $\G$, and let $Q$ be an object of $\A\V$.  Then the following are
 equivalent:
 \begin{itemize}
  \item[(a)] $Q \simeq i_!(P)$ for some projective object 
   $P \in \A\U$.
  \item[(b)] $Q$ is a retract of $i_!(P)$ for some projective object 
   $P \in \A\U$.
  \item[(c)] $Q$ is a retract of some direct sum of objects $e_G$, 
   with $G \in \U$.
  \item[(d)] $Q$ is projective, and the counit map $i_!i^*Q \to Q$ is 
   an isomorphism.
 \end{itemize}		
 Moreover, if these conditions hold then $i^*(Q)$ is projective in $\A\U$.
\end{proposition}
\begin{proof}
 From what we have seen already it is clear that
 $(a) \Rightarrow (b) \Leftrightarrow (c)$ and that
 $(a) \Rightarrow (d)$.  Now suppose that $(b)$ holds, so there is a
 projective object $P \in \A\U$ and an idempotent
 $e \colon i_!P \to i_!P$ with $Q= e.(i_!P)= \cok(1-e)$. As $i_!$ is
 full and faithful, there is an idempotent $f \colon P \to P$ with
 $i_!(f)=e$. As $i_!$ preserves cokernels, it follows that
 $Q=i_!(f.P)$, and of course $f.P$ is projective, so $(a)$ holds.
 Also, if $Q \simeq i_!P$ as in $(a)$ holds then $i^*Q$ is isomorphic
 to $P$ and so is projective.

 Now all that is left is to prove that $(d) \Rightarrow (b)$. Suppose
 that $Q$ is projective, and that the counit map $i_!i^*Q \to Q$ is an
 isomorphism. Choose a projective $P\in \A\U$ and an epimorphism
 $f \colon P \to i^*Q$. As $i_!$ preserves epimorphisms, we see that
 $i_!(f) \colon i_!P \to i_!i^*Q \simeq Q$ is an epimorphism, but $Q$
 is projective, so $Q$ is a retract of $i_!P$ as required.
\end{proof}

Recall the functors $L_{\leq n}$ and $L_{n}$ from
Construction~\ref{con-filtration-projective}.  Recall also that we put
$\U_n=\{G\in\U\mid |G|=n\}$ (which is a groupoid), and we write $i_n$
for the inclusion $\U_n\to\U$.

\begin{proposition}\label{prop-filtration-split}
 If $P$ is projective in $\A\U$, then the filtration $L_{\leq *}P$ can
 be split, so there is an unnatural isomorphism
 $P\simeq\bigoplus_nL_nP$, and the filtration quotients $L_nP$ are
 themselves projective.  Furthermore, $i^*_n(L_nP)$ is projective in
 $\A\U_n$ and $(i_n)_!(i^*_nL_nP)=L_nP$.
\end{proposition}
\begin{proof}
 We have seen that $P$ can be written as a retract of a direct sum of
 generators.  In more detail, we can choose an object
 $Q=\bigoplus_{\alpha} e_{G_\alpha}$ and an idempotent
 $u\colon Q\to Q$ such that $P\simeq u.Q$, so without loss of
 generality $P=u.Q$.  Let $Q_n$ be the sum of the terms $e_{G_\alpha}$
 for which $|G_\alpha|=n$, so that $Q=\bigoplus_nQ_n$.  We can then
 decompose $u$ as a sum of morphisms $u_{nm}\colon Q_m\to Q_n$.  When
 $m<n$ we have $\A\U(Q_m,Q_n)=0$ and so $u_{nm}=0$.  Given this, the
 relation $u^2=u$ implies that $u_{nn}^2=u_{nn}$.  The object
 $P'_n=u_{nn}.Q_n$ is therefore projective.  Put $P'=\bigoplus_nP'_n$
 and let $f\colon P'\to P$ be the composite
 $P'\xrightarrow{\text{inc}}Q\xrightarrow{u}u.Q=P$.  We claim that
 this is an isomorphism. By passing to the colimit, it will suffice to
 show that $L_{\leq n}(f)$ is an isomorphism for all $n$.  By an evident
 reduction, it will suffice to show that $L_n(f)$ is an isomorphism for
 all $n$.  As $L_n$ is an additive functor we have
 $L_n(P)=L_n(u).L_n(Q)=L_n(u).Q_n=u_{nn}.Q_n=P'_n$, as required.  All
 remaining claims are now easy.
\end{proof}

\begin{corollary}\label{cor-indecomposable-projectives}
 Suppose we choose a complete system of simple objects in $\A\U_n$ for
 all $n$, giving a sequence $(e_{G_i,S_i} \mid G_i \in \U_n)_n$ of
 indecomposable projectives in $\A\U$. Then every projective object is
 a direct sum of objects of the form $e_{G_i,S_i}$. In particular,
 every indecomposable projective is isomorphic to some $e_{G_i,S_i}$.
\end{corollary}
\begin{proof}
 Because $\A\U_n$ is semisimple, we see that $i_n^*(L_nP)$ splits in
 the indicated way.  As $L_nP\simeq(i_n)_!(i_n^*(L_nP))$, we see that
 $L_nP$ also splits, as does $\bigoplus_nL_nP$, which is isomorphic to
 $P$. 
\end{proof}

\begin{proposition}\label{prop-product-projectives}
 Any projective object $P$ can be written as
 $P\simeq\prod_n L_nP$. Furthermore, products of projective objects
 are projective. 
\end{proposition}
\begin{proof}
 By Proposition~\ref{prop-filtration-split} we can write
 $P=\bigoplus_n L_nP$. Now note that for a fixed $G\in\U$, there are
 only finitely many indices $n$ such that $P_n(G)$ is nonzero, so the
 inclusion $\bigoplus_n L_nP\to\prod_n L_nP$ is an isomorphism.  

 For the second claim, let $(P_\alpha)$ be a family of projectives,
 and put $P=\prod_\alpha P_\alpha$.  We can write
 $P_\alpha=\prod_k L_kP_\alpha$ as above, so $P=\prod_k Q_k$ where
 $Q_k=\prod_\alpha L_kP_\alpha$.  We know from
 Proposition~\ref{prop-groupoid-reps-a} that $(i_k)_!$ preserves
 products, so $Q_k$ is in the image of $(i_k)_!$.  It follows that
 $Q_k$ is projective and also that $P=\prod_kQ_k$ is the same as
 $\bigoplus_kQ_k$, so $P$ is projective.
\end{proof}

\begin{proposition}\label{prop-preserve-projectives}
 Let $\U$ be a widely closed subcategory of $\G$. Then the full
 subcategory of projective objects is closed under tensor products. If
 $\U$ is a multiplicative global family, then the full subcategory of
 projective objects is also closed under the internal homs.
\end{proposition}
\begin{proof}
 Consider projective objects $P,Q\in\A\U$.  We can write $P$ as a
 retract of a direct sum of terms $e_G$.  The functor $(-)\otimes Q$
 sends sums to sums, and the functor $\uHom(-,Q)$ sends sums to
 products, and both sums and products of projectives are projective.
 We can therefore reduce to the case $P=e_G$.  Next, we can split $Q$
 as a direct sum or product of terms $L_nQ$.  The functor
 $e_G\otimes(-)$ preserves sums, and the functor $\uHom(e_G,-)$
 preserves products, so we can reduce to the case where $Q=L_nQ$, or
 equivalently $Q=(i_n)_!(M)$ for some $M\in\A\U_n$.  We can now write
 $M$ as a retract of a sum of terms $e_{H_\alpha}$ with
 $|H_\alpha|=n$.  We know from Proposition~\ref{prop-wide} that
 $e_G\otimes e_{H_\alpha}$ is projective, and it follows easily that
 $e_G\otimes Q$ is projective as claimed.

 It also follows from Proposition~\ref{prop-wide}, together with the
 formula $\uHom(e_G,Z)(H)=\A\U(e_G\otimes e_H,Z)$, that the functor
 $\uHom(e_G,-)$ preserves sums.  If $\U$ is a multiplicative global
 family, then Theorem~\ref{thm-hom-fg-projective} tells us that
 $\uHom(e_G,e_{H_\alpha})$ is projective.  From these two facts it
 follows that $\uHom(e_G,Q)$ is also projective, which finishes the
 proof. 
\end{proof}

\section{Colimit-exactness}
\label{sec-exact-colimit}

Let $\U$ be a subcategory of $\G$.  In this section we will write $L$
for the colimit functor $X\mapsto\colim_{G\in\U^{\op}}X(G)$ from
$\A\U$ to $\Vect_k$.  Recall that $\U$ is said to be \emph{colimit-exact}
if $L$ is an exact functor.  We will show that most of our examples
have this property.  First, however, we give an equivalent condition.

\begin{proposition}\label{prop-one-injective}
 There is a natural isomorphism $\A\U(X,\one)\simeq\Vect_k(LX,k)$.
 Thus, the object $\one\in\A\U$ is injective if and only if $\U$ is
 colimit-exact.  If so, then all objects of the form
 $DX=\uHom(X,\one)$ are also injective.
\end{proposition}
\begin{proof}
 The natural isomorphism $\A\U(X,\one)\simeq\Vect_k(LX,k)$ is clear.  The
 functor $V\mapsto V^*=\Vect_k(-,k)$ is certainly exact, so if $L$ is
 exact then $\A\U(-,\one)$ is exact, so $\one$ is injective.
 Conversely, suppose that $\one$ is injective.  For any short exact
 sequence $X\to Y\to Z$ in $\A\U$, we deduce that the resulting
 sequence $(LZ)^*\to(LY)^*\to(LZ)^*$ is also short exact, and then
 linear algebra shows that $LX\to LY\to LZ$ is short exact as well.
 This proves that $L$ is exact.  Also, there is a natural isomorphism
 $\A\U(X,DW)=\A\U(W\otimes X,\one)$.  The functors $W\otimes(-)$ and
 $\A\U(-,\one)$ are exact, and it follows that $DW$ is injective as
 claimed. 
\end{proof}

\begin{lemma}\label{lem-colimit-values}
\leavevmode
 \begin{itemize}
  \item[(a)] For any $G\in\U$ we have $Le_G=k$.  In particular, if 
  $\one\in\U$ then $L\one=L e_1=k$.
  \item[(b)] For any $G\in\U$ and any $k[\Out(G)]$-module $V$ we have
   $L(e_{G,V})=V_{\Out(G)}$ (the module of coinvariants).
  \item[(c)] Unless $G$ is maximal in $\U$ we also have
   $L(t_{G,V})=L(s_{G,V})=0$.
 \end{itemize}
\end{lemma}
\begin{proof}
 For $T\in\Vect_k$ we have 
 \[ \Vect_k(L(e_{G,V}),T) = \A\U(e_{G,V},T\otimes\one) = 
     \mathrm{Mod}_{k[\Out(G)]}(V,T) = 
      \Vect_k(V_{\Out(G)},T).
 \]
 By the Yoneda Lemma, we therefore have $L(e_{G,V})=V_{\Out(G)}$.
 Taking $V=k[\Out(G)]$ gives $L(e_G)=k$.  

 Now consider the object $L(t_{G,V})$.  This is the colimit over
 $H\in\U$ of the groups $t_{G,V}(H)=\Map_{\Out(G)}(\U(G,H),V)$.  If
 there are no morphisms $G\to H$, then $t_{G,V}(H)=0$.  If there is a
 morphism $\alpha\colon G\to H$, then by definition the limit map
 $t_{G,V}(H)\to L(t_{G,V})$ factors through $\alpha^*$.  This makes it
 clear that the map $t_{G,V}(G)\to L(t_{G,V})$ is surjective.  Now
 suppose that $G$ is not maximal in $\U$, so we can choose
 $\beta\colon K\to G$ in $\U$ that is not an isomorphism.  The map
 $t_{G,V}(G)\to L(t_{G,V})$ will then factor through $\beta^*$, but
 the codomain of $\beta^*$ is zero, so $L(t_{G,V})=0$.  A simpler
 version of the same argument also gives $L(s_{G,V})=0$.
\end{proof}

\begin{remark}\label{rem-L-not-monoidal}
 For $X,Y\in\A\U$ there are natural unit maps $X\to(LX)\otimes\one$ and
 $Y\to(LY)\otimes\one$.  We can tensor these together and take adjoints to
 get a map $L(X\otimes Y)\to(LX)\otimes(LY)$.  This gives an oplax monoidal
 structure on $L$.  However, the map $L(X\otimes Y)\to(LX)\otimes(LY)$ is
 rarely an isomorphism.  For example, we have $Le_G\otimes Le_H=k$ but
 Proposition~\ref{prop-wide} shows that $L(e_G\otimes e_H)$ is freely
 generated by the set $\Wide(G,H)/\text{conjugacy}$. 
\end{remark}

We now start to prove that various categories are colimit-exact.  Our
first example is easy:

\begin{proposition}\label{prop-groupoid-colimit-exact}
 If $\U\leq\G$ is a groupoid, then it is colimit-exact.  
\end{proposition}
\begin{proof}
 Choose a family $(G_i)_{i\in I}$ containing precisely one
 representative of each isomorphism class in $\U$.  If $X\in\A\U$ then
 the group $\Out(G_i)$ acts on $X(G_i)$, and we write
 $X(G_i)_{\Out(G_i)}$ for the module of coinvariants.  As we work
 over a field of characteristic zero and $\Out(G_i)$ is finite, this is an 
 exact functor of $X$.  It is easy to identify $\colim X$ with
 $\bigoplus_iX(G_i)_{\Out(G_i)}$, and this makes it clear that the
 colimit functor is exact as well.
\end{proof}

For other examples we will use the following notion:

\begin{definition}\label{def-colimit-tower}
 A \emph{colimit tower} for $\U$ is a diagram 
 \[ G_0 \xleftarrow{\epsilon_0} G_1 \xleftarrow{\epsilon_1} G_2 \xleftarrow{} \dotsb \] 
 in $\U$ such that
 \begin{itemize}
  \item[(a)] For every $H\in\U$ there is a pair $(i,\alpha)$ with
   $i\in\NN$ and $\alpha\in\U(G_i,H)$.
  \item[(b)] For every diagram 
   $G_i\xrightarrow{\alpha}H\xleftarrow{\beta}G_i$ in 
   $\U$ there exists $\gamma\in\U(G_{i+1},G_{i+1})$ making the following 
   diagram commute: 
   \begin{center}
    \begin{tikzcd}
     G_{i+1}
      \arrow[d,"\epsilon_i"]
      \arrow[rr,"\gamma"] & &
     G_{i+1} \arrow[d,"\epsilon_i"']\\
     G_i \arrow[r,"\alpha"'] & H & G_i. \arrow[l,"\beta"]
    \end{tikzcd}
   \end{center}
  \item[(c)] For every diagram
   $G_{i+1}\xrightarrow{\alpha}H\xleftarrow{\phi}K$ in $\U$ with 
   $\U(G_i, K)\not=\emptyset$, there
   exists $\beta\in \U(G_{i+1},K)$ such that $\phi\circ \beta=\alpha$.
 \end{itemize}
\end{definition}

\begin{construction}\label{con-colimit-tower}
 Suppose we have a colimit tower as above.  For any $X\in\A\U$ we
 define $\Lambda_iX$ to be the group of coinvariants
 $X(G_i)_{\Out(G_i)}$, and we let $\rho_i\colon X(G_i)\to\Lambda_iX$
 be the obvious reduction map.  By taking $H=G_i$ and $\beta=1$ in
 condition~(b), we see that every automorphism of $G_i$ can be covered
 by an automorphism of $G_{i+1}$.  It follows that there is a unique
 map $\Lambda_iX\to\Lambda_{i+1}X$ making the following diagram commute:
 \begin{center}
  \begin{tikzcd}
   X(G_i) \arrow[r,"\epsilon_i^*"] \arrow[d,"\rho_i"'] &
   X(G_{i+1}) \arrow[d,"\rho_{i+1}"] \\
   \Lambda_iX \arrow[r] & 
   \Lambda_{i+1}X 
  \end{tikzcd}
 \end{center}
 We will just write $\epsilon_i^*$ for this map.  We define
 $\Lambda_\infty X$ to be the colimit of the sequence 
 \[ \Lambda_0X \xrightarrow{\epsilon_0^*} 
    \Lambda_1X \xrightarrow{\epsilon_1^*} 
    \Lambda_2X \xrightarrow{\epsilon_2^*} \dotsb,
 \] 
 and we let $\sigma_i$ denote the canonical map
 $\Lambda_iX\to\Lambda_\infty X$.  As we are working over a field of 
 characteristic zero and
 $\Out(G_n)$ is finite, we see that $\Lambda_n$ is an exact functor.
 As sequential colimits are exact, we see that
 $\Lambda_\infty\colon\A\U\to\Vect_k$ is also an exact functor. 
\end{construction}

\begin{proposition}\label{prop-colimit-tower}
 For any colimit tower, there is a natural isomorphism
 $\Lambda_\infty X\to LX$.  Thus, if $\U$ has a colimit
 tower, then it is colimit-exact.
\end{proposition}
\begin{proof}
 Let $\theta_H\colon X(H)\to LX$ 
 be the canonical morphism.  It is formal that there is a unique map
 $\phi\colon \Lambda_\infty X\to LX$ with 
 $\phi\sigma_i\rho_i=\theta_{G_i}$
 for all $i$.  In the opposite direction, suppose we have $H\in\U$.
 We can choose $(i,\alpha)$ as in condition~(a) and define 
 \[ \psi_{H,i,\alpha}=
     \sigma_i\rho_i\alpha^*\colon X(H) \to \Lambda_\infty X.
 \]
 Using the obvious cone properties we see that this is the same as
 $\psi_{H,i+1,\alpha\epsilon_i}$, or as $\psi_{H,i,\alpha\mu}$
 for any $\mu\in\Out(G_i)$.  By using these rules together with
 condition~(b), we see that $\psi_{H,i,\alpha}$ is independent of the
 choice of $(i,\alpha)$, so we can just denote it by $\psi_H$.  It is
 now easy to see that for any $\zeta\colon H\to K$ we have
 $\psi_H\zeta^*=\psi_K\colon X(K)\to\Lambda_\infty X$.  This means
 that there is a unique $\psi\colon LX\to\Lambda_\infty X$ with
 $\psi\theta_H=\psi_H$ for all $H$.  This is clearly inverse to $\phi$.
\end{proof}

\begin{remark}\label{rem-condition-c-colim-tower}
 So far we only used conditions (a) and (b) in the definition of 
 colimit tower. Condition (c) will play an important role in 
 Section~\ref{sec-torsion}. 
\end{remark}

\begin{example}\label{ex-cyclic-colimit-exact}
 Let $\C$ be the category of cyclic groups.  The morphisms can be
 described as follows:
 \begin{itemize}
  \item[(a)] If $|G|=n$ then the group $\Aut(G)=\Out(G)$ is
   canonically identified with $(\ZZ/n)^\times$.  
  \item[(b)] If $|H|$ divides $|G|$ then $\C(G,H)$ is a torsor for
   $\Aut(H)$.  Moreover, for any $\alpha\colon G\to H$ and
   $\phi\in\Aut(H)=(\ZZ/|H|)^\times$ there exists
   $\psi\in\Aut(G)=(\ZZ/|G|)^\times$ that reduces to $\phi$, and
   any such $\psi$ satisfies $\alpha\psi=\phi\alpha$.
  \item[(c)] On the other hand, if $|H|$ does not divide $|G|$ then
   $\C(G,H)=\emptyset$. 
 \end{itemize}
 From these observations it follows easily that the groups
 $G_n=\ZZ/{n!}$ form a colimit tower, and so $\C$ is colimit-exact.
 Similarly, the groups $\ZZ/{p^n}$ form a colimit tower in the category
 $\C[p^\infty]$ of cyclic $p$-groups, so $\C[p^\infty]$ is also
 colimit exact.  For a more degenerate example, we can fix a positive
 integer $d$ and consider the category $\C[d]$ of cyclic groups
 of order dividing $d$.  We find that the constant sequence with value
 $\ZZ/d$ is a colimit tower for $\C[d]$, so this category is
 again colimit-exact.
\end{example}

 Recall the category $\Z[p^r]$ of finitely generated $\ZZ/p^r$-modules and its 
 subcategory $\F[p^r]$ of free $\ZZ/p^r$-modules.

\begin{lemma}\label{lem-dotted-arrow}
  Consider a diagram of epimorphisms 
  \[
  \begin{tikzcd}
   A \arrow[dr, two heads, "\alpha"']
     \arrow[rr, dotted, two heads, "\gamma"] & & 
   B \arrow[dl, two heads, "\beta"] \\
    & C & 
  \end{tikzcd}
  \] 
  with $A \in \F[p^r], B \in \Z[p^r]$ and $\rk(A) \geq \rk(B)$. 
  Then the dotted arrow can be filled by another epimorphism. 
\end{lemma}
\begin{proof}
 Put $\rk(A)=n$, $\rk(B)=m$ and $\rk(C)=l$ so that $n\geq m\geq l$.
 Choose elements $c_1,\dotsc,c_l\in C$ that project to a basis of
 $C/pC$ over $\ZZ/p$ (so they form a minimal generating set for $C$).
 Choose elements $b_1,\dotsc,b_l\in B$ with $\beta(b_i)=c_i$.  The
 images of $b_1,\dotsc,b_l$ in $B/pB$ will then be linearly
 independent.  Choose additional elements $b_{l+1},\dotsc,b_m\in B$ so
 that $b_1,\dotsc,b_m$ gives a basis for $B/pB$.  After adding
 multiples of $b_1,\dotsc,b_l$ to $b_{l+1},\dotsc,b_m$ if necessary,
 we can ensure that $\beta(b_i)=0$ for $i>l$.  In the same way, we can
 find elements $a_1,\dotsc,a_n\in A$ such that $\alpha(a_i)=c_i$ for
 $i\leq l$, and $\alpha(a_i)=0$ for $i>l$, and $a_1,\dotsc,a_n$ gives
 a basis for $A/pA$ over $\ZZ/p$.  As $A$ is free of rank $n$ over
 $\ZZ/p^r$, it follows that the same elements give a basis over
 $\ZZ/p^r$.  Thus, there is a unique morphism $\gamma\colon A\to B$
 with 
 \[ \gamma(a_i)= 
     \begin{cases}
      b_i & \text{ if } \;\; 0 \leq i \leq m \\
      0   & \text{ if } \;\; m < i \leq n.
     \end{cases}
 \]
 As all the generators $b_i$ lie in the image of $\gamma$, we see that
 $\gamma$ is surjective.  It also satisfies $\beta\gamma=\alpha$ by
 construction. 
\end{proof}

\begin{example}\label{ex-abelian-groups-colimit-exact}
 Consider the category $\Z[p^r]$ of finite abelian $p$-groups of
 exponent dividing $p^r$, which is equivalent to the category of
 finitely generated $\ZZ/p^r$-modules and linear maps. Using
 Lemma~\ref{lem-dotted-arrow} one sees that the groups $(\ZZ/p^r)^n$
 form a colimit tower.  It follows that $\Z[p^r]$ is colimit-exact. As
 these groups lie in the subcategory $\F[p^r]\leq\Z[p^r]$, it is clear
 that they form a colimit tower for that subcategory as well.
\end{example}

Most of the rest of this section is devoted to the proof of the
following result:

\begin{theorem}\label{thm-exact}
 If $\U\leq\G$ is submultiplicative then it is colimit-exact.
\end{theorem}

We will prove this by giving a less explicit but much more general
construction of colimit towers. For this, we will need a bit of
preparation. Recall the functor $T$ from Example~\ref{ex-free-group}.

\begin{lemma}\label{lem-TX-lift}
 Let $X$ be a finite set and consider a diagram of epimorphisms
 between groups in $\U$
 \[
  \begin{tikzcd}
    & G \arrow[d, two heads, "\alpha"'] \\
    TX \arrow[r, two heads, "\lambda"]
       \arrow[ur, dotted, two heads, "\mu"] & H
  \end{tikzcd}
 \]
 in which $|G|\leq |X|$. Then the dotted arrow can be filled in by
 another epimorphism.
\end{lemma}
\begin{proof}
 Put $L=\ker(\alpha)$, so $|L||H|=|G|\leq |X|$. Let $i\colon X\to TX$
 be the canonical inclusion, and put
 $X_h=(\lambda i)^{-1}\{h\}\subseteq X$ for each $h \in H$.  We then
 have $\sum_h|X_h|=|X|\geq |H||L|$, so we can choose $h_0$ with
 $|X_{h_0}|\geq |L|$.  Let $\mu_h\colon X_h\to\alpha^{-1}\{h\}$ be
 chosen arbitrarily, except that we choose $\mu_{h_0}$ to be
 surjective. By combining these maps, we get $\mu'\colon X \to G$
 such that $\alpha\mu'=\lambda i$. By the defining properties of $TX$,
 we see that there is a unique homomorphism $\mu\colon TX\to G$ with
 $\mu i=\mu'$. This satisfies $\alpha\mu i=\lambda i$ and $i(X)$
 generates $TX$ so $\alpha\mu=\lambda$. Now note that the restriction
 of $\alpha$ to the image of $\mu$ is an epimorphism since $\alpha\mu$
 is surjective. Also, the image of $\mu$ contains $L$ as $\mu_{h_0}$
 is surjective. It follows that $|\mathrm{Im}(\mu)|=|L||H|=|G|$ so
 $\mu$ is surjective as required.
\end{proof}

\begin{lemma}\label{lem-TG-size}
 If $G\neq 1$ then $\epsilon \colon TG\to G$ is not injective, so
 $|TG|\geq 2|G|$.
\end{lemma}
\begin{proof}
 Choose any nontrivial $g\in G$ and let $\tau \colon G\to G$ be the
 transposition that exchanges $1$ and $g$. Let $e_1$ and $e_g$ denote
 the corresponding generators of $FG$ or $TG$. The map $\tau$ induces
 an automorphism $\alpha$ of $TG$ which exchanges $e_1$ and $e_g$.
 The homomorphism $\epsilon \alpha$ sends $e_1$ to $g\neq 1$, so
 $e_1\not\in N$, so $e_1$ gives a nontrivial element of $TG$.
 However, this lies in the kernel of $\epsilon$, so $\epsilon$ is not
 injective, and $|TG|=|G||\ker(\epsilon)|\geq 2|G|$.
\end{proof}

\begin{remark}\label{rem-TG-size}
 This lower bound is pitifully weak; in practice $TG$ is enormously
 larger than $G$.
\end{remark}

\begin{lemma}\label{lem-condition-b}
 Suppose that $\alpha,\beta \colon G \to H$ are surjective
 homomorphisms in $\U$. Then there is an automorphism $\gamma$ of $TG$
 making the following diagram commute:
 \[
  \begin{tikzcd}
   TG \arrow[rr,"\gamma"] \arrow[d, "\epsilon"'] & & 
   TG \arrow[d,"\epsilon"] \\
   G \arrow[r, "\alpha"'] & H & \arrow[l, "\beta"] G.
  \end{tikzcd}
 \]
\end{lemma}
\begin{proof}
 Put $m=|G|/|H|=|\ker(\alpha)|=|\ker(\beta)|$. For each $h \in H$ we
 have $|\alpha^{-1}\{h\}|=m=|\beta^{-1}\{h\}|$, so we can choose a
 bijection $\alpha^{-1}\{ h\} \to \beta^{-1}\{h \}$. By combining
 these choices, we obtain a bijection $\sigma \colon G \to G$ such
 that $\beta\sigma=\alpha$.  This gives an automorphism
 $\gamma=T\sigma$ of $TG$.  We claim that
 $\beta \epsilon \gamma=\alpha \epsilon \colon TG \to H$.  It will
 suffices to check this on the generating set $G \subset TG$, and that
 reduces to the relation $\beta \sigma=\alpha$, which holds by
 construction.
\end{proof}

\begin{proof}[Proof of Theorem~\ref{thm-exact}]
 The claim is clear if $\U=\{1\}$.  Suppose instead that $\U$ contains
 a nontrivial group $G_0$.  Put $G_n=T^nG_0$, so we have a tower
 \[
  G_0 \xleftarrow{\epsilon}
  G_1 \xleftarrow{\epsilon}
  G_2 \xleftarrow{\epsilon} \dotsb.
 \]
 We claim that this is a colimit tower for $\U$. 
 Using Remark~\ref{rem-TG-size} we see that $|G_n|\to\infty$ as
 $n\to\infty$.  For fixed $H\in\U$ we can therefore choose a
 surjective function $G_n\to H$ for some $n$, and this will induce a
 surjective homomorphism $G_{n+1}\to H$ giving condition (a) of the 
 colimit tower. Condition (b) holds by Lemma~\ref{lem-condition-b} and 
 (c) follows from Lemma~\ref{lem-TX-lift}, so $\U$ is colimit-exact.
\end{proof}

\begin{proposition}\label{prop-upwards-colimit}
 Suppose that $\V\subseteq\U\subseteq\G$, that $\U$ is colimit-exact
 and that $\V$ is closed upwards in $\U$.  Then $\V$ is also
 colimit-exact.
\end{proposition}
\begin{proof}
 Let $i\colon\V\to\U$ be the inclusion, and let $c$ be the functor
 $\U\to 1$.  Note that $i_!\colon\A\V\to\A\U$ is just extension by
 zero, as proved in Lemma~\ref{lem-omnibus}, and so is exact.  We are
 given that the functor $L_{\U}=c_!$ is exact, so the composite
 $L_{\V}=(ci)_!=c_!i_!$ is exact as well.
\end{proof}

We conclude with a counterexample.

\begin{example}\label{ex-not-colimit-exact}
 The category $\G_{\leq 3}$ is not colimit-exact.
\end{example}
\begin{proof}
 The subcategory $\U'=\{1,C_2,C_3\}$ is a skeleton of $\G_{\leq 3}$,
 which makes it easy to calculate colimits.  Let $X<\one$ be given by
 $X(G)=0$ when $|G|=1$ and $X(G)=\one(G)=k$ when $|G|>1$.  We find
 that $LX=k^2$ but $L\one=k$, so $L$ does not send the monomorphism
 $X\to\one$ to a monomorphism, so $L$ is not exact.
\end{proof}

\section{Complete subcategories}
\label{sec-complete}
	
In this section we introduce a well-behaved type of subcategory and
present some examples.
	
\begin{definition}\label{def-expansive}
 Let $\U$ be a subcategory of $\G$.
 \begin{itemize}
  \item[(a)] For $T \in \G$, we denote by $\delta(T)$ the minimum 
  possible size of a generating set for $T$.
  \item[(b)] For $m \in \NN$, we put
   $\mathcal{R}_m = \lbrace T \in \U \mid \delta(T) \geq m \rbrace.$
  \item[(c)] We say that $\U$ is \emph{expansive} if for all $G\in\U$ 
  and $m\in \NN$ we have 
  $\mathcal{R}_m \downarrow G \not=\emptyset$.
  \item[(d)] Let $\U$ be expansive. For $X \in \A\U$ and $n>0$ we put
   \[
    \omega^{\U}_n(X)= \limsup_{m \to \infty}\lbrace \dim
    (X(T))/n^{\delta(T)} \; | \; T \in \mathcal{R}_m \rbrace \in [0,
    \infty].
   \]
   and
   \[
    \mathcal{W}(\U)_n= \lbrace X \in \A\U \; | \; \omega^{\U}_n(X)<
    \infty \rbrace.
   \]
   It is easy to see that if $\omega_n^{\U}(X)>0$ then
   $\omega^{\U}_m(X)= \infty$ for $m<n$.  Similarly, if
   $\omega^{\U}_n(X) < \infty$ then $\omega^{\U}_m(X)=0$ for $m>n$.
   Thus, there is at most one $n$ such that
   $0< \omega^{\U}_n(X)< \infty$.  If such an $n$ exists, we call it
   the \emph{order} of $X$.
 \end{itemize}
\end{definition}

\begin{remark}\label{rem-expansive}
 We will often drop the superscript and just write $\omega_n(X)$.
\end{remark}

Using the properties of the limsup we obtain the following result.

\begin{lemma}\label{lem-lim-sup}
 For any short exact sequence $X \to Y \to Z$ in $\A\U$ we have
 \[
  \max (\omega_n(X), \omega_n(Z)) \leq
  \omega_n(Y) \leq
  \omega_n(X) + \omega_n(Z).
 \]
 In particular, for any $X$ and $Z$ we have
 \[
  \max (\omega_n(X), \omega_n(Z)) \leq
  \omega_n(X \oplus Z) \leq
  \omega_n(X)+ \omega_n(Z). \qed
 \] 
\end{lemma}

\begin{corollary}\label{cor-growth}
 The category $\mathcal{W}(\U)_n$ is closed under finite direct sums,
 subobjects, quotients, extensions and retracts. It also contains
 $e_G$ for all $G \in \U_{\leq n}$.
\end{corollary}

\begin{proof}
 The closure properties easily follow from Lemma~\ref{lem-lim-sup}.
 For the second claim, note that if $A \subset T$ is a generating set
 for $T \in \U$, then the restriction map $\Hom(T,G) \to \Map(A,G)$ is
 injective, so $|\Hom(T,G)| \leq |G|^{|A|}$. It follows that
 \[
  |\U(T,G)|= |\Epi(T,G)|/ |\Inn(G)| \leq
  |\Hom(T,G)|/|\Inn(G)| \leq
  |G|^{\delta(T)}/ |\Inn(G)| =
  |G|^{\delta(T)-1}|ZG|.
 \]
 From this it is easy to see that $\omega_n(e_G) \leq |\Inn(G)|^{-1}$
 if $|G|=n$, and $\omega_n(e_G)=0$ if $|G|<n$.
\end{proof}

We are now ready to introduce an important family of subcategories.

\begin{definition}\label{def-complete}
 A subcategory $\U$ of $\G$ is \emph{complete} if the following
 conditions are satisfied:
 \begin{itemize}
  \item $\U$ is expansive, i.e., for all $G \in \U$ and $n>0$ there
   exists $T \in \U$ with $\delta(T) \geq n$ and
   $\U(T,G) \not= \emptyset$;
  \item For all $n >0$ and $G \in \U_n$, we have
   $0 < \omega^{\U}_n(e_G)< \infty$.  In other words, $e_G$ has order
   exactly $|G|$.
 \end{itemize}
\end{definition}

\begin{example}\label{ex-complete}
 Recall that we always have $\omega_n(e_G) \leq |\Inn(G)|^{-1}$ if
 $|G|=n$.
 \begin{itemize}
  \item The category $\C[p^\infty]$ of cyclic $p$-groups is not
   complete, as it is not expansive.  
  \item The category $\E[p]$ of elementary abelian $p$-groups is
   complete. Indeed we have 
   \[
     \omega_{p^n}(e_{C_p^n})= 
     \lim_{m \to \infty} \frac{|\Epi(C_p^m, C_p^n)| }{p^{nm}} = 
     \lim_{m \to \infty}\frac{(p^m-1)(p^m -p) \cdots (p^m-p^{n-1})}{p^{nm}}
      =1. 
   \] 
 \end{itemize}
\end{example}

Let us produce more examples of complete subcategories.

\begin{proposition}\label{prop-complete}
 If $\U\leq\G$ is nontrivial and submultiplicative, then it is
 complete. 
\end{proposition}

\begin{proof}
 As $\U$ is nontrivial and subgroup-closed, it must contain $C_p$ for
 some $p$.  Then for $G\in\U$ we have $G\times C_p^n\in\U$ with
 $\delta(G\times C_p^n)\geq n$, showing that $\U$ is expansive.  We now
 need to show that $\omega_{|G|}(e_G)>0$ for all $G \in \U$. Without
 loss of generality we can assume that $G \neq 1$.  For $X_m$ a set
 with $m$ elements, consider the group $T X_m \in \U$ as defined in
 Example~\ref{ex-free-group}.  By definition, there is a natural
 bijection $\Hom(T X_m, G) = \Hom(F X_m, G)\simeq G^m$ for all the
 groups $G \in \U_{\leq m}$. Since by~\cite{Pak}*{Theorem 1} we have
 \[
         \lim_{m \to \infty} |\Epi(FX_m, G)|/|G|^m =1
 \]
 we deduce that
 \[ \lim_{m \to \infty} |\Epi(T X_m,G)|/|G|^m = 1. \]
 It only remains to notice that $\delta(T X_m) \leq m$ so
 \[
  \omega_{|G|}(e_G) \geq 
  \lim_{m \to \infty}\frac{|\U(T X_m,G)|}{|G|^m} = 
  \lim_{m \to \infty}\frac{|\Epi(T X_m,G)|}{|\Inn(G)||G|^m} =
  \frac{1}{|\Inn(G)|} 													 >0. 
 \] 
\end{proof}
	
The completeness assumption give us information on the growth of the
indecomposable projectives.
	
\begin{lemma}\label{lem-omega-eGV}
 Let $\U$ be a complete subcategory of $\G$. For $G \in \U$ and
 $V$ an $\Out(G)$-representation, we have
 $0< \omega_{|G|}(e_{G,V}) < \infty$.
\end{lemma}

\begin{proof}
 We show that
 $\dim(e_{G,V}(T))= \dim(V) |\Out(G)|^{-1} |\dim(e_{G}(T))|$, and so
 the claim follows by completeness. It is easy to see that $\Out(G)$
 acts freely on $\U(T,G)$. Choose a subset $M \subset \U(T,G)$
 containing one representative of every orbit, so that
 $|M|= |\Out(G)|^{-1}|\U(T,G)|$.  We also see that $M$ is a basis for
 $e_G(T)$ as a module over the ring $R=k[\Out(G)]$, so
 \[ e_{G,V}(T)= V \otimes_{R} e_G(T) \simeq V^{|M|}. \]
 This gives
 \[ \dim(e_{G,V}(T)) =
    \dim(V) |M|=
     \dim(V) |\Out(G)|^{-1} \dim(e_{G}(T))
 \]
 as claimed.
\end{proof}

\begin{proposition}\label{prop-splitting}
 Let $\U$ be complete subcategory of $\G$. Then any monomorphism
 between projective objects of $\A\U$ is split.
\end{proposition}
\begin{proof}
 Let $u \colon P \to Q$ be a monomorphism between projective objects. 
 By Proposition~\ref{prop-filtration-split}, 
 we can write $P=\bigoplus_n P_n$ and $Q=\bigoplus_n Q_n$, where $P_n$
 and $Q_n$ are in the image of $(i_n)_!\colon\A\U_n\to\A\U$, so
 $\A\U(P_n,Q_m)=0$ when $n<m$.  We put
 $P_{\leq m}=\bigoplus_{k\leq m}P_k=L_{\leq m}P$, and similarly for
 $Q$.  It is then clear that $u$ restricts to give a monomorphism
 $u_{\leq m} \colon P_{\leq m} \to Q_{\leq m}$.  We will prove by
 induction on $m$ that $u_{\leq m}$ splits.  The claim is trivial if
 $m=0$. Let $m>0$ and let $s_{<m} \colon Q_{<m} \to P_{<m}$ be a
 splitting of $u_{<m} \colon P_{<m} \to Q_{<m}$.  Now let $K_m$ be the
 kernel of the map $u_{m} \colon P_{m} \to Q_m$.  As all monomorphisms
 in $\A\U_m$ are split, we see that $K_m$ is a retract of $P_m$.  As
 $u_{m}(K_m)=0$ and $u_{\leq m}$ is a monomorphism, we see that
 $u_{\leq m}$ induces a monomorphism from $K_m$ to $Q_{<m}$.  However,
 by completeness the order of $Q_{<m}$ is at most $m-1$, whereas if
 $K_m$ is nonzero, it must have order $m$.  It follows that $K_m$ must
 actually be zero, so $u_{m}$ is a monomorphism in $\A\U_m$, so there
 is a splitting $v \colon Q_m \to P_m$.  Let
 $s_{\leq m} \colon Q_{\leq m} \to P_{\leq m}$ be given by $s_{<m}$ on
 $Q_{<m}$, and by $v$ on $Q_m$. Then $s_{\leq m}u_{\leq m}$ is the
 identity of $P_{< m}$, and it is the identity modulo $P_{<m}$ on
 $P_m$, so it is an automorphism of $P_{\leq m}$. It follows that
 $(s_{\leq m}u_{\leq m})^{-1} \circ s_{\leq m}$ is a splitting of
 $u_{\leq m}$, as required. By construction, the sections $s_{\leq m}$
 assemble into a map $s \colon Q \to P$ satisfying $s\circ u = \id_P$,
 so $u$ splits.
\end{proof}

\section{Finiteness conditions}
\label{sec-finiteness-conditions}

We introduce various finiteness conditions on objects of $\A$ and
prove some implications amongst them.  We refer the reader to Remarks
\ref{rem-map1} and \ref{rem-map2} for a summary.

\begin{definition}\label{def-finiteness-conditions}
 Consider a subcategory $\U\leq\G$ and an object $X\in\A\U$.
 \begin{itemize}
  \item[(a)] We say that $X$ has \emph{finite type} if 
  $\dim(X(G)) < \infty$
   for all $G\in\U$.
  \item[(b)] We say that $X$ is \emph{finitely projective} if it can be
   expressed as the direct sum of a finite family of indecomposable
   projectives.
  \item[(c)] We say that $X$ is \emph{finitely generated} if there 
  is an epimorphism $P_0\to X$, for some finitely projective object 
  $P_0$ (or equivalently, for some object $P_0$ of the form
   $\bigoplus_{i=1}^ne_{G_i}$).
  \item[(d)] We say that $X$ is \emph{finitely presented} if there is a
   right exact sequence $P_1 \to P_0 \to X$, where $P_0$ and $P_1$ are
   finitely projective.
  \item[(e)] We say that $X$ is \emph{finitely resolved} if there is a
   resolution $P_* \to X$, where each $P_i$ is finitely projective.
  \item[(f)] We say that $X$ is \emph{perfect} if there is a resolution
   $P_* \to X$, where $P_i$ is finitely projective for all $i$,
   and $P_i=0$ for $i \gg 0$.
   \item[(g)] We say that $X$ has \emph{finite order} if there exists 
  $n >0$ such that $\omega_n(X)<\infty$.  
  (This is only meaningful in the
   case where $\U$ is expansive.)
 \end{itemize}
\end{definition}

\begin{lemma}\label{lem-preservation-finiteness}
 Let $i\colon \U \to \V$ be the inclusion of a subcategory.
 \begin{itemize}
  \item[(a)] The functor $i^*$ always preserves objects of finite type.
   If $\U$ is closed downwards, then $i^*$ preserves all finiteness
   conditions from Definition~\ref{def-finiteness-conditions} excluding
   that of finite order.
  \item[(b)] The functor $i_!$ always preserves finitely presented and
   finitely generated objects. If $\U$ is closed upwards (and
   therefore expansive), then $i_!$ preserves all finiteness conditions. 
  \item[(c)] If $\U$ is closed downwards, then $i_*$ preserves objects
   of finite type.
 \end{itemize}
\end{lemma}

\begin{proof}
 Clearly, $i^*$ preserves objects of finite type.  If $\U$ is closed
 downwards, then $i^*(e_{G})$ is either $e_G$ (if $G \in \U$) or $0$
 (if $G \not \in \U$). It follows that $i^*$ preserves (finitely)
 projective objects.  Since $i^*$ is also exact by
 Lemma~\ref{lem-omnibus}(e), it follows that $i^*$ preserves conditions 
 (a) to (f) in Definition~\ref{def-finiteness-conditions}.
  
 By Lemma~\ref{lem-omnibus}(e) and (i), the functor $i_!$ preserves
 colimits and preserves (finitely) projective objects. It follows that
 $i_!$ preserves finitely presented and finitely generated objects.
 If $\U$ is closed upwards, then $i_!$ is extension by zero by
 Lemma~\ref{lem-omnibus}(f) so it preserves objects of finite type and
 finite order (if $\U$ expansive).  It is also exact so it preserves
 all the other finiteness conditions.

 Finally, part (c) follows from Lemma~\ref{lem-omnibus}(g) as $i_*$ is
 extension by zero.
\end{proof}

\begin{remark}\label{rem-restriction-no-projectives}
  We have seen that the restriction functor $i^*$ preserves projectives if 
  $\U$ is closed downwards. This is no longer true if we relax the conditions on 
  $\U$ as the following example shows. Choose a group $G \in \G$, and consider 
  \[
  \U=\{ H \in \G \mid \U(H,G)\not =\emptyset, \; \U(G,H)=\emptyset\}=
  \{H \in \G_{\geq G} \mid H\not \simeq G \}.
  \] 
  Note that $\U$ 
  is complete as it is closed upwards in $\G$. Let $i \colon \U \to G$ denote 
  the inclusion functor. We claim that $i^*e_G$ is 
  not projective in $\U$. 
  Suppose that $i^*e_G$ was projective, so we could write
  $i^*e_G=\bigoplus_i e_{H_i, V_i}$ for some groups $H_i \in \U$. Note that 
  we must have $|H_i|> |G|$ for all $i$. If we calculate the order of these 
  objects we see that $\omega_{|G|}^\U(i^*e_G)=\omega_{|G|}^\G(e_G)$ and 
  so $i^*e_G$ has order $|G|$ by completeness of $\G$. 
  On the other hand, for
  $n=\max_i |H_i|$ we have
  $0<\omega_n^\U(\bigoplus_i e_{H_i,V_i})<\infty$ so this has order $n$. 
  We have found a contradiction since $n>|G|$ so $i^*e_G$ cannot be projective.
\end{remark}

\begin{lemma}\label{lem-restriction-preserve-fg}
 Consider the inclusion $i_n \colon \F[p^n]\to \Z[p^\infty]$ for some $n\geq 1$. 
 Then the restriction functor $i_n^*\colon \A\Z[p^\infty]\to \A\F[p^n]$ preserves 
 finitely generated objects. 
\end{lemma}

\begin{proof}
 Consider a finitely generated object $X\in\A\Z[p^\infty]$ and choose an 
 epimorphism $\varphi \colon \bigoplus_{i=1}^s e_{A_i}\to X$.  
 Since $i_n^*$ preserves epimorphisms by 
 Lemma~\ref{lem-omnibus}(e), it will suffices to prove the following claim: 
 if $A\in\Z[p^\infty]$, then $i_n^*e_{A}\in \A\F[p^n]$ is finitely generated.   
 Let $F\in \F[p^n]$ be minimal such that $\Z[p^\infty](F,A)\not =\emptyset.$ 
 A choice of an epimorphism $\varphi\colon F \to A$, then gives a morphism 
 $e_\varphi \colon e_F \to i_n^*e_A$ and we claim this is an 
 epimorphism. In other words, we ought to show that for any epimorphism 
 $\psi \colon F' \to A$ with $F' \in \F[p^n]$, there exists an 
 epimorphism $\zeta \colon F'\to F$ making the following diagram commute:
 \[
 \begin{tikzcd}
  F' \arrow[dr, "\psi"'] \arrow[r, "\zeta"]& F \arrow[d, "\varphi"]\\
     & A. \\
 \end{tikzcd}
 \]   
 This is the content of Lemma~\ref{lem-dotted-arrow}.
\end{proof}

It is useful to have a criterion to detect objects which are not
finitely generated. 
Recall the notion of support from Definition~\ref{def-base-support}.

\begin{lemma}\label{lem-infinitess-criterion}
 If $X$ is finitely generated, then $\min(\supp(X))$ is finite.
\end{lemma} 
\begin{proof}
 If $X$ is finitely generated, we can find an epimorphism
 $\varphi\colon\bigoplus_{i=1}^r e_{G_i} \to X$.  Without loss of
 generality we can assume that each component $ e_{G_i}\to X$ is
 nonzero so that $X(G_i)\neq 0$ for all $i$.  We claim that
 $\min(\supp(X))\subseteq \{[G_1],\ldots,[G_r]\}$ which will prove
 the lemma.  If $[H]\in\min(\supp(X))$, then $X(H)\neq 0$, so
 $\bigoplus_ie_{G_i}(H)\neq 0$, so we can choose an index $i$ with
 $e_{G_i}(H)\neq 0$, so we can choose a morphism
 $\alpha\colon H\to G_i$ in $\U$.  Now both $[H]$ and $[G_i]$ lie in
 $\supp(X)$, and $[H]$ is assumed to be minimal, so $\alpha$ must be
 an isomorphism, so $[H]=[G_i]$ as required.
\end{proof}

\begin{proposition}\label{prop-perfect-is-projective}
 Let $\U$ be a complete subcategory of $\G$. Then any object of $\A\U$
 with a finite projective resolution is projective. In particular any
 perfect object is finitely projective.
\end{proposition}
\begin{proof}
 Let $P_*\to X$ be a projective resolution and suppose that $P_i=0$
 for all $i>n$.  If $n>0$ it follows that the differential
 $d_n \colon P_n \to P_{n-1}$ must be a monomorphism, so
 Proposition~\ref{prop-splitting} tells us that it is split.  Now let
 $Q_*$ be the same as $P_*$ except that $Q_n=0$ and
 $Q_{n-1}=\cok(d_n)$. We find that this is again a projective
 resolution of $X$.  By repeating this construction, we eventually
 obtain a projective resolution of length one, showing that $X$ itself
 is projective.
\end{proof}

\begin{remark}\label{rem-perfect-is-not-projective}
 The Proposition above is not true if we drop the completeness
 condition.  For example let $\C[p^\infty]$ be the subcategory of 
 cyclic $p$-groups. Then there is a short exact
 sequence $0 \to c_{C_{p^2}} \to c_{C_p} \to t_{C_p, k} \to 0$ which
 shows that $t_{C_p,k}$ is perfect.  On the other hand, we have
 $t_{C_p,k}(C_{p^r})=0$ for all $r>1$, and it follows easily from this
 that $t_{C_p,k}$ is not projective.
\end{remark}

\begin{proposition}\label{prop-fin-proj-fin-ord}
 Let $\U$ be a complete subcategory of $\G$. Then any finitely
 projective object in $\A\U$ has finite order.
\end{proposition}
\begin{proof}
 The zero object has by definition finite order. For $r\geq 1$, we
 have
 \[
  0 <\omega_n\left(\bigoplus_{i=1}^r e_{G_i, S_i}\right) < \infty
  \qquad \text{if} \;\; n =\max_{i}(|G_i|)
 \]	
 by Lemma~\ref{lem-omega-eGV}.
\end{proof}

\begin{lemma}\label{lem-truncated-objects-are-perfect}
 Let $\U\leq \G$ be finite (meaning that it has only finitely many
 isomorphism classes). Then the following are equivalent for an object
 $X \in\A\U$:
 \begin{itemize}
  \item[(a)] $X$ has finite type;
  \item[(b)] $X$ is finitely generated;
  \item[(c)] $X$ is perfect. 
 \end{itemize}
\end{lemma}
\begin{proof}
 Recall the explicit projective resolution from
 Construction~\ref{con-proj-inj-resolutions}.  We have
 $P_0=l_!l^*(X)=\bigoplus_{G \in \U'}e_{G, X(G)}$. If $X$ has finite
 type, then $P_0$ is a finitely generated projective object since $\U$
 is finite.  This gives (a) $\Rightarrow$ (b).  Clearly, (b)
 $\Rightarrow$ (a) so (a) is equivalent to (b).
  
 Now suppose that $X$ is finitely generated (and hence of finite type)
 and consider the canonical projective resolution $P_\bullet \to X$.
 The explicit formula for $P_i$ tells us that $P_i$ has finite type,
 and it follows from the previous paragraph that $P_i$ is finitely
 generated too.  To prove (b) $\Rightarrow$ (c), we need to show that
 $P_n=0$ for $n\gg 0$.  Recall from
 Remark~\ref{rem-base-proj-resolution} that
 $\base(P_n)\geq \base(X)+n$. Now note that any object in $\A\U$ with
 sufficiently large base is zero as $\U$ is finite.  Hence $P_n =0$
 for $n\gg 0$ as required.  The final implication (c) $\Rightarrow$
 (a) is clear.
\end{proof}

\begin{remark}\label{rem-map1}
 So far we have the following implications:
 \[
  \begin{tikzcd}[column sep=huge]
   \text{finitely resolved}  \arrow[r, Rightarrow] &
   \text{finitely presented} \arrow[r, Rightarrow] &
   \text{finitely generated} \arrow[r, Rightarrow] &
   \text{finite type} \\
   \text{perfect} 
    \arrow[u, Rightarrow] 
    \arrow[r, Leftrightarrow, "completeness"] &
   \text{finitely projective}
    \arrow[r, Rightarrow,"completeness"] & 
   \text{finite order.} &
  \end{tikzcd}
 \]
\end{remark}

\section{Torsion and torsion-free objects}
\label{sec-torsion}

In this section we introduce the notions of torsion, absolutely
torsion and torsion-free object. We study their formal properties and
give some examples.
	
\begin{definition}\label{def-torsion-and-torsion-free}
 Consider an object $X$ of $\A\U$.
 \begin{itemize}
  \item We say that $x \in X(G)$ is \emph{torsion} if there exists
   $H \in \U$ and $f \in \U(H,G)$ such that $f^*(x)=0$.
  \item We say that $x \in X(G) $ is \emph{absolutely torsion} if there
   exists $m \in \NN$ such that for all $f \in \U(H,G)$ with
   $|H|\geq m$ we have $f^*(x)=0$.
  \item We say that $X$ is \emph{torsion} (resp., \emph{absolutely
    torsion}) if it consists entirely of torsion (resp., absolutely
   torsion) elements.
  \item We say that $X$ is \emph{torsion-free} if it does not contain
   any nonzero torsion element. Equivalently, $X$ is torsion-free if
   and only if all the maps $\alpha^* \colon X(G) \to X(H)$ are
   injective.
  \item We write $\tors(X)(G)$ for the subset of torsion elements in
   $X(G)$.
 \end{itemize}
\end{definition}

\begin{hypothesis}\label{hyp-colimit-tower}
  Throughout we will assume that $\U \leq \G$ has a colimit tower as in 
  Definition~\ref{def-colimit-tower}. This is not a very restrictive 
  assumption as we have shown in Section~\ref{sec-exact-colimit} that 
  most natural examples satisfy this. 
\end{hypothesis}

\begin{lemma}\label{lem-torsion-G_n}
 For an element $x \in X(H)$, the following are equivalent:
 \begin{itemize}
  \item[(a)] $x$ is torsion.
  \item[(b)] There exists $\alpha\in\U(G_n,H)$ for some $n$ such that
   $\alpha^*(x)=0$ in $X(G_n)$.
  \item[(c)] There exists $n_0$ such that for all $n \geq n_0$ and all
   $\alpha\in\U(G_n,H)$ we have $\alpha^*(x)=0$ in $X(G_n)$.
 \end{itemize}
\end{lemma}

\begin{proof}
 By condition (a) of the colimit tower, we
 see that $\U(G_n,H)\not=\emptyset$ for large $n$. It follows that
 $(c)\Rightarrow(b)\Rightarrow (a)$. Now suppose that $(a)$ holds, so
 we can choose $\beta\in \U(G,H)$ for some $G$ with
 $\beta^*(x)=0$. Now let $n_0$ be least such that
 $\U(G_{n_0-1},G)\not=\emptyset$. 
 Suppose that $n \geq n_0$, so $\U(G_{n-1},G)\not=\emptyset$. 
 If $\alpha \in \U(G_n, H)$, then condition (c) of the colimit tower gives us a morphism $\gamma \in \U(G_n,G)$ with $\alpha=\beta\circ\gamma$, and it
 follows that $\alpha^*(x)=0$. Thus, part $(c)$ holds.
\end{proof}

Recall the colimit functor $L \colon \A\U \to \Vect_k$ from 
Section~\ref{sec-exact-colimit}.  

\begin{lemma}\label{lem-torsion-zero-in-L}
 Consider an object $X\in\A\U$ and an element $x\in X(G)$. Then $x$ is torsion 
 if and only if the element $1_G \otimes x \in (e_G \otimes X)(G)$ maps to 
 zero in $L(e_G \otimes X)$. 
\end{lemma}

\begin{proof}
 Suppose that $x$ is torsion, so we can choose $\alpha\colon G'\to G$ with 
 $\alpha^*(x)=0$. This means that $\alpha^*(1_G \otimes x)=\alpha \otimes 
 \alpha^*(x)=0$. The description $L(e_G \otimes X)=\colim_{H} (e_G \otimes X)(H)$ 
 shows that $1_G \otimes x$ is sent to zero in $L(e_G \otimes X)$. 
 
 For the converse, suppose we have an integer $n$ and a morphism 
 $\alpha \in \U(G_n,G)$. 
 Put $\Gamma= \Out(G_n)$ and $\Delta=\{\delta \mid \alpha \delta =\alpha \}$. 
 Define a map
 \[
 \xi \colon (e_G \otimes X)(G_n) \to X(G_n), \quad 
 \xi(\pi \otimes m)=\sum\{\gamma^* m\mid\gamma\in\Gamma,\;\pi\gamma=\alpha\}
 \] 
 
 One checks that $\xi \theta^* =\xi$ for all $\theta \in \Gamma$, so there is 
 an induced map from the coinvariants 
 $\overline{\xi}\colon (e_G \otimes X)(G_n)_\Gamma \to X(G_n)$. 
 One also checks that $\xi(\alpha^*(1_G \otimes x))= |\Delta|\alpha^*(x)$ for all 
 $x\in X(G)$. The condition that $\alpha^*(1_G \otimes x)$ maps to zero in 
 $L(e_G \otimes X)$ is equivalent to $\alpha^*(1_G \otimes x)$ mapping to 
 zero in $(e_G \otimes X)(G_n)_\Gamma$ for some $n\geq 0$ by 
 Proposition~\ref{prop-colimit-tower}. It then follows 
 that $0=\xi(0)=\xi(\alpha^*(1_G \otimes x))= |\Delta|\alpha^*(x)$ so $x$ is torsion.
\end{proof}

\begin{corollary}\label{cor-not-torsion-map}
 If $x\in X(G)$ is not torsion, then there is a morphism 
 $u\colon e_G \otimes X \to \mathbbm{1}$ such that $u(1_G\otimes x) =1$.
\end{corollary}

\begin{proof}
 As the image of $1_G \otimes x$ is nonzero in $L(e_G \otimes X)$, we can choose 
 a $k$-linear map $u_0 \colon L(e_G \otimes X) \to k$ sending this image to $1$. 
 Then $u_0$ is adjoint to a morphism $u\colon e_G \otimes X \to \mathbbm{1}$ 
 as claimed. 
\end{proof}

\begin{lemma}\label{lem-torsion-subobject}
 For any finite dimensional subspace $V \leq \tors(X)(G)$, there is a 
 map $\alpha \colon H \to G$ in $\U$ with $\alpha^*(V)=0$. Moreover, 
 $\tors(X)$ defines a subobject of $X$ in $\A\U$ which is the largest 
 torsion subobject of $X$. The assignment $\tors$ is functorial in $X$ so we
 have a functor $\tors \colon \A\U \to \A\U$.  
\end{lemma}

\begin{proof}
 Suppose we have torsion elements $x_1,\ldots, x_s \in \tors(X)(G)$.
 By Lemma~\ref{lem-torsion-G_n}(c), we can choose $n$ large so that 
 $\alpha^*(x_i)=0$ for all $\alpha \in \U(G_n, G)$ and all $1\leq i\leq s$. 
 Thus, if $V$ is the span of $\{x_1, \ldots, x_n \}$, we have 
 $\alpha^*(V)=0$, so $V \leq \tors(X)(G)$. 
 This proves in particular that $\tors(X)(G)$ is a vector subspace of $X(G)$.
  
 Now suppose we have $\alpha^*(x)=0$, and we also have another
 morphism $\beta \colon G' \to G$ in $\U$. By condition (b) of the colimit 
 tower, we can fill the dotted arrow in the diagram
 \[
  \begin{tikzcd}
    & G' 
  \arrow[d, two heads, "\beta"] \\
  H \arrow[r, two heads, "\alpha"] \arrow[ru, dotted, "\gamma"]& G.
  \end{tikzcd}
 \]
 We have $\gamma^*\beta^*(x)=\alpha^*(x)=0$, so $\beta^*(x)$ is
 a torsion element. This shows that $\tors(X)$ is a subobject of
 $X$. All remaining claims are now clear.
\end{proof}

The following example illustrates the fact that many things can go wrong 
if we do not assume that $\U$ has a colimit tower.

\begin{example}\label{ex-torsion-not-v-space}
 Consider the following object of $\A\G_{\leq 3}$
 \[
 X=( k_x \xleftarrow{\mathrm{pr}_x} 
 k_x \oplus k_y \xrightarrow{\mathrm{pr}_y} 
 k_y).
 \]
 Note that $x,y \in X(1)$ are torsion since 
 $\mathrm{pr}_x(y)=0=\mathrm{pr}_y(x)$. 
 On the other hand, $x+y \in X(1)$ is not torsion since 
 $\mathrm{pr}_x(x+y)=x$ and $\mathrm{pr}_y(x+y)=y$. 
 In particular $\tors(X)(1)$ is not a vector subspace of $X(1)$.
\end{example}

\begin{remark}\label{rem-torsion-free-sum}
 The sum of two torsion-free subobjects need not be torsion-free. To
 see this, consider a torsion-free object $Y$, a nonzero torsion
 object $Z$ and an epimorphism $f \colon Y \to Z$. In $Y \oplus Z$ we
 have a copy of $Y$, and the graph of $f$ is another subobject
 $Y' \leq Y \oplus Z$ which is also isomorphic to $Y$ and so is
 torsion-free. However, $Y+Y'$ is all $Y \oplus Z$ and so is not
 torsion-free.
\end{remark}

\begin{lemma}\label{lem-tor-quot-tor-free}
 For any object $X$
 of $\A\U$, the quotient $X/\tors(X)$ is torsion-free.
\end{lemma}
\begin{proof}
 Consider an element $\overline{x} \in (X/\tors(X))(G)$, so
 $\overline{x}$ is represented by some element $x \in X(G)$. If
 $\overline{x}$ is a torsion element, then we have
 $\alpha^*(\overline{x})=0$ for some $\alpha \in \U(H,G)$, or
 equivalently $\alpha^*(x) \in \tors(X)(H)$. This means that there
 exists $\beta \in \U(K,H)$ with
 $(\alpha\beta)^*(x)=\beta^*(\alpha^*(x))=0$. Thus $x$ is a torsion
 element and $\overline{x}=0$ as required.
\end{proof}
	
Recall the objects $e_{G,V}$ and $t_{G,V}$ from
Definition~\ref{def-main-objects}.

\begin{lemma}\label{lem-gens-torsion-free}
 For all $G \in \U$, we have that $e_{G,V}$ is torsion-free and
 $t_{G,V}$ is absolutely torsion. Thus, any projective object
 is torsion-free.
\end{lemma}
\begin{proof}
 It is clear that $t_{G,V}$ is absolutely torsion as $t_{G,V}(K)$ is 
 zero as soon as $|K|> |G|$. It is enough to show that $e_G$ 
 is torsion-free as $e_{G,V}$ is a retract of a direct sum of $e_G$'s. 
 Thus, we need to show that for any epimorphism 
 $\varphi \colon H \to K$ the linear map 
 $\varphi^* \colon k[\U(K,G)]\to k[\U(H,G)]$ is injective. This is
 equivalent to proving that the map
 $\varphi^* \colon \U(K,G) \to \U(H,G)$ is injective, or in other
 words that $\varphi$ is an epimorphism in the category $\U$. This is
 the content of Lemma~\ref{lem-epimorphism}.
\end{proof}

We write $\A\U_t$ and $\A\U_f$ for the subcategories of torsion and
torsion-free objects.

\begin{lemma}\leavevmode\label{lem-torsion-torsion-free}
  \begin{itemize}
    \item[(a)] For an object $X \in \A\U$, we have $X \in \A\U_t$ if and only if 
    $\A\U(X,Y)=0$ for all $Y \in \A\U_f$.
    \item[(b)] For an object $Y \in \A\U$, we have $Y \in \A\U_f$ if and only if 
    $\A\U(X,Y)=0$ for all $X \in \A\U_t$. 
  \end{itemize}
\end{lemma}
\begin{proof}
 If $f \colon X \to Y$ then $f(\tors(X)) \leq \tors(Y)$. If
 $X \in \A\U_t$ and $Y \in\A\U_f$ then $\tors(X)=X$ and $\tors(Y)=0$ so
 this becomes $f(X)=0$ and $f=0$.  Thus, for $X \in \A\U_t$ and
 $Y \in \A\U_f$ we have $\A(X,Y)=0$.
 
 Now suppose that $X$ is such that $\A\U(X,Y)=0$ for all $Y \in \A\U_t$. In particular, 
 this means that the projection $X \to X/\tors(X) $ is zero, so $\tors(X)=X$ and 
 $X\in\A\U_t$.
 
 Suppose instead that $Y$ is such that $\A\U(X,Y)=0$ for all
 $X \in \A\U_t$. In particular, this means that the inclusion
 $\tors(Y) \to Y$ is zero, so $\tors(Y)=0$ and $Y \in \A\U_f$.
\end{proof}

\begin{lemma}\label{lem-tensor-and-homs-torsion}
 Consider objects $X \in \A\U_t$ and $Y \in \A\U_f$. Then for all
 $Z \in \A\U$, we have
 \begin{itemize}
  \item[(a)] $X \otimes Z$ is torsion;
  \item[(b)] $\uHom(X\otimes Z,Y)=0$.
 \end{itemize}
\end{lemma}

\begin{proof}
 Any element of $(X \otimes Z)(G)$ can be written as a finite linear combination of 
 homogeneous terms $x_i \otimes z_i$. For each of such term we can find 
 $\alpha_i \colon H_i \to G$ such that
 $\alpha_i^*(x_i)=0$.  Thus we have
 $\alpha^*(x \otimes z)= \alpha^*(x) \otimes \alpha^*(z)=0$. 
 As a finite linear combination of torsion elements is again torsion by 
 Lemma~\ref{lem-torsion-subobject}, we deduce that $X\otimes Z$ is torsion. 
 For all $G \in \U$, we have
 \[
  \uHom(X\otimes Z,Y)(G)=\A\U(e_G \otimes X\otimes Z, Y)=0
 \]
 by part (a) and Lemma~\ref{lem-torsion-torsion-free}.
\end{proof}
	
\begin{lemma}\label{lem-torsion-cat-localizing}
 The subcategory
 $\A\U_t$ is localizing that is, it is closed under arbitrary sums,
 subobjects, extensions and quotients.
\end{lemma}

\begin{proof}
 Consider an exact sequence $X \xrightarrow{i} Y \xrightarrow{p} Z$ in
 which $X$ and $Z$ are torsion objects. Consider an element
 $y \in Y(G)$. As $Z$ is a torsion object, we can choose
 $\alpha\colon H \to G$ with $\alpha^*(p(y))=0$. This means that
 $p(\alpha^*(y))=0$, so $\alpha^*(y)=i(x)$ for some $x \in X(H)$. As
 $X$ is a torsion object, we can choose $\beta \colon K \to H$ with
 $\beta^*(x)=0$, and it follows that
 \[
  (\alpha \beta)^*(y)=\beta^*i(x)=i(\beta^*(x))=i(0)=0.
 \]
 This shows that $Y$ is also a torsion object so $\A\U_t$ is closed
 under extensions.
  
 Now let $X$ be a sum of torsion objects $X_i$ and consider an element
 $x \in X(G)$.  By definition, we can write $x=x_{i_1}+\ldots+x_{i_n}$
 for torsion elements $x_{i_k} \in X_{i_k}(G)$. 
 By Lemma~\ref{lem-torsion-G_n}(c), we can find large $n$ such that 
 $\alpha^*(x_{i_k})=0$ for all $i_k$ and $\alpha \in \U(G_n, H_{i_k})$. 
 Thus, we have
 \[
  \alpha^*(x)=\alpha_{i_1}^*(x_{i_1})+ \ldots\alpha^*_{i_n}(x_{i_n})=0
 \]
 so $x$ is torsion. This shows that $\A\U_t$ is closed under arbitrary
 sums.  All the other claims are clear.
\end{proof}
  
\begin{lemma}\label{lem-subobject-torsion-free}
 The subcategory $\A\U_f$ is closed 
 under subobjects, extensions, arbitrary sums and arbitrary products.
\end{lemma}
\begin{proof}
 From Lemma~\ref{lem-torsion-torsion-free}(b) it is clear that
 $\A\U_f$ is closed under products and subobjects. As products and
 sums are computed objectwise, we see that every sum injects in the
 corresponding product, so $\A\U_f$ is also closed under coproducts.
 Now consider a short exact sequence as follows, in which $X$ and $Z$
 are torsion-free
 \begin{center}
  \begin{tikzcd}
   X \arrow[r,rightarrowtail,"f"] &
   Y \arrow[r,two heads,"g"] &
   Z.
  \end{tikzcd}
  \end{center}
  If $T$ is a torsion object, this gives a left exact sequence
  \begin{center}
  \begin{tikzcd}
   0=\A\U(T,X) \arrow[r,rightarrowtail,"f_*"] &
   \A\U(T,Y) \arrow[r,"g_*"] &
   \A\U(T,Z)=0,
  \end{tikzcd}
 \end{center}
 proving that $\A\U(T,Y)=0$. It follows that $\A\U_f$ is also closed
 under extensions. 
\end{proof}

We will now give another characterization of torsion-free objects
under some mild conditions on $\U$. But first we need a little bit of
preparation.

\begin{construction}\label{con-map-into-SX}
 Recall the inclusion functor $l \colon \U^\times \to \U$ and the
 functor $l_!l^*$ from Construction~\ref{con-proj-inj-resolutions}.
 For any object $X\in\A\U$, we set $SX=D(l_!l^*(DX))$ which is
 injective by Proposition~\ref{prop-one-injective}.  Adjoint to the
 counit map $l_!l^*(DX) \to DX$, we have a map
 $X \otimes l_!l^*(DX)\to \mathbbm{1}$ which is itself adjoint to a
 map $\xi \colon X \to SX$. If we fix a skeleton $\U'$ for $\U$, we
 have the explicit formula
 \[ SX=\prod_{G \in \U'} De_{G, DX(G)}, \]
 and the map $\xi$ has $G$-component which is adjoint to the evaluation map 
 $\mathrm{ev}\colon X \otimes e_{G, DX(G)} \to \mathbbm{1}$. 
 More explicitly, we have
 \[
  \mathrm{ev} \colon X(H) \otimes e_G(H) \otimes_{\Out(G)}
  \A\U(e_G\otimes X, \mathbbm{1}) 
  \to k, \quad x \otimes [\alpha] \otimes f\mapsto f([\alpha] \otimes x)
 \]
 for all $H\in\U$.
\end{construction}

\begin{proposition}\label{prop-torsion-free-embed-proj}
 Suppose that $\U$ is a multiplicative global family and consider an object 
 $X\in\A\U$.  Then $SX$ is
 projective and we have an exact sequence
 \[
 0\to \tors(X) \to X \xrightarrow{\xi} SX.
 \]
 In particular, $X$ is torsion-free if and only if it can 
 be embedded in a projective object. 
\end{proposition}

\begin{proof}
 We first show that $SX$ is projective. By
 Proposition~\ref{prop-product-projectives}, it is enough to show that
 $De_{G, DX(G)}$ is projective.  Note that $\Out(G)$ acts freely on
 $DX(G)=\A\U(e_G \otimes X, \mathbbm{1})$.  Choose
 $u_1, \ldots, u_r \in DX(G)$ containing precisely one element from
 each $\Out(G)$-orbit so that $DX(G)=\bigoplus_{i=1}^r k[\Out(G)]$ and
 hence $e_{G, DX(G)}=\bigoplus_{i=1}^r e_G$. Therefore we have reduced
 the problem to showing that $De_G$ is projective. This now follows from 
 Theorem~\ref{thm-hom-fg-projective}.

 If $SX$ is projective, then it is also torsion-free by 
 Lemma~\ref{lem-gens-torsion-free}, so we get $\tors(X) \subseteq \ker(\xi)$. 
 If $x \in X(G)$ is not torsion, then by Corollary~\ref{cor-not-torsion-map} we 
 can find $u\in DX(G)$ such that $u(1_G\otimes x)\not =0$. In particular, 
 we have $\mathrm{ev}( 1_G\otimes x \otimes u)\not =0$ and hence 
 $\ker(\xi)\subseteq\tors(X)$. This shows that the sequence is exact. 
\end{proof}

Let $X$ be a finitely generated torsion object. It is tempting to
conclude that $X(G)$ should be zero when $G$ is sufficiently large, in
some sense. However, the following example shows that this is not
correct.
  
\begin{example}\label{ex-non-split}
 Let $\theta \colon P \to Q$ be a non-split epimorphism between groups
 in $\U$. This gives a map $\theta_* \colon e_P \to e_Q$, and we
 define $X$ to be the cokernel (so $X$ is finitely presented). The
 obvious generator $x \in X(Q)$ satisfies $\theta^*(x)=0$ by
 construction, so $x$ is torsion. As $x$ generates $X$, it follows
 that $X$ is a torsion object. Note that $X(G)$ is the quotient of
 $k[\U(G,Q)]$ in which we kill every basis element $[\alpha]$ for
 which the homomorphism $\alpha\colon G \to Q$ can be lifted to
 $P$. Note that no split epimorphism $\alpha \colon G \to Q$ can be
 lifted to $P$, because that would give rise to a splitting of
 $\theta$. In particular, if $H$ admits a split epimorphism to $Q$,
 then $X(H) \not = 0$. Thus, we have $X(H \times Q)\not = 0$ for all
 $H \in \U$.
\end{example}

It is true, however, that if $X$ is a finitely generated torsion
object, and $G$ is both sufficiently large and sufficiently free, then
$X(G)=0$. We now proceed to make a precise version of this statement.

\begin{definition}\label{def-G-star-null}
 We say that an object $X\in \A\U$ is $G_*$-\emph{null} if $X(G_n)=0$ for large $n$.
\end{definition}

\begin{lemma}\label{lem-G-star-torsion}
 If $X$ is $G_*$-null, then it is torsion. The converse holds if $X$
 is finitely generated.
\end{lemma}
\begin{proof}
 First suppose that $X$ is $G_*$-null. Consider an element
 $x \in X(H)$. Choose $n$ large enough that $X(G_n)=0$ and
 $\U(G_n, H) \not =\emptyset$. Then for $\alpha\in \U(G_n,H)$ we have
 $\alpha^*(X)=0$, as required.
  
 Conversely, suppose that $X$ is finitely generated, with generators
 $x_i \in X(H_i)$ for $i=1,\ldots,d$ say. By
 Lemma~\ref{lem-torsion-G_n} we can choose $n_i$ such that
 $\alpha^*(x_i)=0$ in $X(G_m)$ for all $m \geq n_i$, and all
 $\alpha\in \U(G_m,H_i)$.  Put $n=\max(n_1, \ldots,n_d)$; then we find
 that $X(G_n)=0$ for all $n \geq m$.
\end{proof}
  
  We finish this section by giving some examples of torsion objects.

\begin{example}\label{ex-relations-no-torsion}
 Let $G$ be cyclic of order $p$, so $\Aut(G)$ is cyclic of order
 $p-1$, and let $\psi\in\Aut(G)$ be a generator.  Let $X$ be the
 cokernel of $\psi_*-1\colon e_G\to e_G$.  By definition $X(H)$ is the
 quotient of $k[\U(H,G)]$ by the subspace generated by the elements 
 $[\psi\alpha]-[\alpha]$ for all $\alpha$. As $\psi$ is a generator of $\Aut(G)$, 
 we can identify $X$ with $c_{G}$ from Definition~\ref{def-main-objects}.  
 In particular, $X$ is projective and torsion-free.  
 This illustrates the fact that we can introduce quite a lot of relations 
 without creating torsion.
\end{example} 

\begin{example}\label{ex-torsion-misc-a}
 Take $\U=\Z[p^\infty]$ and let 
 $C$ be cyclic of order $p$. 
 Let $\lambda,\rho\colon C^2\to C$ be the two projections, and let $X$ be
 the cokernel of $\lambda_*-\rho_*\colon e_{C^2}\to e_C$. 
 This means that $X(G)=k[T(G)]$, where $TG$ is the coequaliser of the maps
 $\lambda_*,\rho_*\colon k[\U(G,C^2)]\to k[\U(G,C)]$. Let $Q(G)$ be the Frattini
 quotient of $G$, so $Q(G)\simeq C^{d(G)}$ for some $d(G)\geq 0$.  If
 $d(G)=0$ then $G=1$ and $T(G)=\emptyset$ and $X(G)=0$.  If $d(G)=1$
 then $G$ is cyclic and $\U(G,C^2)=\emptyset$ so
 $T(G)=\U(G,C)=\U(Q(G),C)$ (which is a set of size $p-1$) so
 $X(G)\simeq k^{p-1}$.  Now suppose that $d(G)\geq 2$.  If $\alpha$ and
 $\beta$ are epimorphisms from $G$ to $C$ with different kernels then
 the combined map $\phi=(\alpha,\beta)\colon G\to C^2$ is again surjective with
 $\lambda\phi=\alpha$ and $\rho\phi=\beta$ so $[\alpha]=[\beta]$ in $T(G)$. 
 Even if $\alpha$ and $\beta$ have the same kernel, we can choose a third
 epimorphism $\gamma\colon G\to C$ with different kernel (because of the fact
 that $d(G)\geq 2$); we then have $[\alpha]=[\gamma]=[\beta]$.  From this we
 see that $T(G)$ is a singleton and so $X(G)=k$. To summarize
 \[
 X(G)=k[T(G)]=\begin{cases}
 0                     & \text{if}\;\; d(G)=0 \\
 \U(G,C)\simeq k^{p-1} & \text{if}\;\; d(G)=1 \\
 k                     & \text{if}\;\; d(G)\geq 2.
 \end{cases}
 \]
 From our discussion we also see that 
 \begin{align*}
  \tors(X)(G) &\simeq
   \begin{cases}
    k^{p-2} & \text{ if $G$ is nontrivial and cyclic } \\
    0 & \text{ otherwise }
   \end{cases} \\
  (X/\tors(X))(G) &\simeq
   \begin{cases}
    0 & \text{ if } G = 1 \\
    k & \text{ if } G \neq 1.
   \end{cases}
 \end{align*}
\end{example}

\begin{example}\label{ex-torsion-misc-b}
 Take $\U=\Z[2^\infty]$.  There are then three
 morphisms $\lambda,\rho,\sigma\in\U(C^2,C)$, and we define $X$ to be
 the cokernel of $\lambda_*+\rho_*+\sigma_*\colon e_{C^2}\to e_C$. 
 We claim that $X$ is a torsion object. To see this, we put 
 $u=\lambda+\rho+\sigma\in e_C(C^2)$ so that
 $X(G)$ is the quotient of $k[\U(G,C)]$ by all elements of the form
 $\phi^*(r)$ as $\phi$ runs over $\U(G,C^2)$. If $d(G)=1$ then
 $\U(G,C)$ is a singleton and $\U(G,C^2)=\emptyset$ and $X(G)=k$.  If
 $d(G)=2$ then $k[\U(G,C)]$ has three elements, say
 $\alpha,\beta,\gamma$, and
 \[ X(G) = k\{\alpha,\beta,\gamma\} /(\alpha+\beta+\gamma)\simeq k^2. \]
 Now consider $X(C^3)$. This is spanned by the seven nonzero homomorphisms 
 $C^3\to C$. 
 There are seven subgroups of order $4$ in $\Hom(C^3,C)\simeq C^3$: 
 \begin{align*}
 A_1&=\{0, e^*_1, e^*_2, (e_1+e_2)^*\}   &      
 A_2&=\{0, e^*_1, e^*_3, (e_1+e_3)^*\}   \\    
 A_3&=\{0, e^*_2, e^*_3, (e_2+e_3)^*\}    &     
 A_4&=\{0, e^*_3, (e_1+e_2)^*, (e_1+e_2+e_3)^*\} \\
 A_5&=\{0, e^*_2, (e_1+e_3)^*, (e_1+e_2+e_3)^*\} & 
 A_6&=\{0, e^*_1, (e_2+e_3)^*, (e_1+e_2+e_3)^*\} \\
 A_7&=\{0, (e_1+e_2)^*, (e_2+e_3)^*, (e_1+e_3)^* \} & &
\end{align*}  
 where $e_1^*, e_2^*$ and $e_3^*$ denote the canonical 
 generators. 
 For each of these $A_i$ we have a relation, saying that
 the sum of the three nonzero homomorphisms in that subgroup is zero. 
 For example, the relation attached to $A_1$ tells us that 
 $e_1^*+e_2^*+(e_1+e_2)^*=0$.
 Let $u$ be the sum of all these relations, and let $v_\alpha$ be the sum
 of the subset that involve a particular morphism $\alpha$. A calculation 
 shows that $(3v_\alpha-u)/6=\alpha$. 
 It follows that the resulting quotient
 $X(C^3)$ is zero.  If $d(G)\geq 3$ then any $\alpha\in\U(G,C)$
 can be factored through $C^3$, and it follows from this that
 $X(G)=0$.  Thus $X$ is a torsion object as claimed.
\end{example}

\section{Noetherian abelian categories}
\label{sec-noetherian}

The goal of this section is to study when the category $\A\U$ is
locally noetherian.

\begin{definition}\label{def-locally-noetherian}
 Let $\U$ be a subcategory of $\G$.
 \begin{itemize}
  \item An object $X\in\A\U$ is \emph{noetherian} if every subobject
   of $X$ is finitely generated. 
  \item The category $\A\U$ is \emph{locally noetherian} if $e_G$ is
   noetherian for all $G \in \U$.
 \end{itemize}
\end{definition}

\begin{remark}\label{rem-map2}
 Suppose that $\U$ is locally noetherian.  After adding the obvious
 consequences of the noetherian property to Remark~\ref{rem-map1} we
 get the following diagram of finiteness conditions for objects in
 $\A\U$: 
 \[
  \begin{tikzcd}[column sep=huge]
   \text{finitely resolved}  \arrow[r, Leftrightarrow] &
   \text{finitely presented} \arrow[r, Leftrightarrow] &
   \text{finitely generated} \arrow[r, Rightarrow] &
   \text{finite type} \\
   \text{perfect} 
    \arrow[u, Rightarrow] 
    \arrow[r, Leftrightarrow, "completeness"] &
   \text{finitely projective}
    \arrow[r, Rightarrow,"completeness"] & 
   \text{finite order.} &
  \end{tikzcd}
 \]
\end{remark}

It is not difficult to find subcategories of $\G$ for which
$\A\U$ is not locally noetherian.

\begin{proposition}\label{prop-not-noetherian}
 Let $\U$ be a full subcategory containing the trivial group and
 infinitely many cyclic groups of prime order.  Then $\A\U$ is not
 locally noetherian.
\end{proposition}

\begin{proof}
 Let $\chi^+ \in \A\U$ be the subobject of $\one$ given by
 \[
  \chi^+(T)=
   \begin{cases}
    0 \qquad \text{if} \; |T|=1 \\
    k \qquad \text{if} \; |T|>1.
   \end{cases}
 \]
 Note that $\min(\supp(\chi^+))$ contains the 
 isomorphism classes of all cyclic groups of prime orders.  Apply
 Lemma~\ref{lem-infinitess-criterion} to see that $\chi^+$ cannot be 
 finitely generated.
\end{proof}

The rest of this section will be devoted to prove the following theorem.

\begin{theorem}\label{thm-noetherian-subcategory-list}
 Fix a prime number $p$ and a positive integer $n$. The abelian category 
 $\A\U$ is locally noetherian for the following choices of subcategories 
 $\U$:
 \begin{itemize}
  \item[(a)] $\F[p^n]= \{ \mathrm{free}\; \ZZ/p^n\text{-}\mathrm{modules}\}$.
  \item[(b)] $\C[p^\infty]=\{\mathrm{cyclic}\;p\text{-}\mathrm{groups} \}$.
  \item[(c)] $\Z[p^n]=\{ \mathrm{fin. \;gen.\; }\ZZ/p^n
  \text{-}\mathrm{modules}\}$.
  \item[(d)] $\Z[p^\infty]=\{\mathrm{finite \;abelian}\;p
  \text{-}\mathrm{groups} \}$.
 \end{itemize}
\end{theorem}

\begin{proof}
  Sam and Snowden proved part (a)~\cite{Sam2017}*{8.3.1}. 
  The proofs of part (b),(c),(d) will be given in the next 
  subsections.
\end{proof}

\subsection{Part (b)}
We start by introducing the criterion for noetherianity developed in
\cite{Gan2015} which applies to a special type of subcategories.

\begin{definition}[\cite{Gan2015}*{2.2}]\label{def-A-infty}
 Let $\U$ be a subcategory of $\G$ and fix a skeleton
 $\U'$ for $\U$. If $G,H\in\U$ we write $G\gg H$ to mean that
 $\U(G,H)\neq\emptyset$.  We say that $\U$ has \emph{type} $A_\infty$
 if there exists an isomorphism of posets
 $(\U',\gg)\simeq(\NN,\geq)$.
\end{definition}

\begin{example}\label{ex-A-infty}
 The subcategory $\C$ of all cyclic groups is not of type $A_\infty$
 as there are no epimorphisms $C_3\to C_2$ or $C_2\to C_3$.  However
 if we fix a prime number $p$, then the subcategory $\C[p^\infty]$ of
 cyclic $p$-groups has type $A_\infty$.  Recall that $\F[p^n]$ is
 (equivalent to) the category of finitely generated free modules over
 $\ZZ/p^n$; this also has type $A_\infty$.  The same is true of the
 category $\E[p]$ of elementary abelian $p$-groups (because it is the
 same as $\F[p]$).
\end{example}

For compatibility with our work, we reformulate~\cite{Gan2015}*{3.1}
for contravariant diagrams.

\begin{definition}\label{def-trans-prop}
 We say that the category $\U$ has the \emph{transitivity property} if
 the action of $\Out(G)$ on $\U(G,H)$ is transitive whenever
 $G \gg H$.
\end{definition}

\begin{definition}\label{def-U-two-action}
 Suppose that $\U$ has the transitivity property.  For any
 pair $(G,H)$ with $G\gg H$ we let $\Out(G)$ act diagonally on
 $\U(G,H)^2$ and put $\U_2(G,H)=\U(G,H)^2/\Out(G)$.  
\end{definition}

\begin{lemma}\label{lem-Phi}
 Suppose we fix $\alpha\in\U(G,H)$ and put
 $\Phi(\alpha)=\{\phi\in\Out(G)\mid \alpha\phi=\alpha\}$.  Then there
 is a natural bijection
 $\zeta\colon \U(G,H)/\Phi(\alpha)\to\U_2(G,H)$.
\end{lemma}

\begin{proof}
 We have a map $\U(G,H)\to\U(G,H)^2$ given by
 $\gamma\mapsto(\alpha,\gamma)$, and this induces a map
 \[
 \zeta\colon \U(G,H)/\Phi(\alpha)\to\U_2(G,H).
 \]
 If $(\beta,\gamma)\in\U(G,H)^2$
 then the transitivity property gives $\theta\in\Out(G)$ with
 $\beta\theta=\alpha$ and it follows that
 $[\beta,\gamma]=[\beta\theta,\gamma\theta]=\zeta(\gamma\theta)$ in $\U_2(G,H)$.  
 This shows that $\zeta$ is surjective.  
 
 On the other hand, if $\zeta[\beta_0]=\zeta[\beta_1]$ then there
 exists $\phi\in\U(G)$ with
 $(\alpha\phi,\beta_0\phi)=(\alpha,\beta_1)$.  This means that
 $\alpha\phi=\alpha$ (so $\phi\in\Phi(\alpha)$) and
 $\beta_0\phi=\beta_1$ (so $[\beta_0]=[\beta_1]$ in
 $\U(G,H)/\Phi(\alpha)$).  This shows that $\zeta$ is also injective.
\end{proof}

\begin{lemma}
 Suppose that $G'\gg G$ and $u\in\U_2(G,H)$, so
 $u\subseteq\U(G,H)^2$.  Put
 \[ \lambda(u) = \lambda_{GG'}(u) = 
      \{(\alpha\phi,\beta\phi) \mid (\alpha,\beta)\in u,\;\phi\in\U(G',G)\} 
      \subseteq \U(G',H)^2.
 \]
 Then $\lambda(u)$ is a $\Out(G')$-orbit, or in other words an element
 of $\U_2(G',H)$.  The map $\lambda$ can also be characterised by
 $\lambda[\alpha,\beta]=[\alpha\phi,\beta\phi]$ for any
 $\phi\in\U(G',G)$.
\end{lemma}

\begin{proof}
 A typical element of $\lambda(u)$ has the form
 $x=(\alpha\phi,\beta\phi)$ with $(\alpha,\beta)\in u$ and
 $\phi\in\U(G,H)$.  If $\theta\in\Out(G)$ then the map
 $\phi'=\phi\theta$ also lies in $\U(G,H)$ and
 $\theta^*x=(\alpha\phi',\beta\phi')$; this shows that $\lambda(u)$ is
 preserved by $\Out(G)$.

 Now suppose we fix an element $x=(\alpha,\beta)\in u$ and a map
 $\phi\in\U(G,H)$ and put $x'=(\alpha\phi,\beta\phi)\in\lambda(u)$.
 Any element of $u$ has the form $(\alpha\zeta,\beta\zeta)$ for some
 $\zeta\in\Out(G)$.  Thus, any element $y\in\lambda(u)$ has the form
 $y=(\alpha\zeta\psi,\beta\zeta\psi)$ for some $\zeta\in\Out(G)$ and
 $\psi\in\U(G',G)$.  By the transitivity property we can find
 $\xi\in\Out(G')$ with $\zeta\psi=\phi\xi$, so
 $y=(\alpha\phi\xi,\beta\phi\xi)=\xi^*(x')$.  It follows that
 $\lambda[x]=[x']$, so in particular $\lambda[x]$ is an orbit as
 claimed.
\end{proof}

\begin{definition}\label{def-bij-prop}
 We say that $\U$ has the \emph{bijectivity property} if for all $H$
 there exists $G\gg H$ such that for all $G'\gg G$ the map
 \[ \lambda \colon \U_2(G,H) \to \U_2(G',H) \]
 is bijective.
\end{definition}

\begin{remark}\label{rem-bij-prop}
 Our bijectivity property is not visibly the same as that
 of~\cite{Gan2015}*{3.2}.  However, Lemma~\ref{lem-Phi} shows that
 they are equivalent (and we consider that our version is more
 transparent).
\end{remark}

We are finally ready to state the criterion.

\begin{theorem}[\cite{Gan2015}*{3.7}]
 Let $\U$ be a subcategory of $\G$ of type
 $A_\infty$. Suppose that $\U$ satisfies the transitivity and
 bijectivity properties. Then $\A\U$ is locally noetherian.
\end{theorem}

We now apply the criterion to our case of interest.

\begin{theorem}\label{thm-noetherian-cyclic-case}
 Fix a prime number $p$ and let $\C[p^\infty]$ be the family of cyclic
 $p$-groups.  Then the category $\A\C[p^\infty]$ is locally noetherian.
\end{theorem}

\begin{proof}
 We have already seen that $\C[p^\infty]$ has type $A_\infty$ so it is 
 enough to check that it satisfies the transitivity and bijectivity property.
 Recall the discussion on the morphisms of $\C[p^\infty]$ from 
 Example~\ref{ex-cyclic-colimit-exact}.  
 
 Consider cyclic groups $G$ and $H$ and suppose that $|H|$ divides $|G|$ so 
 that $\U(G,H)\not=\emptyset$. We know that 
 for any $\alpha\in\U(G,H)$ and $\phi\in\Aut(H)$ there exists 
 $\psi \in \Aut(G)$ such that $\alpha\psi=\phi\alpha$. 
 Combining this with the fact that $\U(G,H)$ is a torsor for $\Aut(H)$,  
 we find that $\U(G,H)$ is a single orbit for $\Aut(G)$. 
 Thus $\C[p^\infty]$ satisfies the transitivity condition.
 
 If $(\alpha,\beta)\in\U(G,H)^2$ then there is a unique element
 $\phi \in \Aut(H)$ with $\beta= \phi\circ\alpha$.  This is
 unchanged if we compose $\alpha$ and $\beta$ with any surjective
 homomorphism $\epsilon\colon G'\to G$.  It follows that the rule
 $[\alpha,\beta]\mapsto \phi$ gives a well-defined bijections
 $\xi=\xi_{GH}\colon \U_2(G,H)\to \Aut(H)$.  This also
 satisfies $\xi_{G'H}\lambda=\xi_{GH}$, so all the maps $\lambda$ are
 bijective, and so $\C[p^\infty]$ satisfies the bijectivity condition.
\end{proof}

\subsection{Part (c) and (d)}
The rest of this section will be devoted to proving the following result. 

\begin{theorem}\label{thm-p-groups-is-locally-noetherian}
 Fix a prime number $p$.  Recall that $\Z[p^\infty]$ is the category
 of finite abelian $p$-groups, and that $\Z[p^n]$ is the subcategory
 where the exponent divides $p^n$.  Then the categories
 $\A\Z[p^\infty]$ and $\A\Z[p^n]$ are locally noetherian.
\end{theorem} 
  
We will apply a different criterion due to Sam and Snowden that we
shall now recall~\cite{Sam2017}.  The basic outline is as follows.
One way to prove that polynomial rings are noetherian is to use the
technology of Gr\"obner bases.  If $\C$ is a category satisfying
appropriate combinatorial and order-theoretic conditions, we can use
similar techniques to prove that $[\C,\Vect_k]$ is locally noetherian.
If $\U\leq\G$ and we have a functor $\C\to\U^\op$ with appropriate
finiteness properties, we can then deduce that $\A\U$ is locally
noetherian.  In the case $\U=\Z[p^\infty]$ we will take $\C$ to be
something like the category of finite abelian $p$-groups with a
specified presentation, although the precise details are somewhat
complex.

\begin{remark}\label{rem-preordered-sets-and-categories}
 Some of the definitions and constructions below can be done for
 preordered sets or for small categories.  We regard a preordered set
 $P$ as a small category with one morphism $a\to b$ whenever
 $a\leq b$, and no morphisms $a\to b$ if $a\not\leq b$.  We regard a
 small category $\C$ as a preordered set by declaring that $a\leq b$
 if and only if $\C(a,b)\neq\emptyset$.
\end{remark}

The first combinatorial condition that we need to use is as follows:
\begin{definition}\label{def-wqo}
 Let $\C$ be a small category. 
 \begin{itemize} 
  \item A \emph{sequence} in $\C$ means a map
   $u \colon\NN\to\obj(\C)$
  \item A \emph{subsequence} of $u$ is a map of the form $u \circ f$,
   where $f \colon \NN \to \NN$ is strictly increasing.
  \item We say that $u$ is \emph{good} if there exists $i<j$ such
   that $u(i)\leq u(j)$ (meaning that $\C(u(i),u(j))\neq\emptyset$, as
   in Remark~\ref{rem-preordered-sets-and-categories}).
  \item We say that $u$ is \emph{very good} if $u(i)\leq u(j)$
   for all $i\leq j$. 
  \item We say that $\C$ is \emph{well-quasi-ordered} (or \emph{wqo})
   if every sequence in $\C$ is good.
  \item We say that $\C$ is \emph{cowqo} if $\C^{\op}$ is wqo.
  \item We say that $\C$ is \emph{slice-wqo} if the slice category
   $X \downarrow \C$ is wqo for all objects $X$.
 \end{itemize}
\end{definition}

\begin{remark}\label{rem-wqo-preorder}
 It is clear that the definition of wqo is compatible with the
 identifications in Remark~\ref{rem-preordered-sets-and-categories}.
\end{remark}

\begin{remark}\label{rem-finite-wqo}
 If $\C$ is finite then any sequence $u\colon\NN\to\obj(\C)$ is
 non-injective and therefore good.
\end{remark}

\begin{remark}\label{rem-wo-wqo}
 Now let $P$ be a well-ordered set.  
 For any sequence $u\colon\NN\to P$,
 the set $u(\NN)$ must have a smallest element, say $u(k)$, and then we
 have $u(k)\leq u(k+1)$, showing that $u$ is good.  It follows that
 $P$ is wqo.
\end{remark}

The following lemma is a basic ingredient.
\begin{lemma}\label{lem-wqo-equiv}
 Suppose that $\C$ is wqo.  Then any sequence in $\C$ has a very good
 subsequence. 
\end{lemma}
\begin{proof}
 Given any sequence $u\colon \NN\to\obj(\C)$ and $i\in\NN$, put
 \[ I(u,i) = \{j > i \mid u(i) \leq u(j) \}. \]
 Then put $J(u)=\{i\mid |I(u,i)|=\infty\}$.  Suppose that
 $J(u)$ is empty, so $I(u,i)$ is finite for all $i$.  Define
 $f\colon\NN\to\NN$ recursively by $f(0)=0$ and 
 \[ f(i+1) = \min\{j\mid j> f(i) \text{ and } j>k 
      \text{ for all } k \in I(u,f(i)) \}.
 \]
 It is then not hard to see that $u\circ f$ is bad, contradicting the
 assumption that $\C$ is wqo.  It follows that $J(u)$ must actually be
 nonempty.  Put $j(u)=\min(J(u))$, so $I(u,j(u))$ is infinite.  Put
 $T(u)=u\circ f$, where $f\colon \NN\to\NN$ is the unique strictly
 increasing map with image $I(u,j(u))$.  Now define
 $R(u)\colon \NN\to\obj(\C)$ recursively by $R(u)(0)=u(j(0))$ and
 $R(u)(i+1)=R(T(u))(i)$.  We find that $R(u)$ is a very good
 subsequence of $u$.
\end{proof}

\begin{definition}\label{def-hom-ordering}
 Let $\C$ be a small category.
 \begin{itemize}
  \item We say that $\C$ is \emph{rigid} if every endomorphism is an
   identity.
  \item A \emph{hom-ordering} on $\C$ consists of a system of
   well-orderings of the hom sets $\C(X,Y)$ such that for all
   $\alpha \colon Y \to Z$, the induced map
   $\alpha_* \colon \C(X,Y) \to \C(X,Z)$ is monotone.
 \end{itemize}
\end{definition}

\begin{definition}\label{def-Groebner}
 Let $\C$ be a small category and let $\D$ be essentially small.
 \begin{itemize}
  \item We say that $\C$ is \emph{Gr\"{o}bner} if it is rigid,
   slice-wqo and it admits a hom-ordering.
  \item We say that $\D$ is \emph{quasi-Gr\"{o}bner} if there is a
   Gr\"{o}bner category $\C$ and an essentially surjective functor
   $M\colon\C\to\D$ such that each comma category $(x\downarrow M)$
   has a finite weakly initial set.  In more detail, the condition is
   as follows: for each $x\in\D$ there must exist a finite list of
   objects $y_1,\dotsc,y_n\in\C$ and morphisms $f_i\colon x\to M(y_i)$,
   such that for any $y\in\C$ and any $f\colon x\to M(y)$ there exists $i$
   and $g\colon y_i\to y$ with $f=M(g)\circ f_i$.  This is known as
   \emph{Condition (F)}.
 \end{itemize}
\end{definition}

We are finally ready to state the criterion.
	
\begin{theorem}{\cite{Sam2017}*{4.3.2}}\label{thm-grobner-noetherian}
 Let $\D$ be a quasi-Gr\"{o}bner category. Then the category
 $[\D,\Vect_k]$ is locally noetherian.
\end{theorem} 

\begin{remark}\label{rem-terminology}
 Here and elsewhere we have used terminology and notation that seems
 clear to us and compatible with the rest of our work, but which
 differs from that in~\cite{Sam2017} and related references.  In
 particular, our ``rigid'' (as in Definition~\ref{def-hom-ordering})
 is their ``direct'', and our ``wqo'' is their ``noetherian''.  Our
 ``hom-ordering'' is their condition (G1), and our ``slice-wqo''
 condition is their (G2).
\end{remark}

Before proving Theorem~\ref{thm-p-groups-is-locally-noetherian} we
need to introduce more notation and prove some technical results.

\subsection*{Well-quasi orders}

\begin{remark}\label{rem-hereditarily-finite}
 To deal with some set-theoretic issues, we let $\X$ denote the set of
 hereditarily finite sets, so $\X$ is countable and closed under
 taking subsets, products and quotients, and contains sets of all
 finite orders. When we discuss categories of finite sets with extra
 structure, we will implicitly assume that the underlying sets are in
 $\X$, so that the category will be small.
\end{remark}

\begin{definition}\label{def-comonotone}
 Let $\C$ and $\D$ be preordered sets, and let $f \colon \C \to \D$ be
 a function.
 \begin{itemize}
  \item[(a)] We say that $f$ is \emph{monotone} if $p \leq p'$ implies
   $f(p)\leq f(p')$.
  \item[(b)] We say that $f$ is \emph{comonotone} if $f(p) \leq f(p')$
   implies $p \leq p'$.
 \end{itemize}
\end{definition}

\begin{remark}\label{rem-monotone-functor}
 Here $\C$ is and $\D$ might be small categories, regarded as
 preordered sets as in
 Remark~\ref{rem-preordered-sets-and-categories}. In that case, any
 functor $f \colon \C \to \D$ gives a monotone map.
\end{remark}	

\begin{proposition}\label{prop-comonotone-wqo}
 If $f\colon\C\to\D$ is comonotone and $\D$ is wqo then $\C$ is wqo.
\end{proposition}
\begin{proof}
 If $u \colon \NN \to \C$ is a sequence, then $f \circ u$ must be
 good, so there exists $i \leq j$ with $fu(i) \leq fu(j)$, but that
 implies $u(i) \leq u(j)$ by the comonotone property.
\end{proof}

\begin{proposition}\label{prop-product-of-wqo-is-wqo}
 Any finite product of wqo preordered sets is again wqo.
\end{proposition}
\begin{proof}
 It suffices to show that if $P$ and $Q$ are wqo, then so is
 $P \times Q$.  Let $u \colon \NN \to P \times Q$ be a sequence. As
 $P$ is wqo, we can find a subsequence $v$ such that $\pi_P \circ v$
 is nondecreasing. As $Q$ is wqo, we can then find a subsequence $w$
 of $v$ such that $\pi_Q \circ w$ is nondecreasing. Now $w$ is
 nondecreasing subsequence of $u$.
\end{proof}

We now recall the Nash-Williams theory of minimal bad
sequences~\cite{Nash}. 
\begin{definition}\label{def-min-bad}
 Let $P$ be a preordered set. We say that a finite list $u \in P^n$ is
 \emph{bad} if there is no pair $(i,j)$ with $0 \leq i <j<n$ and
 $u(i) \leq u(j)$. We say that such a finite list $u$ is \emph{very
  bad} if there is an infinite bad sequence extending it. If so, the
 set
 \[
   E(u)= \left\lbrace u' \in P \mid (u(0), \ldots, u(n-1), u') \;
    \text{is \; very \; bad} \right\rbrace
 \]
 is nonempty. Now suppose we have a well-ordered set $W$ and a
 function $\lambda \colon P \to W$.  Put
 \[
  EM(u) = \{u' \in E(u) \mid \lambda(u')=
            \min(\lambda(E(u))) \} \neq \emptyset.
 \]
 We say that a very bad list $u \in P^n$ is $\lambda$-\emph{minimal}
 if for all $k < n$ we have $u(k) \in EM(u_{<k})$. We say that a bad
 sequence $u$ is $\lambda$-\emph{minimal} if every initial segment
 $u_{<k}$ is $\lambda$-minimal.
\end{definition}

\begin{lemma}\label{lem-min-bad}
 If $P$ is not wqo, then it has a $\lambda$-minimal bad sequence.
\end{lemma}

\begin{proof}
 Start with the empty sequence, which is very bad by the assumption
 that $P$ is not wqo. Then choose recursively $u(k) \in EM(u_{< k})$
 for all $k \geq 0$.
\end{proof}

The following result abstracts the logic used for various wqo proofs
in the literature.
\begin{proposition}\label{prop-chi-wqo}
 Let $P$ and $\lambda$ be as above. Let $P_0$ be a subset of $P$, and
 let $\chi\colon P_0\to P$ be a map such that
 \begin{itemize}
  \item[(a)] For all $x\in P_0$ we have $\chi(x)\leq x$ and
   $\lambda(\chi(x))<\lambda(x)$.
  \item[(b)] Every bad sequence $u\colon \NN\to P$ has a subsequence
   $v$ contained in $P_0$ with the following property: if $i<j$ with
   $\chi(v(i))\leq\chi(v(j))$, then $v(i)\leq v(j)$.
 \end{itemize}
 Then $P$ is wqo.
\end{proposition}

\begin{proof}
 Suppose not, so there exists a minimal bad sequence $u$. Let $v$ be a
 subsequence as in~(b), so $v(n)=u(f(n))$ for some strictly increasing
 map $f\colon \NN\to\NN$. Define $w(n)=u(n)$ for $n<f(0)$ and
 $w(f(0)+k)=\chi(v(k))$. We claim that $w$ is bad. If not, we have
 $i<j$ with $w(i)\leq w(j)$. If $j<f(0)$ this gives $u(i)\leq u(j)$,
 contradicting the badness of $u$. Suppose instead that
 $i<f(0)\leq j$, so $w(i)=u(i)$ and $w(j)=\chi(v(j'))=\chi(u(j''))$
 for some $j'\geq 0$ and $j''\geq f(0)$. We now have
 $u(i)\leq\chi(u(j''))\leq u(j'')$, again contradicting the badness of
 $u$. This just leaves the possibility that $f(0)\leq i<j$, so
 $w(i)=\chi(v(i'))=\chi(u(i''))$ and $w(j)=\chi(v(j'))=\chi(u(j''))$
 for some $i',j',i'',j''$ with $i'<j'$ and $i''<j''$. We now have
 $\chi(v(i'))\leq\chi(v(j'))$ so $v(i')\leq v(j')$ b y condition~(b),
 so $u(i'')\leq u(j'')$, yet again contradicting the badness of
 $u$. It follows that $w$ must be bad after all. However, this
 contradicts the $\lambda$-minimality of $u(f(0))$ in $E(u_{<f(0)})$.
\end{proof}

\begin{definition}\label{def-words}
 Let $\C$ be a wqo category. We define $\S\C$ to be the category of
 pairs $(X,p)$, where $X$ is a finite, totally ordered set, and
 $p\colon X\to\C$. A morphism from $(X,p)$ to $(Y,q)$ consists of a
 strictly monotone map $\overline{\phi}\colon X\to Y$ together with a
 family of morphisms $\phi_x\colon p(x)\to q(\overline{\phi}(x))$ for
 each $x\in X$.  These are composed in the obvious way.  We put
 $\lambda(X,p)=|X|$.
\end{definition}

\begin{remark}\label{rem-words-preorder}
 If $\C$ is just a preordered set, then a morphism from $(X,p)$ to
 $(Y,q)$ is just a strictly monotone map
 $\overline{\phi}\colon X\to Y$ such that
 $p(x)\leq q(\overline{\phi}(x))$ for all $x$.
\end{remark}

The following result is standard (although typically formulated a
little differently).  We give the proof to illustrate the use of
Proposition~\ref{prop-chi-wqo}. 
\begin{proposition}[Higman's Lemma]\label{prop-words-wqo}
 $\S\C$ is wqo.
\end{proposition}
\begin{proof}
 For $(X,p)$ with $X\neq\emptyset$ we define $x_0=\min(X)$ and
 $\epsilon(X,p)=p(x_0)\in\C$ and $\chi(X,p)=(X',p')$, where
 $X'=X \sm\{x_0\}$ and $p'=p|_{X'}$.  This clearly satisfies
 condition~(a) of Proposition~\ref{prop-chi-wqo}. If
 $u\colon \NN\to\S\C$ is bad then $u(n)$ can never be empty (otherwise
 we would have $u(n)\leq u(n+1)$), so we have a sequence
 $u_1=\epsilon \circ u\colon\NN\to\C$. As $\C$ is wqo, we can choose a
 strictly increasing map $f\colon\NN\to\NN$ such that
 $u_1\circ f\colon \NN\to P$ is very good. Now put $v=u\circ f$.  If
 $i<j$ and $\chi(v(i))\leq\chi(v(j))$ then we also have
 $\epsilon(v(i))\leq\epsilon(v(j))$ and it follows easily that
 $v(i)\leq v(j)$. Using Proposition~\ref{prop-chi-wqo} we can now see
 that $\S\C$ is wqo.
\end{proof}

\begin{definition}\label{def-dag-monotone}
 Let $X$ and $Y$ be nonempty finite totally ordered sets. Let
 $\phi\colon X\to Y$ be a surjective map, which need not preserve the
 order.  We define an $\phi^\dagger \colon Y\to X$ by
 $\phi^\dagger(y)=\min(\phi^{-1}\{y\})$.  We say that $\phi$ is
 \emph{$\dag$-monotone} if $\phi^\dagger$ is monotone.
\end{definition}

\begin{lemma}\label{lem-dag-monotone}
 For any $\phi$ we have $\phi\phi^\dagger(y)=y$ for all $y\in Y$, and
 $\phi^\dagger\phi(x)\leq x$ for all $x\in X$. If $\phi$ is
 $\dagger$-monotone then we have $\phi(x)<y$ whenever
 $x<\phi^\dagger(y)$. In particular, if $x_0$ and $y_0$ are the
 smallest elements of $X$ and $Y$, then $\phi(x_0)=y_0$ and
 $\phi^\dagger(y_0)=x_0$.
\end{lemma}

\begin{proof}
 It is clear by definition that $\phi\phi^\dagger(y)=y$. Next, if
 $x\in X$ then $x$ is a preimage of $\phi(x)$, whereas
 $\phi^\dagger\phi(x)$ is the smallest preimage, so
 $\phi^\dagger\phi(x)\leq x$. Now suppose that $\phi$ is
 $\dagger$-monotone.  If $y\leq\phi(x)$ then
 $\phi^\dagger(y)\leq\phi^\dagger\phi(x)\leq x$. By the
 contrapositive, if $x<\phi^\dagger(y)$ we must have $\phi(x)<y$, as
 claimed. We now claim that $x_0=\phi^\dagger(y_0)$. Indeed, if not
 then $x_0<\phi^\dagger(y_0)$ so $\phi(x_0)<y_0$, contradicting the
 definition of $y_0$. We must therefore have $x_0=\phi^\dagger(y_0)$
 after all, and it follows that $\phi(x_0)=\phi\phi^\dagger(y_0)=y_0$.
\end{proof}

\begin{corollary}\label{cor-dag-comp}
 Suppose we have $\dagger$-monotone maps
 \[ X \xrightarrow{\phi} Y \xrightarrow{\psi} Z. \]
 Then $(\psi\phi)^\dagger=\phi^\dagger\psi^\dagger$, and so $\psi\phi$
 is also $\dagger$-monotone.
\end{corollary}

\begin{proof}
 Given $z\in Z$ put $y=\psi^\dagger(z)$ and
 $x=\phi^\dagger(y)=\phi^\dagger\psi^\dagger(z)$. Using the Lemma we
 get $\psi\phi(x)=z$. We also see that if $x'<x=\phi^\dagger(y)$ then
 $\phi(x')<y=\psi^\dagger(z)$ and thus $\psi(\phi(x'))<z$. This means
 that $x$ has the defining property of $(\psi\phi)^\dagger(z)$. We
 therefore have $(\psi\phi)^\dagger=\phi^\dagger\psi^\dagger$. This is
 the composite of two increasing maps, so it is again increasing, so
 $\psi\phi$ is $\dagger$-monotone.
\end{proof}

\begin{definition}\label{def-L-dag}
 We define a category $\L_\dagger$ as follows.  The objects are finite
 nonempty sets equipped with a map $e_X \colon X \to \NN$, together
 with a total order on $X$.  The morphisms from $X$ to $Y$ are
 $\dagger$-monotone surjective maps $\phi\colon X\to Y$ such that
 $e_Y(\phi(x))\leq e_X(x)$ for all $x\in X$.
\end{definition}

\begin{definition}\label{def-L-dag-map}
 We define $\alpha,\beta \colon \L_\dagger\to\NN$ by
 $\alpha(X)=e_X(\min(X))$ and $\beta(X)=\min(e_X(X))$. Next, for
 $x\in X\sm\{\min(X)\}$ we define
 \[ e'_X(x) = \min\{e_X(x')\mid x'<x\} \in\NN, \]
 and $e^*_X(x)=(e_X(x),e'_X(x))\in\NN^2$. The set $X\sm\{\min(X)\}$
 together with the map $e_X^*$ define an object
 $\gamma(X)\in\S(\NN^2)$.
\end{definition}

\begin{proposition}\label{prop-L-dag}
 The map
 $(\alpha,\beta,\gamma)\colon \L_\dagger^{\op} \to
 \NN^2\times\S(\NN^2)$ is comonotone, so $\L_\dagger$ is cowqo.
\end{proposition}

\begin{proof}
 Suppose that $\alpha(X)\leq\alpha(Y)$ and $\beta(X)\leq\beta(Y)$ and
 $\gamma(X)\leq\gamma(Y)$; we need to construct a morphism from $Y$ to
 $X$.  As $\beta(X)\leq\beta(Y)$, we can choose a strictly increasing
 map $\psi\colon X\sm\{\min(X)\}\to Y\sm\{\min(Y)\}$ with
 $e_X(x)\leq e_Y(\psi(x))$ and $e'_X(x)\leq e'_X(\psi(x))$ for all
 $x$. We extend $\psi$ over all of $X$ by putting
 $\psi(\min(X))=\min(Y)$, and note that the relation
 $e_X(x)\leq e_Y(\psi(x))$ remains true. We define
 $\phi\colon\psi(X)\to X$ by $\phi(\psi(x))=x$. Now consider an
 element $y\in Y\sm\psi(X)$, so $y\neq\min(Y)$. If $y>\max(\psi(X))$
 we choose $x$ with $e_X(x)=\beta(X)$ and define $\phi(y)=x$, noting
 that $e_Y(y)\geq\beta(Y)\geq\beta(X)=e_X(x)$. Otherwise, we let $x'$
 be least such that $\psi(x')>y$, then choose $x<x'$ with
 $e_X(x)=e'_X(x')$.  This gives
 \[ e_Y(y) \geq e'_Y(\psi(x')) \geq e'_X(x') = e_X(x), \] and we
 define $\phi(y)=x$. We now have a surjective map $\phi\colon Y\to X$
 with $e_Y(y)\geq e_X(\phi(y))$ for all $y$. We also have
 $\phi(\psi(x))=x$, and $\phi(y)<x$ whenever $y<\psi(x)$, so that
 $\psi=\phi^\dagger$. This means that $\phi$ is a morphism in
 $\L^\dagger$, as required.
\end{proof}

\begin{corollary}\label{cor-L-dag-cowqo}
 $\L_\dagger$ is slice-cowqo
\end{corollary}
\begin{proof}
 The construction $(X\xleftarrow{p}U)\mapsto (p^{-1}\{x\})_{x\in X}$
 gives a full and faithful embedding 
 $\L_\dagger \downarrow X\to\prod_{x\in X}\L_\dagger$. 
 Finally apply Proposition~\ref{prop-product-of-wqo-is-wqo}.
\end{proof}

\subsection*{Hom-orderings}
\label{subsec-hom-ordering}

\begin{remark}\label{rem-hom-ordering-op}
 In Definition~\ref{def-hom-ordering} we defined the notion of a
 hom-ordering on $\C$.  We can spell out the dual notion as follows:
 a hom-ordering of $\C^{op}$ consists of a system of well-orderings
 of the hom sets $\C(X,Y)$ such that for all $\beta \colon W \to X$,
 the induced map $\beta^* \colon \C(X,Y)\to \C(W,Y)$ is monotone.
\end{remark}

\begin{remark}\label{rem-hom-order-pullback}
 If $F \colon \C\to\D$ is a faithful functor and we have a
 hom-ordering on $\D$ then we can define a hom-ordering on $\C$ by
 declaring that $\phi\leq\psi$ if and only if $F\phi\leq F\psi$.
\end{remark}

\begin{definition}\label{def-hom-order-dag}
 Let $\F_\dagger$ be the category of finite totally ordered sets and
 $\dagger$-monotone surjections. We order $\F_\dagger(X,Y)$
 lexicographically, so $\phi<\psi$ if and only if there exists $x_0\in X$ with
 $\phi(x_0)<\psi(x_0)$ and $\phi(x)=\psi(x)$ for all $x<x_0$.
\end{definition}

\begin{proposition}\label{prop-hom-order-dag}
 This gives a hom-ordering on $\F^{\op}_\dagger$.
\end{proposition}

\begin{proof}
 It is standard and easy that the above rule gives a total order on
 the finite set of surjections from $X$ to $Y$.  Now suppose we have
 $\theta\colon W\to X$ and $\phi,\psi\colon X\to Y$ with
 $\phi\leq\psi$; we must show that $\phi\theta\leq\psi\theta$. By
 assumption there exists $x_0\in X$ with $\phi(x_0)<\psi(x_0)$ and
 $\phi(x)=\psi(x)$ for all $x<x_0$. Put
 $w_0=\theta^\dagger(x_0)=\min(\theta^{-1}\{x_0\})$. Then
 $(\phi\theta)(w_0)=\phi(x_0)<\psi(x_0)=(\psi\theta)(w_0)$. On the
 other hand, if $w<w_0$ then Lemma~\ref{lem-dag-monotone} tells us
 that $\theta(w)<x_0$ and so $(\phi\theta)(w)=(\psi\theta)(w)$.
\end{proof}

\begin{corollary}\label{cor-L-dag}
 The faithful forgetful functor
 $\L^{\op}_{\dagger}\to \F^{\op}_\dagger$ gives a hom-ordering to
 $\L^{\op}_\dagger$. \qed
\end{corollary}

\subsection*{Proof of Theorem~\ref{thm-p-groups-is-locally-noetherian}}
\label{subsec-noetherian-proof}

For the duration of this proof we put 
\[ \P=\Z[p^\infty]=\{\text{ finite abelian $p$-groups} \} \] 
and $C[k]=\ZZ/{p^k}\in\P$.  If $k\geq m$, we write $\pi$ for the
standard surjective homomorphism $C[k]\to C[m]$.  For $A\in\P$ and
$a \in A$, we let $\eta_a$ be the natural number such that $a$ has
order $p^{\eta_a}$

By combining Corollaries~\ref{cor-L-dag-cowqo} and~\ref{cor-L-dag}, we
see that $\L^{\op}_{\dagger}$ is Gr\"{o}bner.
	
We define an essentially surjective functor
$M\colon \L^{\op}_{\dagger}\to\P^{\op}$ as follows.  For an object
$X\in\L_{\dagger}$, we set $MX=\prod_{x\in X}C[e_X(x)]$. Given a
morphism $\phi\colon X\to Y$ in $\L_{\dagger}$, we define
$\phi_*\colon MX\to MY$ by $$(\phi_*m)_y=\prod_{\phi(x)=y}\pi(m_x).$$
 	
Let us introduce some terminology before proceeding with the proof.  A
\emph{framing} of $A\in\P$ is a surjective homomorphism $MX \to A$ for
some $X\in \L_{\dagger}$.  This corresponds to a map
$\alpha_0 \colon X \to A$ such that $\eta(\alpha_{0}(x)) \leq e_X(x)$
for all $x$, and $\alpha_0(X)$ generates $A$. We say that the framing
is \emph{tautological} if $X$ is a subset of $A$ and $\alpha_0$ is
just the inclusion and
\[
 e_{X}(x)= \max \{ \eta(w) \mid w \in X, \; w \leq x \}.
\]
It is clear from the definition that there are only finitely many
tautological framings.  Unravelling the definitions, we see that $M$
satisfies condition (F) if any framing $\alpha_0 \colon X \to A$
factors as
\[ X \to \overline{X} \to A \]
where the first arrow is in $\L_{\dagger}$ and the second one is a
tautological framing. So if $\alpha \colon X \to A$ is an arbitrary
framing, we define $\overline{X}= \alpha_0(X) \subset A$ and
$e_{\overline{X}}= \eta |_{\overline{X}}$ and set
$\overline{\alpha}_0 \colon \overline{X} \to A$ to be the
inclusion. We also define $\alpha_0^{\dagger} \colon A \to X$ by
$\alpha_0^{\dagger}(a)= \min (\alpha_0^{-1}(a))$ and order
$\overline{X}$ by declaring that $a < b$ iff
$\alpha_0^{\dagger}(a) < \alpha_0^{\dagger}(b)$.  This makes
$\overline{\alpha}_{0}$ into a tautological framing and gives the
required factorization. Therefore $\P^{\op}$ is quasi-Gr\"{o}bner and
so part (a) holds.
 	
For part (b), we put
\[ \Omega = \{ \eta_a \mid \ A \in \U, \; a \in A \} \subset \NN. \]
Define $\L_{\dagger}^{\U}$ to be the full subcategory of
$\L_{\dagger}$ consisting of objects $X$ with
$\img(e_X) \subset \Omega$. This is still Gr\"{o}bner
by~\cite{Sam2017}*{4.4.2}. It is now easy to check that the functor
$M \colon (\L_{\dagger}^{\U})^{\op} \to \U^{\op}$ defined as above is
essentially surjective and satisfies property $(F)$.  Thus $\U^{\op}$
is quasi-Gr\"{o}bner and $\A\U$ is locally noetherian.
	
\section{Representation stability}
\label{sec-central-stability}

In this section we show that any finitely presented object can be
recovered by a finite amount of data via a stabilization recipe. This
phenomenon is called central stability and it was first introduced by
Putman~\cite{Putman15}.  We also show that under the noetherian
assumption, any finitely generated object satisfies the analogue of
the injectivity and surjectivity conditions in the definition of
representation stability due to Church--Farb~\cite{Farb}*{1.1}.

\begin{definition}\label{def-truncation-functor}
 Let $\U$ be a subcategory of $\G$. For $X \in \A\U$, we
 put
 \[ \tau_n(X)=i_!^{\leq n} i^*_{\leq n}(X) \in \A\U, \]
 and note that there is a counit map $\tau_n(X) \to X$. We also
 define natural maps $\tau_n(X) \to \tau_{n+1}(X)$ as follows. Let
 $j$ denote the inclusion $\U_{\leq n}\to \U_{\leq (n+1)}$, so we have
 a counit map $j_!j^*(Y) \to Y$ for all $Y \in \A\U_{\leq
  (n+1)}$. Taking $Y=i^*_{\leq (n+1)}(X)$ for some $X \in \A\U$, we
 get a map $j_!i^*_{\leq (n+1)} X \to i^*_{\leq (n+1)} X$. Applying
 the functor $i_!^{\leq (n+1)}$ to this gives the required map
 $\tau_n(X) \to \tau_{n+1}(X)$.
\end{definition}

We list a few important properties of the truncation functor.

\begin{proposition}\label{prop-tau_n}
 Consider an object $X \in \A\U$.
 \begin{itemize}
  \item[(a)] Then $X$ is the colimit of the objects $\tau_n(X)$.
  \item[(b)] We have $\tau_n(e_G)=e_G$ if $G \in \U_{\leq n}$ and $\tau_n(e_G)=0$ 
   otherwise.
  \item[(c)] For all $G \in \U$ and $n \geq 0$, we have
   \[ \tau_n(X)(G)= \colim_{H \in N(G,n)} X(G/H) \]
   where $N(G,n)= \{H \triangleleft G \mid |G/H| \leq n\}$.
 \end{itemize}
\end{proposition}

\begin{proof}
 For part (a) it is enough to notice that $\tau_n(X)(G)= X(G)$ for
 $|G| \leq n$.  Part (b) follows from Lemma~\ref{lem-omnibus}(i).
 Using the formula for Kan extensions, we see that $\tau_n(X)(G)$ can
 be written as a colimit over the comma category
 $(G \downarrow \U_{\leq n})$.  Suppose we have objects
 $(G \xrightarrow{\alpha} A)$ and $(G \xrightarrow{\beta} B)$ in the
 comma category so $A,B \in \U_{\leq n}$.  As $\alpha$ and $\beta$ are
 surjective, we find that there is a unique morphism from $\alpha$ to
 $\beta$ if $\ker(\alpha) \leq \ker(\beta)$, and no morphisms
 otherwise. This shows that the comma category is equivalent to the
 poset $N(G,n)$ so part (c) follows.
\end{proof}

The following is a characterization of finitely generated and finite
presented objects.

\begin{proposition}\label{prop-condition-left-induced}
 Consider an object $X \in \A\U $.
 \begin{itemize}
  \item[(a)] $X$ is finitely generated if and only if $X$ has finite
   type and there exists $N \in \NN$ such that the canonical map
   $\tau_n(X) \to X$ is an epimorphism for all $n \geq N$.
  \item[(b)] $X$ is finitely presented if and only if $X$ has finite
   type and there exists $N \in \NN$ such that the canonical map
   $\tau_n(X) \to X$ is an isomorphism for all $n \geq N$.
 \end{itemize}
\end{proposition}

\begin{proof}
 For part (a), assume that the map $\tau_n(X) \to X$ is an epimorphism
 for all $n \geq N$. Note that we can construct an epimorphism
 \[
  \bigoplus_{G \in \U_{\leq n}} \dim(X(G)) \, e_G \to i^*_{\leq n}(X)
 \]
 as $X$ has finite type. We apply $i_!^{\leq n}$ to get an epimorphism
 \[
  \bigoplus_{G \in \U_{\leq n}} \dim(X(G)) \, e_G \to \tau_n(X)
 \]
 since $i_!^{\leq n}$ preserves all colimits by Lemma
 \ref{lem-omnibus}(f).  Post-composition with $\tau_n(X) \to X$ gives
 the desired epimorphism.  Conversely, assume that $X$ is finitely
 generated so that we have a short exact sequence
 $0 \to K \to P \to X \to 0$ with $P$ finitely projective.  Note that
 by Proposition~\ref{prop-tau_n}(b), there must exist $N \in \NN$ such
 that $\tau_n(P) \simeq P$ for all $n \geq N$. The commutativity of
 the diagram
 \[
  \begin{tikzcd}
   P \arrow[r] & X \arrow[r] & 0 \\
   \tau_n(P) \arrow[u, "\simeq"] \arrow[r] & \tau_n(X) \arrow[u]
  \end{tikzcd}
 \]
 implies that the map $\tau_n(X) \to X$ is an epimorphism for all
 $n \geq N$.

 For part (b), assume that $X$ is finitely presented. Then there
 exists a short exact sequence $0 \to K \to P \to X \to 0$ with $P$
 finitely projective and $K$ finitely generated. By Part (a), it is
 enough to show that the canonical map $\tau_n(X) \to X$ is eventually
 monic. Note that for large $n$, we have a diagram
 \[
  \begin{tikzcd}
   & 
   \ker(i_K^n) \arrow[d] &
   0 \arrow[d] &
   \ker(i_X^n) \arrow[d] & \\
   & 
   \tau_n(K) \arrow[r] \arrow[d, "i_K^n"] &
   \tau_n(P) \arrow[d,"\simeq"] \arrow[r] &
   \tau_n(X) \arrow[r] \arrow[d, "i_X^n"] & 0 \\
   0 \arrow[r] &
   K \arrow[r] \arrow[d] &
   P \arrow[d] \arrow[r] &
   X \arrow[r] \arrow[d] &
   0 \\
   &
   \cok(i_K^n) &
   0 &
   \cok(i_X^n) &
  \end{tikzcd}
 \]
 where the bottom row is exact and the top is only right exact.  By
 assumption both $K$ and $X$ are finitely generated, so the maps
 $i_K^n$ and $i^n_X$ are epimorphisms by part (a). Thus, the Snake
 Lemma tell us that $\ker(i^n_X)=0$. Conversely, assume that the
 natural map is an isomorphism.  By part (a), $X$ is finitely
 generated so we have a short exact sequence
 $0 \to K \to P \to X \to 0$ with $P$ finitely projective.  By
 applying the Snake Lemma to the diagram above, we see that
 $\cok(i_K^n)=0$ for large $n$, so $K$ is finitely generated and $X$
 is finitely presented.
\end{proof}

We note that by combining Propositions~\ref{prop-tau_n}
and~\ref{prop-condition-left-induced} we obtain that any finitely 
presented object satisfies central stability as mentioned in the introduction.

\begin{remark}\label{rem-q-leq-n}
Recall the functor $q_{\leq n}$ from Example~\ref{ex-q-leq-n}. 
We have seen that $q_{\leq n}$ is left adjoint to the inclusion 
$\U_{\leq n}^{\star} \to \G$. If $\U$ is closed downwards, then $q_{\leq n}$ is 
also the left adjoint to the inclusion $\U_{\leq n} \to \U$.
\end{remark}

\begin{proposition}\label{prop-q-star}
 Let $\U$ be multiplicative and closed under passage to subgroups, and
 consider a finitely presented object $X \in \A\U$. Then there exists
 $n \in \NN$ such that $X(G)=X(q_{\leq n} G)$ for all $G\in \U$.
\end{proposition}
\begin{proof}
 Choose a finite presentation
 \[
  \bigoplus_{i=1}^r e_{G_i} \xrightarrow{f}
  \bigoplus_{j=1}^s e_{H_j} \to X \to 0.
 \]
 Choose $n$ large enough so that $G_i,H_j \in \U_{\leq n}^*$ for all
 $i$ and $j$.  Let $Y$ be cokernel of $f$ in $\A\U_{\leq n}^*$. We
 claim that $X=q_{\leq n}^*(Y)$.  As the functor $q_n^*$ preserves all
 colimits it is enough to show that $q_{\leq n}^* e_G=e_{G}$ for all
 $G \in \U_{\leq n}^*$. Using that $q_{\leq n}$ is left adjoint to the
 inclusion $\U_{\leq n}^* \to \U$ we see that
 \[
  (q_{\leq n}^*e_G)(H)=k[\U(q_{\leq n} H,G)]=k[\U(H, G)]=e_G(H)
 \]
 which concludes the proof.  
\end{proof}

We now restrict to the locally noetherian case.  Recall the definition
of eventually torsion-free and generated in finite degree object from the
introduction, see Definition~\ref{def-intro-stability}.

\begin{theorem}\label{thm-stability}
 Let $X \in \A\Z[p^\infty]$ be a finitely generated object. 
 Then the restriction of $X$ to $\A\C[p^\infty]$ and 
 $\A\F[p^n]$, for all $n \geq 1$, is generated in finite degree and eventually 
 torsion-free.
\end{theorem}

\begin{proof}
 Firstly we note that the restriction of $X$ to $\A\C[p^\infty]$ and 
 $\A\F[p^n]$ is again finitely generated by Lemmas~\ref{lem-preservation-finiteness} 
 and~\ref{lem-restriction-preserve-fg}. Note also that $\C[p^\infty]$ and 
 $\F[p^n]$ satisfy the transitivity property, see Definition~\ref{def-trans-prop}.
 For the family of cyclic $p$-groups, we have proved this in the proof of 
 Theorem~\ref{thm-noetherian-cyclic-case}. For the families $\F[p^n]$ this is a special 
 case of Lemma~\ref{lem-dotted-arrow}. Since the abelian categories 
 $\A\C[p^\infty]$ and $\A\F[p^n]$ are locally noetherian by 
 Theorem~\ref{thm-noetherian-subcategory-list}, we can apply~\cite{Gan2015}*{5.1, 5.2} 
 and deduce that the restriction is generated in finite degree and eventually 
 torsion-free.
 \end{proof}

We conclude this section by proving Theorem~\ref{thm-homotopy-stability} from the introduction.

\begin{proof}[Proof of Theorem~\ref{thm-homotopy-stability}]
  First of all note that the equivalence~(\ref{equivalence global}) in the 
  introduction 
  descents to an equivalence between the full subcategories of compact 
  objects $(\Sp_\U^{\mathbb{Q}})^\omega \simeq \D(\A\U)^\omega$ for 
  any family $\U\leq \G$.  
  We can apply~\cite{HPS}*{2.3.12} to deduce that 
  \[
  \D(\A\Z[p^\infty])^\omega =
   \mathrm{thick}(e_G \mid G \in \Z[p^\infty])
  \]
  where the right hand side denotes the smallest thick 
  (=closed under retracts) triangulated subcategory containing 
  the generators $e_G$ for $G\in \Z[p^\infty]$. 
  
  Consider the full subcategory 
  \[
  \T=\{X \mid H_*(X) \;\; \text{is finitely generated} \} \subset 
  \D(\A\Z[p^\infty])^\omega.
  \]
  Since $\A\Z[p^\infty]$ is locally 
  noetherian one easily checks that $\T$ is a thick triangulated 
  subcategory. Clearly $e_G \in \T$ for all $G\in\Z[p^\infty]$ so by 
  the discussion in the previous paragraph we 
  see that any compact object lies in $\T$. Finally apply 
  Theorem~\ref{thm-homotopy-stability}.
\end{proof}

\section{Injectives}
\label{sec-injective}

We now turn to study the injective objects of $\A\U$.  Unlike in the
projective case, a complete classification of the indecomposable
injective objects seems at the moment far out of reach. The main
difficulty arises from the fact that any projective object is
necessarily torsion-free whereas an injective object can be torsion,
absolutely torsion or torsion-free.

Recall that if $\U$ has a colimit tower then the dual of any object is
injective by Proposition~\ref{prop-one-injective}. Let us produce more
examples of injective objects.

\begin{proposition}\label{prop-projective-is-injective}
 Let $\U$ be a multiplicative global family.  
 Then the torsion-free injective objects coincide with the projective objects.
\end{proposition}

\begin{proof}
 Suppose that $\U$ is a multiplicative global family and consider a projective 
 object $P$. We will show that $P$ is injective giving one of the 
 implications in the proposition.  
 We can write $P=\prod_n P_n$ by Proposition~\ref{prop-product-projectives}, so 
 it will suffice to show that $P_n$ is injective. We have
 $P_n=(i_n)_!(i^*_nP_n)$ and $i^*_n P_n$ is projective in 
 $\A\U_n$.  We
 can write $i^*_n P_n$ as a retract of an object 
 $Q=\bigoplus_t e_{G_t}$
 with $G_t\in\U_n$.  This embeds in the product $R=\prod_te_{G_t}$,
 and all monomorphisms in $\A\U_n$ are split, so $i^*_nP_n$ is a 
 retract of $R$.  We know that $(i_n)_!$ preserves products by
 Proposition~\ref{prop-groupoid-reps-a}, so $P_n=(i_n)_!(i^*_n P_n)$ 
 is a retract of $\prod_t(i_n)_!(e_{G_t})=\prod_te_{G_t}$.  
 Therefore, it
 is enough to show that $e_{G_t}$ is injective.  This now follows from
 the fact that $De_{G_t}$ is injective and that $e_{G_t}$ is a summand
 of $De_{G_t}$ by Theorem~\ref{thm-hom-fg-projective}. Therefore $P$ is 
 injective as claimed. 
 Conversely, let $I$ be a torsion-free injective. 
 By Proposition~\ref{prop-torsion-free-embed-proj}, we can embed $I$ into a
 projective object $SI$. Since $I$ is injective, the inclusion $I \to SI$ splits 
 showing that $I$ is projective as required. 
\end{proof}

\begin{remark}\label{rem-not-injective}
 Let $\C[2^\infty]$ be the family of cyclic $2$-groups. Then we have a short
 exact sequence 
 \[ 0 \to e_{C_2} \to \one \to t_{1, k} \to 0 \] that cannot
 split as $\one$ is torsion-free and $t_{1, k}$ is torsion.
 Hence $e_{C_2}$ is not injective in $\A\C[2^\infty]$.
\end{remark}

The following structural result, classically due to
Matlis~\cite{Matlis}, suggests that we can restrict our attention to
indecomposable injectives.
	
\begin{theorem}[\cite{Gabriel}*{Chaper IV}]\label{thm-inj-as-sum}
 Any injective object in a locally noetherian abelian category is a
 sum of indecomposable injectives.
\end{theorem}

\begin{lemma}\label{lem-torsion-injective}
 Let $\U$ be multiplicative global family of $\V$.
 \begin{itemize}
  \item[(a)] For any $G \in \V$ and $V$ irreducible
   $\Out(G)$-representation, the object $t_{G,V}$ is indecomposable and
   injective in $\A\V$. Furthermore, $t_{G,V}$ is the
   injective envelope of $s_{G,V}$.
  \item[(b)] For any $G \in \U$ and $V$ irreducible
   $\Out(G)$-representation, the object $\chi_\U \otimes e_{G,V}$ is
   indecomposable and injective in $\A\V$.
 \end{itemize}
\end{lemma}
\begin{proof}
 We have seen that $t_{G, V}$ is injective and it is indecomposable by
 Lemma~\ref{lem-omnibus}(e). If $\U$ is a multiplicative global
 family, then $e_{G,V}$ is injective and so combining part (e) and (i)
 of Lemma~\ref{lem-omnibus} we see that
 $i_*(e_{G,V})=\chi_\U \otimes e_{G,V}$ is an indecomposable
 injective.  Finally note that there is a canonical monomorphism
 $s_{G,V} \to t_{G,V}$, so the injective hull of $s_{G,V}$ is a direct
 summand of $t_{G,V}$ so the claim follows by indecomposability
\end{proof}

The next result classifies the indecomposable injective objects which are 
absolutely torsion. 

\begin{lemma}\label{lem-env}
 Let $\U$ be a subcategory of $\G$ and let $I \in \A\U$
 be injective. Then $I$ is a retract of a product of objects $t_{G,V}$
 with $G \in \U$.  If in addition $I$ is absolutely torsion, then it
 is a retract of a sum of objects $t_{G,V}$ with $G \in \U$.
\end{lemma}
\begin{proof}
 By Construction~\ref{con-proj-inj-resolutions}, we have a monomorphism
 \[
  \env \colon I \to \prod_{G \in \U'} t_{G, I(G)}=l_*l^*(I)
 \]
 By injectivity of $I$, the map $\env$ splits and so $I$ is a retract of
 $l_*l^*(I)$. If in addition $I$ is absolutely torsion, then the image of
 any element of $I$ under $\env$ is nonzero only for finitely many
 $G \in \U'$, so the morphism $\env$ factors through the direct sum.
\end{proof}

\bibliography{rationalglobalspectra}

\end{document}